%% file: bgl_arxiv.tex
\documentclass[11pt,a4paper]{article}
\usepackage[utf8]{inputenc}
\usepackage[T1]{fontenc}

\usepackage{amsmath}
\usepackage{amsfonts}
\usepackage{amssymb}
\usepackage{amsthm}
\usepackage{ upgreek }
\usepackage{ mathrsfs }

\usepackage{mathtools}
\usepackage[all]{xy}
\usepackage{float}
\usepackage{amssymb,amsfonts,amsthm}
\usepackage{enumerate}
\usepackage{epstopdf}
\usepackage{color,fancybox}
\usepackage{url}
\usepackage{graphicx,psfrag,epsfig}
\usepackage{bm}

\usepackage{changes}

\usepackage{mathtools,xparse}
\usepackage{dsfont}
\usepackage{bbold}
\usepackage{enumitem}
\usepackage{array}
\newcolumntype{M}[1]{>{\centering\arraybackslash}m{#1}}

\NewDocumentCommand{\ceil}{s O{} m}{%
 \IfBooleanTF{#1} 
 {\left\lceil#3\right\rceil} 
 {#2\lceil#3#2\rceil} 
}
\NewDocumentCommand{\floor}{s O{} m}{%
 \IfBooleanTF{#1} 
 {\left\lfloor#3\right\rfloor} 
 {#2\lfloor#3#2\rfloor} 
}

\usepackage{cleveref}
\usepackage{autonum} 
\usepackage{calc}

\usepackage{pgf,tikz}
\usetikzlibrary{patterns,hobby,automata,arrows,positioning,calc,patterns,fit,calc,decorations.pathreplacing,matrix,tikzmark,shadows.blur}
 \usepgflibrary{fpu}
\usepackage{pgfplots}
 \pgfplotsset{compat=1.6}

\usepackage{filecontents}
\usetikzlibrary{arrows,positioning,patterns}

\Crefformat{equation}{(#2#1#3)}

\newtheorem{theorem}{Theorem}[section]
\newtheorem{proposition}[theorem]{Proposition}
\newtheorem{lemma}[theorem]{Lemma}
\newtheorem{corollary}[theorem]{Corollary}
\newtheorem{definition}[theorem]{Definition}

\theoremstyle{definition}

\newtheorem{remarkx}[theorem]{Remark}
\newenvironment{remark}
 {\pushQED{\qed}\remarkx}
 {\popQED\endremarkx}
\newtheorem{assumption}{Assumption}

\numberwithin{equation}{section}
 
 \usepackage{float}

 \input{racc.sty}
\definecolor{pantone}{HTML}{00AEC7}
\colorlet{pantone2}{pantone!90!black}
\colorlet{pantone3}{pantone!80!black}

\definecolor{darkblue}{rgb}{0.03,0.23,0.36}
\colorlet{darkblue2}{darkblue!90!white}
\colorlet{darkblue3}{darkblue!80!white}

\definecolor{orangelike}{HTML}{FF8432}
\colorlet{orangelike2}{orangelike!80!black}
\colorlet{orangelike3}{orangelike!60!black}

\definecolor{aluminium}{rgb}{0.11,0.15,0.19}
\colorlet{aluminium2}{aluminium!90!white}
\colorlet{aluminium3}{aluminium!80!white}

\definecolor{gray}{rgb}{0.97,0.97,0.93}
\colorlet{gray2}{gray!90!black}
\colorlet{gray3}{gray!20!black}

\begin{document}

\begin{center}
{\sc \Large On the Hill relation and the mean reaction time
 \\
\vspace{0.2cm}
for metastable processes}
\vspace{0.5cm}

\end{center}

\noindent{\bf Manon Baudel}\\
{\it  CERMICS (ENPC), France}\\
\textsf{manon.baudel@enpc.fr}
\bigskip

\noindent{\bf Arnaud Guyader}\\
{\it LPSM (Sorbonne Universit\'e) \& CERMICS (ENPC), France }\\
\textsf{arnaud.guyader@upmc.fr}
\bigskip

\noindent{\bf Tony Leli\`evre\footnote{Corresponding author.}}\\
{\it  CERMICS (ENPC) \& INRIA, France}\\
\textsf{tony.lelievre@enpc.fr}
\bigskip

\medskip

\begin{abstract}
\noindent
We illustrate how the Hill relation and the notion of quasi-stationary
distribution can be used to analyse the biasing error introduced by many
numerical procedures that have been proposed in the literature, in particular
in molecular dynamics, to compute mean reaction times between metastable
states for Markov processes. The theoretical findings are
illustrated on various examples demonstrating the sharpness of the
biasing error analysis as well as the applicability of our study to elliptic diffusions.
\end{abstract}

 \bigskip

\medskip


\tableofcontents

\section{Introduction}
\label{Sec:Intro}

This work is motivated by the computation of reaction times in
thermostated molecular dynamics. In this context, the evolution of a
molecular system is typically modelled by the Langevin dynamics 
\begin{equation}\label{eq:Langevin}
\left\{
 \begin{aligned}
 d q_t &= M^{-1} p_t dt   ,\\
d p_t &= - \nabla V(q_t) dt - \gamma M^{-1} p_t dt + \sqrt{2 \gamma \beta^{-1}} dW_t  ,
 \end{aligned}
 \right.
\end{equation}
where $q_t\in
\R^d $ and $p_t\in \R^d$ denote the positions and momenta of the nuclei, 
or the overdamped Langevin dynamics, which is in position space only:
\begin{equation}\label{eq:OLangevin}
 d q_t = - \nabla V(q_t) dt + \sqrt{2 \beta^{-1}} dW_t  .
\end{equation}
In these equations, $\beta = (k_\mathrm{B} T)^{-1}$ is the inverse
temperature, $M$ is the mass matrix, $V: \R^d \to \R$ is the potential energy function, 
$\gamma>0$ is the damping parameter, 
and $W_t$ is a $d$-dimensional Brownian motion.

In practice, these dynamics are metastable, meaning that the stochastic process
spends most of its time in some connected sets of the phase space, called
metastable states. Metastable states can be for example energetic
traps (think of basins of attraction of local minima of $V$) or
entropic traps (think of large connected sets where $V$ is relatively flat,
with only narrow exit regions on the boundaries). Metastable states typically correspond to some
macroscopic states of the system, and studying the transitions between
them is thus of first importance to understand the molecular
mechanisms associated with these transitions. 

As an example, one
could think of a system consisting of a protein and a
ligand, in which case studying the transition from the bound state
(when the ligand is within a pocket of the protein) to the unbound
state (when the ligand is detached from the protein) is key for some
applications in drug design~\cite{ribeiro2018kinetics,copeland2006drug}. Indeed, computing the
mean reaction time from the bound state to the unbound state is crucial to rank the efficiencies of ligands for a given target
pocket. 

However, transitions between
metastable states are rare
events, which make the simulation of these transitions very
difficult. This is related to a timescale problem: for example, the typical
timestep for the discretisation of the Langevin dynamics is of the
order of $10^{-15} s$, while transitions between metastable states can
occur over timescales of order $10^{-6} s$ to $10^2 s$. This explains why
naive algorithms cannot be used to simulate such events. 

To formalize the problem, let us denote by
$X_t=(q_t,p_t)$ (resp.\ $X_t=q_t$) the Markov process of interest for
the Langevin (resp.\ overdamped Langevin) dynamics and $A$ and $B$ two
disjoint sets which define the
metastable states of interest. Notice that, in practice, these states
are typically defined in the position space only, so that in the
context of the Langevin dynamics, $A=A_q \times \R^d$ and $B=B_q
\times \R^d$ where $A_q$ and
$B_q$ are two disjoint subsets of the position space $\R^d$. As illustrated on
\Cref{FigEntrancesTimes}, let us also introduce the successive
reactive entrance times in $A$ and $B$ as $\tau^A_0= \inf \setsuch{ t > 0}{ X_t \in 
                 \bar A}$,  $\tau^B_0=\inf \setsuch{ t > \tau^{A}_{0}}{ X_t \in  \bar B }$ and, for all $n\geq0$,
$$ \tau^{A}_{n+1} = \inf \setsuch{ t > \tau^{B}_{n}}{ X_t \in 
                 \bar A}\hspace{1cm}\text{ and }\hspace{1cm}
\tau^{B}_{n+1} = \inf \setsuch{ t > \tau^{A}_{n+1}}{ X_t \in  \bar B }.$$ 
We use the qualifier ``reactive'' in reactive entrance time to
indicate that the times $(\tau^A_n)$ are entrance times in $A$ for
trajectories coming from $B$, and the times $(\tau^B_n)$ are entrance
times in $B$ for trajectories coming from $A$.

\begin{figure}
\begin{center}
\input{f1.pstex_t}
\caption{Reactive entrance times in the sets $A$ and $B$.}
\label{FigEntrancesTimes}
\end{center}
\end{figure}
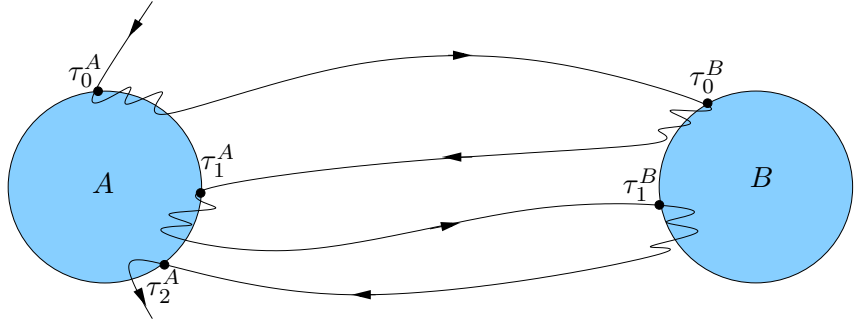

One is then interested in computing the mean reaction time from $A$ to
$B$ at equilibrium 
\begin{equation}
T_{AB}:=\lim_{N \rightarrow \infty} \frac{1}{N}
\sum_{n=1}^N (\tau^B_{n}- \tau^A_n),
\label{Eq:MeanTransitionTimeIntro}
\end{equation}
which is the equilibrium average duration of a path from $A$ to $B$. A path
from $A$ to $B$ is a trajectory that starts from the boundary of $A$
coming from $B$ and goes to $B$: such a trajectory performs many
visits to $A$ before going to $B$. The last part, from the boundary of $A$ to $B$
without going back to $A$, is called a reactive trajectory. More
generally, a reactive
trajectory is a trajectory which leaves $A$ and goes to $B$ without going back
to $A$, or leaves~$B$ and goes
to $A$ without going back to $B$.

Many techniques have been proposed to estimate this quantity, relying on more or less aggressive assumptions. We do not intend here to give an exhaustive list of numerical methods. Specifically, we are interested in techniques that are based on the introduction of a source in $A$ and a sink in $B$ in order to create a non-steady state flux from $A$ to $B$, and that approximate the reaction time by measuring the non-equilibrium flux. This idea dates back to~\cite{farkas1927keimbildungsgeschwindigkeit,Kramers_HA} and has been used in the weighted ensemble technique
\cite{zhang2010weighted,aristoff2018optimizing},
milestoning~\cite{Vanden-EijndenVenturoli,FaradjianRon,aristoff2016mathematical},
and Transition Interface Sampling~\cite{VanErpMoroniBolhuis}. We also
have in mind multilevel splitting techniques
\cite{kahn1951estimation,cgr} such as Forward Flux
Sampling~\cite{allen2009forward}, Non Equilibrium Umbrella
Sampling~\cite{dickson2010enhanced,thiede2016eigenvector} or Adaptive
Multilevel Splitting~\cite{cg,cerou2011multiple}. In particular, we
will show why the Hill relation is a cornerstone to analyse the
biasing error introduced by these methods.

The objective of this work is to give rigorous mathematical
foundations to such computations, by clarifying the bias introduced
when applying the Hill relation, depending on the distribution
used to re-inject the process in $A$ after hitting $B$. In particular,
we identify the ideal probability measure that should be considered
so that no bias is introduced: this is the reactive entrance
distribution $\nu_{\mathrm{E}}$ (where the subscript $\mathrm E$
refers to Entrance). In addition, we explain why, when $A$ is
metastable, it is possible to use a quasi-stationary distribution
$\nu_{\mathrm{Q}}$ to re-inject the process in $A$, while keeping a
small biasing error. This is crucial because this local equilibrium is
widely used in practice since it is easy to sample, contrary to the
reactive entrance distribution. The sharpness of our biasing error estimate is illustrated on various examples.

More precisely, the paper is organized as follows. In \Cref{sec:motiv}, we present in
more details the motivation of this work, namely the computation of
mean reaction times for diffusion processes. We show that this
question is equivalent to computing a quantity of the form
\begin{equation}
 \expecin{\nu_{\mathrm{E}}}{\sum_{n=0}^{\cT_{\cB} -1} f(Y_n) }
 \label{Eq:QOIprime}
\end{equation}
for some test function $f$ and a discrete-time Markov process $(Y_n)$ with values in
a domain $\cA \cup \cB$, where $\cA$ and $\cB$ are disjoint
sets. In~\eqref{Eq:QOIprime}, $\nu_{\mathrm{E}}$ is the so-called reactive
entrance distribution in $\cA$ and $\cT_{\cB}:= \inf \setsuch{ n  \geq 0 }{Y_n \in \cB}$ is the hitting time of $\cB$.
This setting is detailed in~\Cref{Sec:DTP}, where the
problem is stated in a rather general
framework. This section also gives the main assumptions as well as
the definitions of auxiliary Markov chains needed to perform the
analysis. Let us emphasize that the aim of those two
  preliminary sections (Sections~\ref{sec:motiv}
  and~\ref{Sec:DTP}) is to provide a clear mathematical formulation of the
  computational problem we are interested in. We recall 
  mathematical results on the objects we introduce (in particular the
  quasi-stationary distribution and the $\pi$-return process), which are for most of them already known, sometimes in different settings,
  and stated here mainly for the sake of self-completeness. The
  reader who is already familiar with these tools can skip the
  mathematical details. We provide a summary of the notation in \Cref{TabDifferentProcesses}, at the end of \Cref{Sec:DTP}.

Let us go back to the computation of~\eqref{Eq:QOIprime}. A naive
 approximation using a Monte Carlo method is
  clearly challenging when the sets $\cA$ and $\cB$ are metastable for two reasons: first,
$\cT_{\cB}$ is very large and, second, $ \nu_{\mathrm{E}}$ is difficult
to sample, and not known analytically in general. However, since $\cA$ is metastable,
the process $(Y_n)$ reaches a ‘‘local equilibrium
within $\cA$’’ (namely, a quasi-stationary distribution
$\nu_{\mathrm{Q}}$) before transitioning to $\cB$, and this can be
used to recast the problem in a form more appropriate for computations. Hence, the objective of this
work is twofold:
\begin{itemize}
\item First, we explain why, when replacing $\nu_{\mathrm{E}}$ by a
quasi-stationary distribution~$\nu_{\mathrm{Q}}$, the quantity~\eqref{Eq:QOIprime} can be
efficiently approximated using the Hill relation (see
Sections~\ref{Sec:Hill} and~\ref{sec:diff2}). Specifically, denoting
by $\pi_{0 | \mathcal A}$ the stationary distribution $\pi_0$ of
$(Y_n)$ conditioned to $\cA$, the Hill relation yields on the one hand
$$\expecin{\nu_{\mathrm{E}}}{\sum_{n=0}^{\cT_{\cB} -1} f(Y_n) } = \frac{\pi_{0 | \mathcal A} f}{\probin{\pi_{0 | \mathcal A}}{Y_1 \in \cB}},$$
which cannot be used directly in practice since, in general, it is
unclear how to sample $\pi_{0 | \mathcal A}$,  and on the other hand
\begin{equation}\label{eq:Hill_Intro}
\expecin{\nu_{\mathrm{Q}}}{\sum_{n=0}^{\cT_{\cB} -1} f(Y_n) } = \frac{
  \nu_{\mathrm{Q}}f }{\probin{\nu_{\mathrm{Q}}}{Y_1 \in \cB}}.
\end{equation}
The right-hand side of the latter equality can be
estimated thanks to a direct sampling of $\nu_{\mathrm{Q}}$ and the use of rare
event sampling methods such as those mentioned earlier (weighted
ensemble techniques, Transition Interface Sampling, Adaptive
Multilevel Splitting, etc.). The Hill relation can be seen as a tool to replace a
longtime computation (the left-hand side) by a rare event sampling
problem (the right-hand side).
\item Second, we
quantify the biasing error introduced when replacing
$\nu_{\mathrm{E}}$ by $\nu_{\mathrm{Q}}$ (see
Section~\ref{Sec:Bound}). This is the purpose of
Theorem~\ref{Prop:EqGene}, the main mathematical result of this work,
which shows that under a timescale separation assumption, the relative
biasing error
$$\abs{ \frac{ \expecin{\nu_{\mathrm{E}}}{\sum_{n=0}^{\cT_{\cB} -1} f(Y_n) } - \expecin{\nu_{\mathrm{Q}}}{ \sum_{n=0}^{\cT_{\cB} -1} f(Y_n) } } 
 { \expecin{\nu_{\mathrm{E}}}{ \sum_{n=0}^{\cT_{\cB} -1} f(Y_n) } }
}$$
is small.
\end{itemize}
The aim of \Cref{Sec:Hill} and \Cref{Sec:Bound} is thus to provide
all the details on these two points.
Then, returning to the setting of \Cref{sec:motiv}, \Cref{Sec:Diffusion}
discusses the applicability of the theory to the computation of
reaction times for diffusion processes, as well as practical
numerical procedures to estimate the quantities of interest in the right-hand
side of~\eqref{eq:Hill_Intro}. Finally, \Cref{SSec:DTTM}
demonstrates on simple toy models the sharpness of the biasing error estimate
given by Theorem~\ref{Prop:EqGene}, \Cref{SSec:ExMarkovNonRev} shows
that reversibility may not be conserved for the reactive entrance
process, while the proofs are all housed in \Cref{SSec:Proofs}. We
also provide, in \Cref{Sec:birkoff}, some details about the
Birkhoff's approach~\cite{Birkhoff1957} to prove existence, uniqueness and convergence to the
quasi-stationary distribution under the so-called two-sided condition.

\section{Motivation: how to compute mean reaction times?}\label{sec:motiv}

The purpose of this section is to motivate by the following question the discrete-time setting
that will be at the core of Section~\ref{Sec:DTP}: how to compute
reaction times from one metastable state to another for a diffusion
process? This section can be easily skipped if one is only interested
in the mathematical results we have obtained, forgetting about this motivation. 

\subsection*{The diffusion process and reaction time}
Let $(X_t) \in \R^d$ be the solution of the stochastic differential equation 
\begin{equation}
d X_t = f(X_t) dt +  g(X_t) d W_t ,\label{Eq:SDE}
\end{equation}
where
 $(W_t)$ denotes a $k$-dimensional standard Wiener process.
Here we assume that $f: \R^d \to \R^d$ and $g: \R^d \to \R^{d \times k}$ are smooth and satisfy conditions that guarantee the ergodicity of the
Markov process $(X_t)$ with respect to a unique invariant probability
distribution. For example, one could keep in mind the Langevin and
overdamped Langevin dynamics~\eqref{eq:Langevin} and~\eqref{eq:OLangevin}.
%

Let $A,B \subset  \R^d $ be two open
sets with smooth boundaries such that $\bar A$ and $\bar B$ are
disjoint, and with non-zero measure for the invariant probability
distribution of~\eqref{Eq:SDE}.
Since the process is ergodic, $(X_t)$ will visit $A$ and~$B$
infinitely often. We are interested in the paths from $A$ to $B$,
namely the pieces of the trajectory $t
\mapsto X_t$ that, coming from~$B$, pass from $A$ to $B$ (see \Cref{FigEntrancesTimes}). 

For this, let us recall the definition of the reactive entrance times:
$\tau^A_0= \inf \setsuch{ t > 0}{ X_t \in 
                 \bar A}$,  $\tau^B_0=\inf \setsuch{ t > \tau^{A}_{0}}{ X_t \in  \bar B }$ and, for all $n\geq0$,
$$ \tau^{A}_{n+1} = \inf \setsuch{ t > \tau^{B}_{n}}{ X_t \in 
                 \bar A}\hspace{1cm}\text{ and }\hspace{1cm}
\tau^{B}_{n+1} = \inf \setsuch{ t > \tau^{A}_{n+1}}{ X_t \in  \bar B },$$ 
and the associated empirical entrance distribution in  $A$:
\begin{equation}
\mu_{A,N}^+ = \frac{1}{N} \sum_{n=1}^{N}\delta_{X_{\tau^A_{n}}}.
\label{EqNu0Empi}
\end{equation}
The reactive entrance distribution $\nu_{\mathrm{E}}$ in $A$ is defined as
the weak limit of $\mu_{A,N}^+$ as $N \rightarrow
\infty$ (see~\cite[Proposition 1.5]{Lu_Nolen_2015} where $\nu_{\mathrm{E}}$ is denoted $\eta_A^+$), meaning that for any continuous and bounded $f:\partial A\to\R$, one has 
\begin{equation}
\int_{\partial A} f(x) \mu_{A,N}^+(dx)\xrightarrow[N\to\infty]{a.s.}\int_{\partial A} f(x) \nu_{\mathrm{E}}(dx).
\label{lkncmkzdmzd}
\end{equation}
Note that the reactive entrance distribution $\nu_{\mathrm{E}}$ is not
the restriction of the stationary measure of $(X_t)$ to the boundary
$ \partial A$ since $\nu_{\mathrm{E}}$ only takes into account the paths
that reach $A$ coming from $B$. In general, even if the original
process has an explicit invariant measure, $\nu_{\mathrm{E}}$ does not
admit a simple analytical expression, see~\cite{Lu_Nolen_2015} for more details. 


As already mentioned, the quantity of interest in the present article is
the expected reaction time $T_{AB}$ from $A$ to $B$, that is
\begin{equation}
T_{AB }:= \lim_{N \rightarrow \infty} \frac{1}{N}
\sum_{n=1}^N (\tau^B_{n}- \tau^A_n) = \expecin{\nu_{\mathrm{E}}}{\tau_{B}},
\label{Eq:MeanTransitionTime}
\end{equation}
where $\tau_B$ is the hitting time of $\bar B$.
The last equality is expected using the ergodicity of the
original process and the strong Markov property. It is for example
rigorously proved in~\cite{Lu_Nolen_2015}, see Equation (1.32) and
Proposition 1.8, in the case where $A$ and $B$ are bounded smooth, and
$gg^T$ is
  bounded from above and from below by strictly positive constants. The
  proof in~\cite{Lu_Nolen_2015} uses an ergodicity result (Doeblin's minorization condition)
  on the sequence of paths
$( (X_{(\tau^A_n +t)\wedge \tau^B_n})_{t \ge 0})$ seen as a Markov chain
indexed by $n$, see the proof of~\cite[Theorem 1.7]{Lu_Nolen_2015}. 

\subsection*{From a diffusion $(X_t)$ to a Markov
chain $(Y_n)$}

As explained in the introduction, to estimate the expected reaction
time $T_{AB}$, many numerical procedures rely on the
following construction. The
path from $A$ to $B$ is divided into two parts: the so-called loops, namely
the pieces of the trajectory between successive visits of $A$ 
without visiting $B$, and the so-called reactive trajectory, namely
the last part of the trajectory, which leaves~$A$ and goes to $B$
without going back to $A$. The idea is then to estimate the reaction
time by multiplying the average number of loops by their mean duration, and
then adding the expected duration of the reactive trajectory. However, because of the irregularity of the
paths, the successive entrances in $A$ do not define in general a
countable set: one thus needs to define the loops and the successive
entrances in $A$ in a different way. This is why  a separation surface $\Upsigma$ between $A$ and $B$ is introduced.

Let $\Upsigma$ be a smooth submanifold of codimension 1, such that the
probability starting from $ A$ to reach $B$
without crossing $\Upsigma$ is zero and such that
$\bar\Upsigma \cap \bar A = \varnothing$ and $\bar\Upsigma \cap \bar B =
\varnothing$. For example, as illustrated in \Cref{FigHittDis}, such a submanifold
can be constructed as the boundary of a smooth compact set $S$
which contains $\bar A$ and does not intersect $\bar B$. Denote $\cA = \partial A$, $\cB
= \partial B$ and $ \cE= \cA \cup \cB$.
Considering a trajectory $(X_t)$ with initial condition $X_0\in\cE$, the hitting times are defined
inductively by $\tau_{Y,0} = 0$ and, for all $k\geq0$,
\begin{align}
\tau_{\Upsigma,k+1} = \inf \setsuch{t> \tau_{Y,k}}{X_t \in \Upsigma}
                     \hspace{1cm} \text{ and }\hspace{1cm} \tau_{Y,k+1} = \inf \setsuch {t> \tau_{\Upsigma,k+1}}{X_t \in A \cup B} .
\end{align}

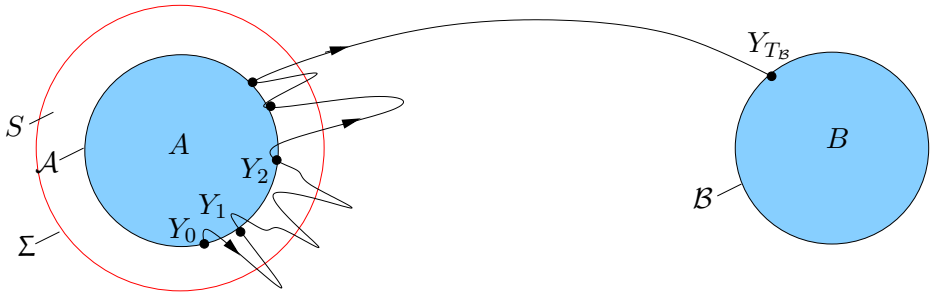
\begin{figure}
\begin{center}
\input{f2.pstex_t}
\caption{The Markov chain $(Y_n)$.}
\label{FigHittDis}
\end{center}
\end{figure}

Let us then introduce the sequence (see \Cref{FigHittDis})
\begin{equation}
\forall n  \geq 0, \, Y_n = X_{\tau_{Y,n}}
\label{Eq:MCDiffusion}
\end{equation}
 of successive intersections of $(X_t)$
 with $ A \cup B$, with intermediate hits to $\Upsigma$ before returning to $ A
 \cup B$. 
The strong Markov property implies that the law of $Y_{n+1}$ given
$Y_n$ is independent of $n$ and of the past trajectory $(Y_m)_{m<n}$. In other
words, $(Y_n)$ forms a time-homogeneous Markov chain
on $ \cE$. We thus obtain a discrete-time continuous-space Markov chain with transition kernel $K$ defined by
\begin{equation}
K(x,\cC) = \probin{x}{Y_1 \in \cC}
\end{equation}
for any $x\in  \cE$ and any Borel set $\cC \in \mathscr{B}( \cE)$.

\begin{remark}[Reversibility of $(Y_n)$]\label{rem:rev1}
In general, the Markov chain $(Y_n)$ defined by
\eqref{Eq:MCDiffusion} is not reversible even if this is the case for the original process
$(X_t)$, and its stationary distribution does not admit an analytic
expression. For a related discussion, we refer to Remarks~\ref{rem:rev2} and~\ref{rakscuhaozuc}, and Appendix~\ref{SSec:ExMarkovNonRev}.
\end{remark}

\subsection*{Expressing the mean reaction time $T_{AB}$ as a function of $(Y_n)$}
It turns out that the mean reaction time $T_{AB}= \expecin{\nu_{\mathrm{E}}}{\tau_{B}}$ in \eqref{Eq:MeanTransitionTime} can be expressed in terms of  the Markov chain $(Y_n)$.
To make this connection, let us define 
the function $\Delta:  \cE \mapsto \R^+ $ by
\begin{equation}
\label{eqFuncDelta}
\Delta(x) = 
\begin{cases}
\expecin{x}{\tau_{Y,1}}				& \text{for $x \in \cA$,}\\
0													& \text{for $x \in \cB$.}
\end{cases}
\end{equation}
Denoting $\cT_{\cB} = \inf \setsuch{ n  \geq 0 }{Y_n \in \cB}$, the reaction time can then be reformulated as
\begin{equation}
T_{AB} = \expecin{\nu_{\mathrm{E}}} {\sum_{n=0}^{\infty} \npar{\tau_{Y,n+1}- \tau_{Y,n}} \mathbf{1}_{n < \cT_{\cB}} } 
\end{equation}
or, thanks to the strong Markov property, 
\begin{equation}
T_{AB} = \expecin{\nu_{\mathrm{E}}} {{\sum_{n=0}^{\cT_{\cB}- 1}} \Delta(Y_n) }.
\label{EqDeltaTimeWithPoisson}
\end{equation}

The latter is the quantity of interest that will be considered in
the following. Notice that, even though the process $(Y_n)$ and the
function $\Delta$ depend on the submanifold
$\Upsigma$, the left-hand side $T_{AB}$ is
independent of $\Upsigma$. Accordingly, $\Upsigma$ can be
seen as a tuning parameter. In the three upcoming sections
we will discuss the computation of quantities of the
form~\eqref{EqDeltaTimeWithPoisson} for a Markov chain
$(Y_n)$. Eventually, we will return to the
application of our results to the diffusive case in
Section \ref{Sec:Diffusion}.

Notice that by considering other functions $\Delta$, one can thus have
access to different equilibrium properties over the ensemble of
 paths 
 entering $A$ and going to $B$  at equilibrium. These are the paths $( (X_{(\tau^A_n +t)\wedge
   \tau^B_n})_{t \ge 0})$ at equilibrium, sometimes also
 called transition paths~\cite{bolhuis-chandler-dellago-geissler-02}, not to be
 confused with the reactive paths introduced above, even though there
 is no consensus on this denomination in the literature:
 transition path is sometimes used as a synonym of reactive path,
 as for example in~\cite{Lu_Nolen_2015,E_VandeEijnden_2006}.

\begin{remark}[Generalisation to other Markov processes]
We wrote everything starting from a diffusion process, but all this
can be  generalised to other Markov processes: jump processes,
discrete-time Markov processes, etc. In particular, one should keep
in mind that in the context of molecular dynamics, only a
time-discretized version of the (overdamped) Langevin
dynamics is used in practice, which means that the
original process is actually a Markov chain in this context. 
\end{remark}

\begin{remark}[The case of the Langevin dynamics]\label{rem:Langevin}
For the Langevin dynamics~\eqref{eq:Langevin}, it is possible to define the successive entrance times
in $A_q \cup B_q$ as
$$\tau_{Y,k+1}= \inf \{t > \tau_{Y,k}, q_t \in A_q \text{ and }
p_t\cdot n_{A_{q}} < 0, \text{ or } q_t \in B_q \text{ and }
p_t\cdot n_{B_{q}} < 0\},$$
where $n_{A_{q}}$ and $n_{B_{q}}$ are the unit outward normals to $A_{q}$ and
$B_{q}$. These successive times are indeed such that $\lim_{k \to \infty}
\tau_{Y,k}=\infty$ almost surely, under some smoothness
assumption on $A_q$ and $B_q$. There is thus no need to introduce the
intermediate submanifold $\Upsigma$ in this context. In addition, it
is possible in this case to identify an analytical formula
for the equilibrium distribution $\pi_0$ of $(Y_n)$ (see~\cite[Chapter
3]{lopes19:PhD} for related considerations and~\cite[Chapter 5]{ramil:PhD} for explicit formulas). This yields to alternative
numerical methods based on the exact Hill relation (see
Equation~\eqref{EqHillEntranceDist} in Section~\ref{sec:EqHillEntranceDist} below), without having to replace~$\nu_\mathrm{E}$ by
the quasi-stationary distribution~$\nu_\mathrm{Q}$. Nevertheless, in practice, this
does not seem to be used by practitioners. We will investigate in
future works the interest of such an approach compared to the usual
method which consists in introducing the intermediate submanifold
$\Upsigma$. Likewise, we do not prove in this work that the
relations~\eqref{Eq:MeanTransitionTime} as well as the assumptions on $(Y_n)$ stated below
hold for the Langevin dynamics (for the sake of simplicity, we indeed check that they hold only for
elliptic diffusions and compact domains $\mathcal A$ and $\mathcal B$, see Section~\ref{Sec:Diffusion}). Again, this will be
the subject of a future work. 
\end{remark}

\section{The discrete-time setting}\label{Sec:DTP}
In this section, we detail the Markov chain setting we have in mind. After the introduction of the main notation and assumptions in Section~\ref{Sec:Yn}, we present related Markov chains that play a central role in our context.
First, the reactive entrance process is studied in Section \ref{sec:reactiveproc}. Then, the process
killed when leaving $\cA$, or killed process, is defined in Section~\ref{Ssec:Condi}. The latter allows us to introduce the notion of quasi-stationary distribution in
$\cA$, which is the formalization of the ``local equilibrium within~$\cA$'' reached by the process when it is trapped in the initial
metastable state. Finally, the $\pi$-return process presented in
Section~\ref{Ssec:RProcess} is the formalization of the
source-sink process discussed in the introduction, and is required to establish the Hill relation which is at the core of Section \ref{Sec:Hill}.

\subsection{Notation and assumptions}\label{Sec:Yn}
The purpose of this section is to introduce notation and assumptions on the Markov chain $(Y_n)$ that we will consider throughout this article.

\paragraph{The Markov chain $(Y_n)$.} Let $\cA$ and $\cB$ be two disjoint compact sets of a separable metric space, typically $\R^d$ or a discrete set such as $\Z^d$.
Consider a Markov chain $(Y_n)$ on
$ \cE:= \cA \cup \cB$ with transition kernel $K$, meaning that for any $x \in  \cE$ and any $\cC \in \mathscr{B}( \cE)$, the corresponding Borel-$\sigma$-algebra, we have
\begin{equation}
K(x,\cC)= \probin{x}{Y_1 \in \cC}.
\end{equation}
As usual, $K^n$ stands for the $n$-step transition kernel defined by
$K^0(x,\cC)=\delta_x(\cC) = \mathbf{1}_{x\in \cC} $ and, for all $n\geq 1$,
\begin{equation}
K^n(x,\cC) = \probin{x}{Y_n \in \cC} = \int_{ \cE} K^{n-1}(x,dz) K(z,\cC) .
\end{equation}
The Markov kernel $K$ induces two Markov semi-groups in the standard way: the
 probability measure $\pi K$ defined for any $\cC \in \mathscr{B}( \cE)$ by
\begin{equation}
\pi K(\cC) = \int_{  \cE} \pi(x) K(x,\cC) = \probin{\pi}{Y_1 \in \cC}
\end{equation}
is associated with any probability measure $\pi$ on $ \cE$, while for any test (i.e., bounded measurable) function $f:  \cE \rightarrow \R$, one can also consider the test function $Kf$, defined for any $x\in \cE$ by
\begin{equation}
Kf(x) = \int_{  \cE} K(x,dy) f(y) = \expecin{x}{f(Y_1)}.
\end{equation}
We recall some concepts for Markov chains on general spaces that will
prove useful in the following (see  for example \cite[Section 4.2]{HernandezLasserre} and \cite[Chapter 9]{MeynTweedieMCSto}). For any $\cC \in \mathscr{B}( \cE)$, the hitting and return times for the Markov chain $(Y_n)$ are respectively defined by
\begin{equation}
\cT_{\cC} = \inf \setsuch{n  \geq 0 }{Y_n \in \cC}\hspace{1cm} \text{ and } \hspace{1cm} \cT^+_{ \cC}= \inf\setsuch{n  \geq 1}{Y_n \in \cC} .
\label{EqHittingTimes}
\end{equation}
If $\pi$ denotes a probability measure on $( \cE,\mathscr{B}( \cE))$, a Markov chain $(Y_n)$ is called $\pi$-\emph{irreducible} if, for all $x \in  \cE$ and all $\cC \in \mathscr{B}( \cE)$ such that $\pi( \cC)>0$, one has
\begin{equation}
\expecin{x}{\sum_{n =1}^\infty \mathbf{1}_{\braces{Y_n \in \cC}}}=\sum_{n =1}^\infty\probin{x}{Y_n \in \cC}=\sum_{n =1}^\infty K^n(x,\cC)  >0,
\end{equation}
or, equivalently, if
\begin{equation}
\probin{x}{\cT^+_{\cC} < \infty} >0.\label{Eq:DefPhiIrreducibleMC}
\end{equation}
It is called \emph{recurrent} if there exists a probability measure $\pi$ such that 
\begin{equation}
\expecin{x}{ \sum_{n =1}^\infty \mathbf{1}_{\braces{Y_n \in \cC}}} = \infty, \qquad \forall x \in  \cE, \forall \cC \in \mathscr{B}( \cE) \text{ such that } \pi( \cC)>0.\label{kahcbikzck}
\end{equation}
 It is called \emph{Harris recurrent} if there exists a probability measure $\pi$ such that 
\begin{equation}
\biggprobin{x}{ \sum_{n =1}^\infty \mathbf{1}_{\braces{Y_n \in \cC}} = \infty} =1, \qquad \forall x \in  \cE, \forall \cC \in \mathscr{B}( \cE) \text{ such that } \pi( \cC)>0 . \label{eq:pos_Har_rec}
\end{equation}
Clearly, Harris recurrence implies recurrence, which itself implies $\pi$-irreducibility. If a (Harris) recurrent Markov chain admits an invariant
probability measure, then it is called positive (Harris)
recurrent.
We recall \cite[Proposition 4.2.11]{HernandezLasserre}:
\begin{proposition}\label{oaihzdozihd}
A (Harris) recurrent Markov chain admits (up to a multiplicative
constant) a unique invariant measure. Hence, if the invariant measure happens to be a probability measure, then the Markov chain is positive (Harris) recurrent.
\end{proposition}

\begin{remark}[Positive Harris recurrence and invariant probability]\label{rem:psi}
According to ~\cite[Section 4.2.2 and Theorem
10.4.9]{MeynTweedieMCSto}, if $(Y_n)$ is positive Harris recurrent with
invariant probability measure $\pi_0$, then~\eqref{eq:pos_Har_rec}
holds with $\pi=\pi_0$.
\end{remark}

The following assumptions will be of constant use throughout this work. As usual, $C_b( \cE,\R)$ stands for continuous and bounded functions from $ \cE$ to $\R$.

\begin{assumption}
\begin{enumerate}[label={[\Alph{assumption}\arabic*]},leftmargin=*]
\item\label{Ass:K0} $\cA$ and $\cB$ are compact disjoint sets and $ \cE=\cA\cup\cB$.
\item\label{Ass:K1} $K$ is weak-Feller, meaning that $Kf \in C_b( \cE,\R)$ whenever $f \in C_b( \cE,\R)$.
\item\label{Ass:K2} The kernel $K$ is positive Harris recurrent, and $\pi_0$ denotes its unique stationary probability measure.
\item\label{Ass:K3} $\pi_0(\cA)>0$ and $\pi_0(\cB)>0$.
\end{enumerate}
\label{Ass:K}
\end{assumption}

\begin{remark}[Assumptions in the case of a diffusion]\label{rem:diff1}
Returning to the setting of Section~\ref{sec:motiv}, one can exhibit
conditions on the diffusion $(X_t)$ solution to~\eqref{Eq:SDE} so that
Assumption~\ref{Ass:K} is satisfied, see Section~\ref{sdlkcakc}.
\end{remark}

\paragraph{Norms and operators.}
For any test function $f:  \cE \rightarrow \R$, the supremum norm is as usual $\norm{f}_{\infty}: = \sup_{x \in  \cE}{\abs{f(x)}}$.
The total variation norm of a finite signed measure $\mu$ on $( \cE,\mathscr{B}( \cE))$ is defined by $\|\mu\|:=\sup_{\|f\|_\infty\leq 1}\mu f$. Accordingly, given two  finite signed measures $\pi$ and $\nu$, the total variation distance between $\pi$ and $\nu$ is 
\begin{equation}
\norm{\pi-\nu }= \sup_{\|f\|_\infty\leq 1}(\pi f-\nu f).
\label{Eq:DefTV}
\end{equation}
Beware that a classic convention in probability is to define the  total variation distance between two probability measures as half of the latter quantity. In the sequel, the operator norm of a finite kernel $K$ acting on test functions is denoted by $\norm{\cdot}_{\infty}$, i.e.,
\begin{align}
\norm{K}_{\infty}
&= \supSur{\norm{f}_{\infty}  \leq 1} {\norm{Kf}_{\infty}}= \supSur{x \in  \cE}{ \supSur{\norm{f}_{\infty}  \leq 1} { \abs{Kf(x)}}}= \supSur{x \in  \cE}{ \norm{ K(x,\cdot) }}
  .
\label{Eq:DefNorm}
\end{align}
Finally, for any $\cC \in \mathscr{B}( \cE)$, $\id_{\cC} $ is the identity
operator on $\cC$, that is 
$$\pi \id_{\cC} f = \pi f, \qquad \forall \text{$\pi$ measure on }
\cC, \,
\forall f: \cC \to \R \text{ bounded measurable,}
$$
and $\mathds{1}_\cC : \cC \to \R$ denotes the function defined on
$\cC$ and identically equal to one on $\cC$.

\paragraph{Markov and sub-Markov kernels.}
For what follows, we
need to consider some restrictions of the Markov kernel $K$ to
subsets of $ \cE$.
Specifically, for any $\cC, \cD \in \mathscr{B}(  \cE)$, we introduce
the nonnegative sub-Markov kernel $K_{\cC \cD}$ defined for all $x \in \cC $ and all $D \in \mathscr{B}(\cD)$ by
\begin{equation}
K_{\cC \cD}(x,D) = \int_{D} K(x,dy)  .
\end{equation}
In other words, for any probability $\pi$ on $\cC$ and any test function $f$ on $\cD$, one simply  has $\pi K_{\cC \cD}f=\pi Kf$. If $\cD=\cC$, we just write $K_{\cC}$.
Using this notation, the transition kernel $K$ can be decomposed as a two-block kernel on $ \cE = \cA \cup \cB$ as
\begin{equation}
K=
\begin{bmatrix}
 K_{ \cA } & K_{ \cA \cB}\\ 
 K_{\cB \cA} & K_{\cB }
 \end{bmatrix} .
\label{EqTwoBlocks}
\end{equation}
Note that for all $x \in \cA$ and all $n  \geq 0$,
\begin{equation}
\probin{x}{\cT_{\cB} >n} = \probin{x}{Y_1 \in \cA ,  \dots, Y_n \in \cA}= K_{\cA}^n\mathds{1}_{\cA}(x) .\label{EqProbinGE}
\end{equation}
Besides, for all $x \in \cA$, it is readily seen that
\begin{equation}
K_{\cA} \mathds{1}_{\cA} (x) + K_{ \cA \cB} \mathds{1}_{\cB}(x) = \mathds{1}_{\cA}(x) ,
\end{equation}
which amounts to saying that
\begin{equation}
(\id_{\cA}-K_{\cA})\mathds{1}_{\cA} = K_{\cA \cB } \mathds{1}_{\cB} .
\label{Eq:PropKAOverOne}
\end{equation}
Accordingly, we also have, for any probability distribution $\pi$ on $\cA$,
\begin{equation}
\probin{\pi}{Y_1 \in \cB} = \pi K_{\cA \cB } \mathds{1}_{\cB} .
\label{alzj}
\end{equation}

\paragraph{Preliminary results on the Markov chain $(Y_n)$.} We give here some important consequences of our assumptions.

\begin{lemma} 
\label{Lemma:}
Under~\Cref{Ass:K}, 
\begin{equation}
\supSur{x\in \cA}{\probin{x}{Y_1 \in \cB}}>0 \qquad \mathrm{and} \qquad \supSur{x\in \cB}{\probin{x}{Y_1 \in \cA}}>0 .
\label{Eq:Accessibility}
\end{equation}
Moreover, there exists an integer $n$ such that
\begin{equation}
\inf_{x \in \cA} \probin{x}{\cT_{\cB}  \leq n} > 0 \qquad \mathrm{and }\qquad \inf_{x \in \cB} \probin{x}{\cT_{\cA}  \leq n} > 0 .
\label{Eq:HittingTimes}
\end{equation}
\end{lemma}

A simple but crucial consequence of this result is the well-posedness of Poisson equations associated with
  $K_\cA$ and $K_\cB$. By convention, an empty sum is equal to zero. Recall that the space $B(\cA,\R)$ of test (i.e., bounded and measurable) functions equipped with the supremum norm is a Banach space. 
    
\begin{corollary}
\label{Cor:Invert}
Under~\Cref{Ass:K}, 
the operator $(\id_{\cA}-K_{\cA})$ is
invertible in the following sense: for any test function
$g$ on $\cA$, the unique test function solution of the Poisson boundary value problem
\begin{equation}\label{EqNonDirichletBoundaryValueProblem}
\left\{
\begin{aligned}
(\id_{\cA} - K_{\cA}) r(x) &= g(x), & x& \in \cA
 \\
r (x) &= 0, & x &\in \cB  
\end{aligned}
\right.
\end{equation}
is given by
\begin{equation}
r(x)=\Biggexpecin{x}{ \sum_{n=0}^{\cT_{\cB}-1} g(Y_n) }.\label{EqSolutionNonDirichletBoundaryValueProblem}
\end{equation}
Mutatis mutandis, the same
result holds for the operator $(\id_{\cB}-K_{\cB})$.
\end{corollary}

Hence, for any test function $g$ and all $x\in\cA$, one has
\begin{equation}
(\id_{\cA}- K_{\cA})^{-1} g(x)=\Biggexpecin{x}{ \sum_{n=0}^{\cT_{\cB}-1} g(Y_n) },
\label{lakcnlazknclc}
\end{equation}
or more generally, for any probability measure $\pi$ on $\cA$,
\begin{equation}
\pi(\id_{\cA}- K_{\cA})^{-1} g=\Biggexpecin{\pi}{ \sum_{n=0}^{\cT_{\cB}-1} g(Y_n) }.
\label{lakcnlazknclckajcb}
\end{equation}
 Taking $g=\mathds{1}_{\cA}$ yields the next result.
\begin{corollary}\label{Cor:PoissonCor}
For all $x\in \cA$,
\begin{equation}
\expecin{x}{\cT_{\cB}}=(\id_{\cA}-K_{\cA})^{-1}\mathds{1}_{\cA}(x)  .
\label{Eq:MRT}
\end{equation}
And similarly, for all $x\in \cB$, $\expecin{x}{\cT_{\cA}}=(\id_{\cB}-K_{\cB})^{-1}\mathds{1}_{\cB}(x)$.
\end{corollary}

\paragraph{Restricted and conditional stationary measures.}
Given a
probability measure~$\pi$ on~$\cE$ and a Borel set $\cC \in
\mathscr{B}( \cE)$ such that $\pi(\cC)>0$, we denote by $\pi^{\cC}$ the measure restricted to $\cC$:
\begin{equation}
\pi^{\cC} (\cD) = \pi ( \cD \cap \cC), \quad \forall \cD \in \mathscr{B}( \cE),
\label{Eq:RestrictedMeasure}
\end{equation}
and by $\pi_{|\cC}$ the measure conditioned to $\cC$:
\begin{equation}
\pi_{|\cC} (\cD) = \frac{\pi^{\cC} (\cD)}{\pi(\cC)} = \frac{\pi ( \cD \cap \cC)}{\pi(\cC)}, \quad \forall \cD \in \mathscr{B}( \cE) .
\label{Eq:ConditionedMeasure}
\end{equation}
In the sequel, we will need various relations involving these restricted and conditional measures. Recall that, under~\Cref{Ass:K},  $\pi_0$ denotes the unique stationary probability measure of $(Y_n)$.

\begin{proposition}\label{prop:pi0}
Under~\Cref{Ass:K}, we have
\begin{equation}
\pi_0^{\cA} (\id_{\cA}- K_{\cA}) = \pi_0^{\cB} K_{\cB \cA}
\label{Eq:StatPi0}
\end{equation}
as well as
\begin{equation}
\pi_{0 | \mathcal A} = \pi_{0 | \mathcal A} K_{\cA} + \pi_{0 | \mathcal A} K_{ \cA \cB} (\id_{\cB}-K_{\cB})^{-1} K_{\cB \cA}.
\label{Eq:pi0trace}
\end{equation}
\end{proposition}

\subsection{The reactive entrance process}\label{sec:reactiveproc}
Under \Cref{Ass:K}, the Markov chain $(Y_n)$ is positive Harris recurrent, hence it will visit the
sets~$\cA$ and~$\cB$ infinitely often. 
Remind that we are interested in the
transitions of the Markov chain from $\cA$ to~$\cB$ at equilibrium. In this aim, we first introduce the reactive entrance process.

\paragraph{Definition of the reactive entrance process.}
From the Markov chain $(Y_n)$ living on $ \cE = \cA \cup \cB$, one can also define a process in $\cA$ monitored only when visiting $\cA$ after having reached $\cB$. We call it the reactive entrance process in $\cA$ and denote it by $(Y^\mathrm{E}_n)$. 
We recall that $\cT^+_{ \cB}= \inf\setsuch{n  \geq 1}{Y_n \in \cB}$ and define 
\begin{equation}
\cT^{\cA} = \inf\set{n > \cT^{+}_{\cB}, Y_n \in \cA} ,
\end{equation}
 which is the hitting time of $\cA$ after having reached $\cB$. 


\begin{definition}[Reactive entrance process]
\label{Def:REProcess}
The reactive entrance process associated with the process $(Y_n)$ is the Markov chain $(Y^\mathrm{E}_n)$ on $ \cA$ with transition kernel 
\begin{equation}
K^{\mathrm{E}}(x,\cC) = \probin{x}{Y_{\cT^{\cA}} \in \cC} \qquad \forall x \in \cA, \forall \cC \in \mathscr{B}(\cA),
\end{equation}
or, equivalently,
\begin{equation}
K^{\mathrm{E}} = (\id_{\cA}-K_{\cA})^{-1}K_{\cA\cB}(\id_{\cB}-K_{\cB})^{-1}K_{\cB \cA} .
\end{equation}
\end{definition}

The justification of the second formula is given in \Cref{SSec:Proofs}.


\paragraph{Some properties of the reactive entrance process.}

Recall that $\pi_0$ is the unique invariant probability measure of $(Y_n)$ and, $\forall \cD \in \mathscr{B}( \cE)$, 
$$\pi_{0 | \mathcal A}(\cD)= \frac{\pi_0 ( \cD \cap \cA)}{\pi_0(\cA)}=\frac{\pi_0^{\cA} (\cD)}{\pi_0(\cA)}.$$

\begin{proposition}\label{Prop:ReactiveEntranceHarris}
Under Assumption~\ref{Ass:K}, the reactive entrance process  $(Y^\mathrm{E}_n)$ is positive Harris recurrent with unique invariant probability measure 
\begin{equation}\label{EqLinkDist}
 \nu_{\mathrm{E}}:= \frac{ \pi_{0 | \mathcal A} (\id_{\cA}-K_{ \cA}) }{\probin{ \pi_{0 | \mathcal A}}{Y_1 \in \cB}} .
\end{equation}
Equivalently, we have
\begin{equation}
\pi_{0 | \mathcal A}= \frac{ \nu_{\mathrm{E}} (\id_{\cA}-K_{\cA})^{-1}}
{\expecin{\nu_{\mathrm{E}}}{\cT_{\cB}}}.
\label{kajcpcj}
\end{equation}
\end{proposition}

Starting from $Y_0=x\in\cA$, we define the reactive entrance times by setting $\cT^{\cA}_{0} =0$ and, for all $n  \geq 0$, 
\begin{equation}\label{Eq:EntranceTimes} 
\cT^{\cB}_{n+1} = \inf \setsuch{ m > \cT^{\cA}_{n} }{ Y_m \in \cB }
\qquad\text{ and }\qquad\cT^{\cA}_{n+1}= \inf \setsuch{ m> \cT^{\cB}_{n+1} }{ Y_m \in \cA } .
\end{equation}
Hence we have $Y^\mathrm{E}_0=x$ and, for all $n\geq 1$, $Y^\mathrm{E}_n=Y_{\cT^{\cA}_{n}}.$
The following sample-path ergodic property is a consequence of the
previous result and of  \cite[Theorem 4.2.13]{HernandezLasserre}. It
can be seen as the equivalent of \eqref{lkncmkzdmzd} for a fixed initial condition $x$ and a broader class of test functions.

\begin{corollary}
Under \Cref{Ass:K}, 
for every initial condition $Y_0=Y^\mathrm{E}_0=x\in\cA$ and every function $f\in L^1(\nu_{\mathrm{E}})$, one has
\begin{equation}
\lim_{N \rightarrow \infty}\frac{1}{N} \sum_{n=1}^{N}f(Y^\mathrm{E}_n)= \int_{\cA} f(x) \nu_{\mathrm{E}}(dx) \qquad \fP_x-\text{almost surely} .
\end{equation}
\label{Lemma:EmpiricalEntranceDistribution}
\end{corollary}

\begin{remark}[From the reactive entrance distributions to the original stationary distribution]
\label{Rem:EntranceStation}
Switching the roles of $\cA$ and $\cB$, the reactive entrance process in $\cB$ also admits a unique stationary distribution. Denoting $\nu_{\mathrm{E}}^{\cB}$ the stationary distribution of the reactive entrance process in~$\cB$, and $\nu_{\mathrm{E}}^{\cA}$ the stationary distribution of the reactive entrance process in $\cA$, \eqref{Eq:StatPi0}, \eqref{EqLinkDist}, and \eqref{kajcpcj} imply that
\begin{equation}
\nu_{\mathrm{E}}^{\cB} = \frac{ \nu_{\mathrm{E}}^{\cA}(\id_{\cA}-K_{\cA})^{-1}K_{\cA\cB}}{\nu_{\mathrm{E}}^{\cA}(\id_{\cA}-K_{\cA})^{-1}K_{\cA\cB} \mathds{1}_{\cB} } .
\end{equation}
Moreover, assuming that both reactive entrance processes are positive Harris recurrent with stationary distributions $\nu_{\mathrm{E}}^{\cA}$ and $\nu_{\mathrm{E}}^{\cB}$, then the original process $(Y_n)$ is also positive Harris recurrent and its unique stationary distribution satisfies for all $\cC \in \mathscr{B}( \cE)$
\begin{equation}\label{eq:pi0nuAnuB}
 \pi_0(\cC)= \frac{1}{\expecin{ \nu_{\mathrm{E}}^{\cA}}{\cT_{\cB}} +\expecin{\nu_{\mathrm{E}}^{\cB}}{\cT_{\cA}} } \bigpar{ \nu_{\mathrm{E}}^{\cA}(\id_{\cA}-K_{\cA})^{-1} \ind{\cA \cap \cC}+ \nu_{\mathrm{E}}^{\cB}(\id_{\cB}-K_{\cB})^{-1}\ind{\cB \cap \cC}} .
\end{equation}
The proof of this formula is detailed in \Cref{SSec:Proofs}.
\end{remark}

\begin{remark}[Reversibility of the reactive entrance process] \label{rem:rev2}
The fact that $K$ is reversible with respect to $\pi_0$ does not imply
that $K^\mathrm{E}$ is reversible with respect to $\nu_{\mathrm{E}}$. An
example of such a situation is given in Appendix~\ref{SSec:ExMarkovNonRev}. 
\end{remark}

\begin{remark}[Back to the case of a diffusion process] \label{rakscuhaozuc}
Returning to the setting of Section \ref{sec:motiv}, $(Y_n)$ can
be seen as the reactive entrance process in $\bar A \cup \bar B$ after visiting~$\Upsigma$, associated with the original diffusion process
$(X_t)$. In connection with Remarks~\ref{rem:rev1} and \ref{rem:rev2}, the fact that $(X_t)$ is reversible does not imply that $(Y_n)$ is reversible. In addition, notice that $(Y_n^{\mathrm{E}})$ can be seen as the reactive entrance process of $(X_t)$ in $\cA$ after visiting $\cB$, and it does not depend on $\Upsigma$. Accordingly, as already mentioned, $\nu_{\mathrm{E}}$ does not depend on $\Upsigma$ either.
\end{remark}

As already mentioned, since the Markov chain $(Y_n)$ is positive Harris recurrent, it will visit the
sets~$\cA$ and~$\cB$ infinitely often. We are interested in the
transitions of the Markov chain from $\cA$ to~$\cB$ at equilibrium. As justified by Corollary \ref{Lemma:EmpiricalEntranceDistribution}, when we refer to an average quantity over the paths from $A$ to $B$
at
equilibrium, we thus refer to the law of the paths from $A$ to $B$ starting
from the reactive entrance distribution $\nu_{\mathrm{E}}$ in~$\cA$. For a given
test function $f: \cA \rightarrow \R$, the aim of this work is to
estimate the following equilibrium quantity:
\begin{equation}
 \expecin{\nu_{\mathrm{E}}}{\sum_{n=0}^{\cT_{\cB} -1} f(Y_n) } ,
 \label{Eq:QOI}
\end{equation}
where $\cT_{\cB} = \inf \setsuch{n  \geq 0 }{Y_n \in \cB}$.
As explained in Section~\ref{sec:motiv}, for an appropriate choice of $f$ 
and $(Y_n)$, the
quantity \eqref{Eq:QOI} is the mean reaction time at equilibrium of
a diffusion process, see \eqref{eqFuncDelta} and \eqref{EqDeltaTimeWithPoisson}.

As explained in the introduction, when the sets $\cA$ and $\cB$ are metastable, simulating directly
\eqref{Eq:QOI} is out of reach because 
$\cT_{\cB}$ is very large and $ \nu_{\mathrm{E}}$ is difficult
to sample. Nonetheless, when $\cA$ is metastable,
the process $(Y_n)$ reaches a local equilibrium
within $\cA$ (namely, a quasi-stationary distribution) before transitioning to $\cB$. In view of this, we will explain why, when replacing $\nu_\mathrm{E}$ by a
quasi-stationary distribution, the quantity~\eqref{Eq:QOI} can be
approximated thanks to the Hill relation and rare event sampling methods
such as those mentioned in the introduction. In addition, we will
quantify the bias introduced when replacing $\nu_\mathrm{E}$ by a
quasi-stationary distribution (see Section~\ref{Sec:Bound}, Theorem~\ref{Prop:EqGene}). 

Before proceeding, we need to detail the notions of killed process and $\pi$-return process. We refer  the reader to Table~\ref{TabDifferentProcesses} at the end of 
Section~\ref{Ssec:RProcess} for a summary of the
various processes at stake in this paper.

\subsection{The killed process}
\label{Ssec:Condi}
From the process $(Y_n)$ living on $ \cE = \cA \cup
\cB$, one can define the process $(Y^{\mathrm{Q}}_n)$
killed when leaving $\cA$ (or, equivalently, killed when hitting
$\cB$). 

\begin{definition}[Killed process]
The killed process when leaving $\cA$ associated with the process $(Y_n)$ is the discrete-time process $(Y^{\mathrm{Q}}_n)$ on $ \cA$ with transition kernel 
\begin{align}
 K_{\cA}(x,\cC) &= \probin{x}{ Y^{\mathrm{Q}}_1 \in \cC} = \probin{x}{Y_1 \in \cC, \cT_{\cB}>1} \quad \forall x \in \cA, \forall \cC \in \mathscr{B}(\cA) .
\end{align}
\end{definition}
 \Cref{Ass:K} implies that this process does not admit a stationary distribution. Indeed, suppose that $\pi$ is such a distribution, then from~\eqref{Eq:HittingTimes}, there exists an $n$ such that
 \begin{equation}
 \pi K^n_{\cA} \mathds{1}_{\cA}  \leq \norm{K^n_{\cA}}_{\infty} = \sup_{x \in \cA}\probin{x}{\cT_{\cB} > n} <1,
 \end{equation}
 which contradicts the stationarity because $\pi \mathds{1}_{\cA}=1$.
Nevertheless, we may consider the notion of quasi-stationary distribution, which
extends the idea of stationary distribution to such sub-Markov
kernels.

\begin{definition}[Quasi-stationary distribution]
 A probability distribution $\pi$ on $ \cA$ is a
 quasi-stationary distribution (QSD) for the killed process
 $(Y^{\mathrm{Q}}_n)$ if for every measurable set $\cC \in \mathscr{B}(\cA)$,
\begin{equation}
 \probin{\pi}{Y_n \in \cC,\cT_{\cB} >n}=\probin{\pi}{\cT_{\cB} > n}\pi( \cC ), \quad \forall n \geq 0 .
\label{Eq:DefQSD}
\end{equation}
\end{definition}

It turns out that our framework ensures the existence of such a QSD. The following result is well known in various contexts. For example, when $(Y_t)$ is a continuous-time Markov process, it corresponds to Equation (2.18) in \cite{Collet_Martinez_SanMartin_book}.

\begin{lemma}
\label{Lemma:KilledEigen}
Under~\Cref{Ass:K}, the  killed process $(Y^{\mathrm{Q}}_n)$ admits a QSD. Specifically, a probability measure $\pi$ on
$\cA$ is a QSD for  $(Y^{\mathrm{Q}}_n)$ if and only if there exists $\theta  \geq 0$ such that 
\begin{equation}
\pi K_{\cA} = \theta \pi ,
\end{equation}
in which case $\theta = \probin{\pi}{\cT_{\cB}>1}$.
\end{lemma}


\begin{remark}[Degenerate case] Although theoretically possible, the case $\theta=0$ is of no interest in our context. It occurs when there exists $x\in\cA$ such that $\probin{x}{\cT_{\cB}>1}=0$, meaning that if $Y_0=x$ then $Y_1\in\cB$ almost surely. Therefore, from now on, we assume that
\begin{equation}
\forall x\in\cA,\hspace{1cm}\probin{x}{\cT_{\cB}>1}>0,
\label{lajcnlajnc}
\end{equation}
which in turn ensures that $\theta>0$ in Lemma \ref{Lemma:KilledEigen}.
\end{remark}

\begin{remark}[On the uniqueness of the QSD]\label{rem:unique_QSD}
As stated in Lemma~\ref{Lemma:KilledEigen}, \Cref{Ass:K} implies the existence of a quasi-stationary
 distribution, but as shown in~\cite[Section
 3.1]{BenaimCloezPanloup} there may exist several quasi-stationary
 distributions. Conditions to ensure the uniqueness can be found for
 example in~\cite{Champagnat_Villemonais_2017,del2001stability,del2004particle}, see
 also the recent review article~\cite{moral2021stability}.
A simple sufficient condition to get
 uniqueness is the following so-called two-sided estimate (see for example \cite[Section
 7.1]{Champagnat_Villemonais_2017},~\cite{BerglundLandon}
 and~Section~\ref{Sec:birkoff}): if for some $n \geq 1$, there exist a constant $C
 \geq1$, a probability measure $\pi$ on $\cA$, and a measurable function $s: \cA \rightarrow (0,1]$ such that, for all $x\in \cA$, 
\begin{equation}
 s(x) \pi(dy)  \leq K^n_{\cA}(x,dy)  \leq C s(x) \pi(dy) ,
 \label{Eq:UniformlyPos}
\end{equation}
then the killed Markov chain admits a unique QSD. Notice
that~\eqref{Eq:UniformlyPos} implies that the so-called Dobrushin
ergodic coefficient is smaller than one, which also yields
uniqueness of the quasi-stationary
distribution, see~\cite{dobrushin1956central,moral2021stability,del2001stability}
and~\cite[Section 12.2]{del2013mean}.
\end{remark}

In the sequel, a quasi-stationary distribution for the process $
(Y^{\mathrm{Q}}_n)$ will be denoted by~$\nu_{\mathrm{Q}}$. We remind
that, by \eqref{lajcnlajnc}, this implicitly implies that
$\probin{\nu_{\mathrm{Q}}}{ \cT_{\cB} =1 }$<1. Let us finally
  recall a classical result in the literature on quasi-stationary
  distributions (see e.g.  \cite[Theorem
  2.2]{Collet_Martinez_SanMartin_book} for an equivalent result in a
  continuous-time setting.)
\begin{lemma}
\label{Lemma:QSDandKilling}
Starting from a quasi-stationary distribution $\nu_{\mathrm{Q}}$, the
killing time $\cT_{\cB}$ is geometrically distributed with parameter
 \begin{equation}
 p=\probin{\nu_{\mathrm{Q}}}{ \cT_{\cB} =1 } \in(0,1).
 \label{Eq:Defp}
 \end{equation}
Notice that $p= \nu_{\mathrm{Q}}K_{\cA \cB} \mathds{1}_{\cB}$ and
 $1-p = \nu_{\mathrm{Q}}K_{\cA} \mathds{1}_{\cA}$.
 \end{lemma}

 According to the notation of Lemma \ref{Lemma:KilledEigen}, we thus have $\theta=1-p$.


\subsection{The $\pi$-return process}
\label{Ssec:RProcess}
The $\pi$-return process was introduced in~\cite{BartlettSto} in the
context of population genetics. In words, the $\pi$-return process in $\cA$ is a process that evolves exactly like $(Y_n)$ until it hits $\cB$, and is then instantaneously redistributed on $ \cA$ according to some probability measure $\pi$ supported on $\cA$. As a consequence, the $\pi$-return process can be seen as a process with a sink in $\cB$ and a source in $\cA$.


For any nonnegative test function $f:  \cE \to \R_+$ and any probability measure $\pi$, the outer product $f\otimes \pi$ is the nonnegative operator defined by $(f \otimes \pi)(x,\cC) =f(x)\pi(\cC)$. Here is the formal definition of the $\pi$-return process. 

\begin{definition}[$\pi$-return process]
Let $\pi$ and $(Y_n)$ denote respectively a probability measure on $ \cA$ and a Markov chain on $ \cA\cup\cB$. The $\pi$-return process associated to $(Y_n)$ is the Markov chain $(Y^\pi_n)$ on~$\cA$ with transition kernel 
\begin{equation}
K^\pi(x,\cC) = \probin{x}{Y_1 \in \cC, \cT_{\cB}>1} +\probin{x}{Y_1 \in \cB}\pi(\cC) \qquad \forall x \in \cA, \forall \cC \in \mathscr{B}(\cA),
\end{equation}
or, equivalently,
\begin{equation}
K^\pi = K_{ \cA} +( K_{\cA \cB}\mathds{1}_{\cB} ) \otimes \pi.
\end{equation}
\end{definition}


The next result shows that the $\pi$-return process admits a unique stationary distribution. This relation is known in several contexts. For example, it coincides with: \cite[Equation (2.1)]{MR3352332}  in a finite state space, \cite[Equation (2.4)]{MR1334159} and~\cite[Equation (2.3)]{MR2752899} in a countable state space, \cite[Proposition 4.5]{BenaimCloezPanloup} under slightly different assumptions, and \cite[Theorem 6]{MR4254493} in continuous time.

\begin{proposition}
\label{Prop:InvMeasureReturn}
Under~\Cref{Ass:K}, the $\pi$-return process admits a unique stationary distribution, that is
\begin{equation}
R(\pi) = \frac{\pi (\id_{\cA}-K_{\cA})^{-1}}{\expecin{\pi}{\cT_{\cB}}}  .
\label{Eq:RPi}
\end{equation}
\end{proposition}

Before going further, \Cref{TabDifferentProcesses} recaps the various
 Markov chains introduced so far. Note that $\nu_{\mathrm{Q}}$ is not an invariant distribution for the killed process but a quasi-stationary distribution.
\begin{table}[H]
\centering
\begin{tabular}{|M{3.6cm} |M{1.0cm} |M{7.1cm}| M{1.4cm}|}
\hline 
Markov chain & State Space & Transition kernel & Invariant Measure \\ 
\hline 
$(Y_n)$: Initial process & $\cA \cup \cB$  & $K$ &$\pi_0$ \\ 
\hline 
$(Y^{\mathrm{Q}}_n)$: Process killed when leaving $\cA$ & $ \cA$  &$K_{\cA} $ & $\nu_{\mathrm{Q}}$ (QSD) \\ 
\hline 
$(Y^{\mathrm{E}}_n)$: Reactive entrance process & $ \cA$  &$ K^{\mathrm{E}}=(\id_{\cA}-K_{\cA})^{-1}K_{\cA\cB}(\id_{\cB}-K_{\cB})^{-1}K_{\cB \cA}$ & $\nu_{\mathrm{E}}$ \\ 
\hline 
$(Y^{\pi}_n)$: $\pi$-Return process & $ \cA$ &$ K^{\pi} = K_{\cA}+ (K_{\cA \cB} \mathds{1}_{\cB}) \otimes \pi $ & $R(\pi)$ \\ 
\hline 
\end{tabular} 
\caption{Summary of the different Markov chains.}
 \label{TabDifferentProcesses} 
\end{table}

\section{The Hill relation }
\label{Sec:Hill}
Considering a source-sink process  at equilibrium, the Hill relation is an equality between the mean reaction time from 
the source to the sink and the inverse of 
the probability flux. The 
probability flux is the proportion of trajectories starting from one 
state and reaching the other state within one unit of time. This relation
was introduced in biochemistry, see~\cite[Section
8]{hill1977free}. Let us just mention that variations exist depending on how the source-sink
process is precisely defined.

\subsection{The general Hill relation}


In our context, we derive a Hill relation for
the $\pi$-return process of Section~\ref{Ssec:RProcess}, in
a more general form than the standard one which only considers
the mean reaction time.

\begin{proposition}
\label{PropHill}
Under~\Cref{Ass:K}, let $\pi$ be a probability distribution on $\cA$ and consider the $\pi$-return process together with its stationary distribution $R(\pi)$. Then, for any test function $f: \cA \rightarrow \R$,
\begin{equation}\label{eq:Hill}
\expecin{\pi}{\sum_{n=0}^{\cT_{\cB} -1} f(Y_n) } = \frac{R( \pi) f}{\probin{R(\pi) }{Y_1 \in \cB}}  .
\end{equation}
The classical Hill relation is obtained by setting $f = \mathds{1}_{\cA}$, that is 
\begin{equation}
\expecin{\pi}{ \cT_{\cB} } = \frac{1}{\probin{R(\pi)}{Y_1 \in \cB}}  .
\end{equation}
\end{proposition}


\begin{remark}[On other source-sink processes]
For related results, we refer to~\cite[Theorem~A.1]{AristoffAnalysis2018} where a Hill relation is derived for a
slightly different source-sink process than the $\pi$-return process
$(Y^\pi_n)$ considered here.
\end{remark}
 
To compute the quantity of interest, namely
\begin{equation}
\expecin{\nu_{\mathrm{E}}}{\sum_{n=0}^{\cT_{\cB} -1} f(Y_n) },
\label{jcnljnc}
\end{equation}
it is natural to
apply the Hill relation to $\pi=\nu_{\mathrm{E}}$, where
$\nu_{\mathrm{E}}$ is the reactive entrance distribution in $\cA$. However, this requires
to identify $R(\nu_{\mathrm{E}})$. This is the purpose of the next section.

\subsection{The Hill relation and the reactive entrance distribution}\label{sec:EqHillEntranceDist}

Thanks to the explicit formulas obtained  previously for $R(\pi)$ and
$\nu_{\mathrm{E}}$, the Hill relation applied to the reactive entrance
distribution yields a first useful expression to compute~\eqref{jcnljnc}. Putting \Cref{Prop:InvMeasureReturn} and \eqref{kajcpcj} together indeed provides the following result. Recall that $\pi_{0 | \mathcal A}$ was defined in \eqref{Eq:ConditionedMeasure}. 

\begin{corollary}
\label{Cor:ReactiveEntranceReturn}
The stationary distribution of the $\nu_{\mathrm{E}}$-return process is
the stationary distribution $\pi_0$ conditioned to $\cA$:
\begin{equation}
R(\nu_{\mathrm{E}}) =\pi_{0 | \mathcal A}  .\label{EqStationaryDistributionOfNu0ReturnProcess}
\end{equation}
In particular, the Hill relation yields
\begin{equation}
\expecin{\nu_{\mathrm{E}}}{\sum_{n=0}^{\cT_{\cB} -1} f(Y_n) } = \frac{\pi_{0 | \mathcal A} f}{\probin{\pi_{0 | \mathcal A}}{Y_1 \in \cB}}  .
\label{EqHillEntranceDist}
\end{equation}
\end{corollary}

\begin{remark}[Connection with potential theory]\label{rem:th}
In the specific case where $\mathcal E$ is discrete and $K$ is
reversible with respect to $\pi_0$,
Equation~\eqref{EqHillEntranceDist} is very similar to~\cite[Equation~(7.1.37)]{bovier-den-hollander-16},
using~\eqref{EqSolutionNonDirichletBoundaryValueProblem}. In this
context,  $\probin{\pi_{0 | \mathcal A}}{Y_1 \in \cB}$ is denoted
${\rm cap}(\cA,\cB)$, and called the capacity between $\cA$ and
$\cB$. In addition, $h_{\cA,\cB}$
in~\cite[Equation~(7.1.37)]{bovier-den-hollander-16} is simply
${\mathbf 1}_{\cA}$ since the state space is $\mathcal E=\cA \cup
\cB$. In potential theory, the measure $\nu_{\mathrm{E}}$ is called
the last-exit biased distribution. In this respect, Proposition~\ref{PropHill}
generalizes~\cite[Equation~(7.1.37)]{bovier-den-hollander-16} to any
$\pi$-return process, without assuming reversibility nor a discrete
space setting.
\end{remark}

From a numerical point of view, the interest of the Hill formula
is that the right-hand side
of~\eqref{EqHillEntranceDist} does not involve
$\cT_{\cB}$ (which is typically very large) anymore.
Accordingly, in the framework of Section~\ref{sec:motiv},
 there is no need to simulate a whole path,
but only reactive trajectories starting from $\pi_{0 | \mathcal A}$.

However, sampling according to $\pi_{0 | \mathcal A}$ can be a very difficult task since, 
in general, no analytical
formula for $\pi_{0 | \mathcal A}$ is
known (see however Remark~\ref{rem:Langevin} on the Langevin process). Therefore, it requires sampling from the stationary state of the initial process, which is computationally
demanding when $\cA$ and $\cB$ are metastable.

Nevertheless, considering the metastability of $\cA$,
it seems intuitively sensible that a good
approximation of $\expecin{\nu_{\mathrm{E}}}{\sum_{n=0}^{\cT_{\cB} -1} f(Y_n) }$ is $
\expecin{\nu_{\mathrm{Q}}}{\sum_{n=0}^{\cT_{\cB} -1} f(Y_n) }$ where $\nu_{\mathrm{Q}}$ is a
QSD of the original process in $\cA$, as
introduced in \Cref{Ssec:Condi} above. We will see in the next section
that this leads to a formula which is even easier to use numerically.
%

\subsection{The Hill relation and quasi-stationary distributions}

It is well-known that a QSD is a fixed point of the map $R$ introduced
in Proposition~\ref{Prop:InvMeasureReturn} (see for example \cite[Proposition
  2.1]{MR1334159} and \cite[Lemma 4.3]{BenaimCloezPanloup} for 
  similar results in slightly different contexts): $\nu_{\mathrm{Q}}$ is a QSD in $\cA$ if and
only if $\nu_{\mathrm{Q}}$ is the stationary law of the $\nu_{\mathrm{Q}}$-return
process. This has a nice consequence when applied to the Hill
relation, as stated in the next result.

\begin{proposition}
\label{Prop:UniquStationaryReturnQSD}
A probability measure $\pi$ is a quasi-stationary distribution for the process $(Y^{\mathrm{Q}}_n)$ killed when leaving $ \cA$ if and only if $\pi$ is a stationary distribution of the $\pi$-return process, i.e.,
\begin{equation}
\pi = R(\pi) .
\end{equation}
Therefore, under \Cref{Ass:K}, there exists a (not necessarily unique) probability measure $\pi$ such that $\pi= R(\pi)$. In particular, if $\nu_{\mathrm{Q}}$ is a QSD, then for any test function $f: \cA \mapsto \R$ the Hill relation writes
\begin{equation}
\expecin{\nu_{\mathrm{Q}}}{\sum_{n=0}^{\cT_{\cB} -1} f(Y_n) } = \frac{ \nu_{\mathrm{Q}}f }{\probin{\nu_{\mathrm{Q}}}{Y_1 \in \cB}}  .
\label{EqHillQSD}
\end{equation}
\end{proposition}

\begin{remark}[Connection with Wald's equation]
A general version of Wald's identity is the following: consider a
sequence of integrable random variables $(X_n)$ with same mean,
 a nonnegative integer-valued and integrable random variable $N$, and
assume that for all $n \ge 0$, $\E[X_n\mathbf{1}_{N\geq n}]=\E[X_n]\prob{N\geq n}$ and  $\sum_{n=1}^\infty\E[|X_n|\mathbf{1}_{N\geq n}]<\infty$, then 
$$\E\left[\sum_{n=1}^NX_n\right]=\E[N]\E[X_1].$$
Interestingly, \eqref{EqHillQSD} might be regarded as a consequence of
this identity, while  the general Hill relation \eqref{eq:Hill} can be considered as a similar relation in a different setting.
\end{remark}

If one admits, as will be justified in
Section~\ref{Sec:Bound}, that $\nu_{\mathrm{Q}}$ is a good approximation of $\nu_{\mathrm{E}}$
when $\mathcal A$ is metastable, the right-hand side of~\eqref{EqHillQSD} is
then a very efficient way to approximate the quantity of
interest~\eqref{jcnljnc}, that is
$$\expecin{\nu_{\mathrm{E}}}{\sum_{n=0}^{\cT_{\cB} -1} f(Y_n) }.$$
 Compared to~\eqref{EqHillEntranceDist}, the
interest is that $\nu_{\mathrm{Q}}$ is easier to sample than $\pi_{0 | \cA}$
since it does not require to observe transitions to $\cB$.
To illustrate the practical interest of~\eqref{EqHillQSD}, Section \ref{sec:diff2} provides some insights on how this can be used for estimating
reaction times for diffusions.

\section{On the biasing error introduced when replacing $\nu_\mathrm{E}$ by $\nu_\mathrm{Q}$}
\label{Sec:Bound}
As explained in the previous section, for numerical purposes, it is
natural to approximate the quantity of interest
$\expecin{\nu_{\mathrm{E}}}{\sum_{n=0}^{\cT_{\cB} -1} f(Y_n) }$ by
$\expecin{\nu_{\mathrm{Q}}}{\sum_{n=0}^{\cT_{\cB} -1} f(Y_n) }$, where
$\nu_{\mathrm{Q}}$ is a QSD as introduced in \Cref{Ssec:Condi}. The
objective of this section is thus to quantify the biasing error
introduced by this approximation. More precisely, we would like a
sharp estimate of the relative biasing error between these two
quantities. Using the Hill relations \eqref{EqHillEntranceDist} and
\eqref{EqHillQSD}, the relative biasing error satisfies, for any test function $f: \cA \rightarrow \R$,
\begin{equation}
\abs{ \frac{ \expecin{\nu_{\mathrm{E}}}{\sum_{n=0}^{\cT_{\cB} -1} f(Y_n) } - \expecin{\nu_{\mathrm{Q}}}{ \sum_{n=0}^{\cT_{\cB} -1} f(Y_n) } } 
 { \expecin{\nu_{\mathrm{E}}}{ \sum_{n=0}^{\cT_{\cB} -1} f(Y_n) } }
}
= \abs{1 - \frac{\probin{\pi_{0 | \mathcal A}}{Y_1 \in \cB} \nu_{\mathrm{Q}} f}{ \probin{\nu_{\mathrm{Q}}}{Y_1 \in \cB} \pi_{0 | \mathcal A} f } }  .
\label{EqExactBiaisTrace}
\end{equation}
Intuitively, one expects that the relative biasing error will be small if the
time needed to reach $\cB$ is much longer than the time needed to
relax to a local equilibrium within $\cA$, namely the time for the
process conditioned to stay in $\cA$ to reach the quasi-stationary
distribution. 

Therefore,  the main idea is to introduce two timescales: the
timescale to observe a transition from $\cA$ to $\cB$, and the timescale
to reach the QSD starting from $\nu_{\mathrm{E}}$, denoted by $T^{\mathrm{E}}_\mathrm{Q}$. This is the subject of
Sections~\ref{sec:timescales1} and~\ref{sec:timescales2}.
Once this is done, we show in Section~\ref{sec:errorbound} that the
biasing error~\eqref{EqExactBiaisTrace} is small
when the former is much larger than the latter. Finally, Section~\ref{sec:estimTQ} 
proposes two ways to estimate~$T^{\mathrm{E}}_\mathrm{Q}$.

\subsection{A lower bound for the reaction time}\label{sec:timescales1}
Let us introduce 
\begin{equation}\label{eq:p+}
p^+= \supSur{x \in \cA}{\probin{x}{Y_1 \in \cB}}.
\end{equation}
 Under \Cref{Ass:K}, we know from \Cref{Lemma:} that $p^+>0$.
 
\begin{lemma}
\label{kajclajc}
Under \Cref{Ass:K}, for all $x \in \cA$, one has
\begin{equation}
\frac{1}{p^+}  \leq \expecin{x}{\cT_{\cB}}. 
\end{equation}
\end{lemma}

In the following, we will use $1/{p^+}$ as a measure of the time for a
transition from $\cA$ to~$\cB$. The previous lemma only shows that $1/{p^+}$ is a lower bound of
the mean reaction time to~$\cB$. In
Section~\ref{sec:toymodel2}, we will check on a simple example that $1/p^+$ indeed yields to a sharp
estimate of the biasing error~\eqref{EqExactBiaisTrace}, as stated in
Theorem~\ref{Prop:EqGene}, in the sense that it cannot be replaced by
$1/\probin{\nu_{\mathrm{Q}}}{Y_1 \in \cB}$ or
$1/\probin{\pi_{0 | \mathcal A}}{Y_1 \in \cB}$.

\subsection{Relaxation time to a QSD}\label{sec:timescales2}

\paragraph{Definition of the relaxation time $T^{\mathrm{E}}_\mathrm{Q}$ to a QSD.}

Under~\Cref{Ass:K}, let $\nu_\mathrm{Q}$ be a QSD for the process killed when leaving $\cA$. We define the
relaxation time to $\nu_\mathrm{Q}$ through the
$\nu_\mathrm{Q}$-return process $(Y^{\nu_{\mathrm{Q}}}_n)$ starting from $\nu_{\mathrm{E}}$. Specifically, let us define the signed kernel~$H_{\mathrm{Q}}$ by 
\begin{equation}
H_{\mathrm{Q}}= (\id_{\cA}-K_{\cA})^{-1}(\id_{\cA}-\mathds{1}_{\cA}\otimes \nu_{\mathrm{Q}}) ,
\label{EqKernelHQOp}
\end{equation}
or, equivalently, for any test function $f: \cA \to \R$ and any $x \in \cA$:
\begin{equation}
H_{\mathrm{Q}}f(x) = \expecin{x}{ \sum_{n = 0}^{\cT_{\cB}-1} \braces{ f(Y_n) - \nu_{\mathrm{Q}}f }}.\label{EqHfReturn}
\end{equation}
The equivalence between \eqref{EqKernelHQOp} and \eqref{EqHfReturn} is
again a consequence of
Corollary~\ref{Cor:Invert}. The relaxation time to the QSD $\nu_{\mathrm{Q}}$ in $\cA$ starting from
$\nu_{\mathrm{E}}$ is then defined as
 \begin{equation}\label{Eq:TQE}
T^{\mathrm{E}}_{\mathrm{Q}}= \norm{\nu_{\mathrm{E}} H_{\mathrm{Q}}}.
 \end{equation}
Let us also define a uniform relaxation time to the QSD by
 \begin{equation}\label{Eq:TQ}
T_{\mathrm{Q}}= \norm{H_{\mathrm{Q}}}_\infty=\sup_{x \in \cA} \norm{H_{\mathrm{Q}}(x,\cdot)}<\infty.
 \end{equation}
  
\begin{proposition}\label{aicjapzcjpzc}
Under~\Cref{Ass:K},  we have
\begin{equation}\label{EqHfReturnprime}
H_{\mathrm{Q}}f(x) = \sum_{n=0}^\infty  \expecin{x}{ f(Y^{\nu_{\mathrm{Q}}}_n) - \nu_{\mathrm{Q}}f },
\end{equation} 
and
$$T^{\mathrm{E}}_{\mathrm{Q}}\le T_{\mathrm{Q}}\leq 2 \norm{(\id_{\cA}-K_{\cA})^{-1}}_{\infty}.$$
\end{proposition}

\paragraph{Why is  $T^{\mathrm{E}}_\mathrm{Q}$ a sensible measure of the
  relaxation time to a QSD?}

The next result shows that the time for the  process
conditioned to stay in $\cA$ to converge to the QSD is closely related
to the time for the $\nu_\mathrm{Q}$-return process to reach its equilibrium.

\begin{lemma}
\label{Lemma:NuQConditioned}
For any probability measure $\mu$ on $\mathcal A$, any test function $f$, and all $n\geq 0$, one has 
\begin{equation}\label{eq:NuQConditioned}
\expecin{\mu}{f(Y^{\nu_{\mathrm{Q}}}_n) - \nu_{\mathrm{Q}}f } = \bigpar{ \econdin{\mu}{ f(Y_n) }{ \cT_{\cB}>n} - \nu_{\mathrm{Q}}f } \probin{\mu}{ \cT_{\cB}>n} .
\end{equation}
As a consequence, denoting  $\cL^{\mu }( \cdot )$ the law of a
process with initial distribution $\mu$, we get
\begin{equation}
\norm{ \cL^{\mu }(Y^{\nu_{\mathrm{Q}}}_n) - \nu_{\mathrm{Q}}}  \leq \norm{ \cL^{\mu }(Y_n \vert \cT_{\cB}> n) - \nu_{\mathrm{Q}}}.
\end{equation}
Using these estimates with $\mu=\nu_{\mathrm{E}}$, one obtains in particular
 $$T^{\mathrm{E}}_\mathrm{Q} \le \sum_{n=0}^\infty \norm{ \cL^{\nu_{\mathrm{E}} }(Y_n \vert \cT_{\cB}> n) - \nu_{\mathrm{Q}}}.$$
\end{lemma}
Let us comment on the result of Lemma~\ref{Lemma:NuQConditioned}, see Equation~\eqref{eq:NuQConditioned}. It
shows that the convergence of the $\nu_{\mathrm{Q}}$-return process to $\nu_{\mathrm{Q}}$
occurs if the process $(Y_n)_{n \ge 0}$ conditioned to stay in $\mathcal A$ converges to
$\nu_{\mathrm{Q}}$, or if the process  $(Y_n)_{n \ge 0}$ reaches
$\mathcal B$ (all these processes starting from $\mu$). In
our context, we consider metastable situations where transitions to
$\mathcal B$ are rare: $ \probin{\mu}{ \cT_{\cB}>n} \ge (1-p^+)^n$ with
$p^+ \ll 1$. In this case, the convergence of the $\nu_{\mathrm{Q}}$-return process to $\nu_{\mathrm{Q}}$
is more related to the distance to
$\nu_{\mathrm{Q}}$ of the process $(Y_n)_{n \ge 0}$ conditioned to stay in
$\mathcal A$, than to the probability for the process  $(Y_n)_{n \ge 0}$ to reach $\mathcal B$.

\begin{remark}[Interpretation of $T_{\mathrm{Q}}$]
Let us recall that a randomized stopping time for $(Y^{\nu_{\mathrm{Q}}}_n)$ is a stopping time with respect to a possibly enlarged
version of the filtration generated by the random variables $(Y^{\nu_{\mathrm{Q}}}_n)$. Then, following~\cite{aldous1997mixing}, define for $ 0<c<1$
 \begin{equation}
 \label{Def:TStop}
 T_{\mathrm{stop}}(c) = \sup_{x}\inf_{T}\{\expecin{x}{T} \text{where $T$ is a randomized stopping time s.t.\ }\norm{\cL^x{(Y^{\nu_{\mathrm{Q}}}_T)} - \nu_{\mathrm{Q}}}  \leq c\} .
 \end{equation}
From \cite[Theorem~1]{aldous1997mixing}, it is known that if $T_{\mathrm{Q}} $ or $ T_{\mathrm{stop}}(c)$ (for some $0<c<1$) is finite then
$T_{\mathrm{Q}}$ and $ T_{\mathrm{stop}}(c)$ are equivalent in the sense that
\begin{equation}
T_{\mathrm{stop}}(c)  \leq\frac{4}{c^2}T_{\mathrm{Q}}< \frac{8}{c^2 (1-c)} T_{\mathrm{stop}}(c) .
\label{Eq:RelTStop}
\end{equation}
As a consequence, $T_{\mathrm{Q}}$ quantifies the time for the
$\nu_{\mathrm{Q}}$-return process to converge to $\nu_{\mathrm{Q}}$,
uniformly over the initial condition.
\end{remark}

\subsection{A general bound on the relative biasing error}\label{sec:errorbound}

In this section, we prove Theorem~\ref{Prop:EqGene} which gives an upper bound on
the relative biasing error introduced in
\eqref{EqExactBiaisTrace} by using the two time scales
introduced above: $1/p^+$ as a measure of the time to observe a
transition from $\cA$ to $\cB$, and $T^{\mathrm{E}}_\mathrm{Q}$ as a measure of the
time to relax to the QSD. The proof of Theorem~\ref{Prop:EqGene} relies on the following estimate of the difference between a quasi-stationary distribution and the conditional stationary distribution.

\begin{lemma}
\label{lzeknvclakev}
Let $\pi_{0 | \mathcal A}$ be the stationary distribution $\pi_0$
conditioned to $ \cA$, and $\nu_{\mathrm{Q}}$ be a QSD for the killed process $(Y^{\mathrm{Q}}_n)$. Then 
 one has
\begin{equation}
\norm{\pi_{0 | \mathcal A}-\nu_{\mathrm{Q}} }   = \probin{\pi_{0 | \mathcal A}}{Y_1 \in \cB} T^{\mathrm{E}}_{\mathrm{Q}} ,
\label{EqDiffMeasure}
\end{equation}
where $T^{\mathrm{E}}_{\mathrm{Q}}$ denotes the relaxation time defined by \eqref{Eq:TQE}.
\end{lemma}

We are now in a position to state the main mathematical result of this
work. 
\begin{theorem}
\label{Prop:EqGene}
Let $\nu_\mathrm{Q}$ be a QSD for the process killed when leaving $\cA$
and let us assume that $p^+T^{\mathrm{E}}_{\mathrm{Q}}<1$, where $p^+$ is defined
by~\eqref{eq:p+} and $T^{\mathrm{E}}_{\mathrm{Q}}$
by~\eqref{Eq:TQE}. Then the relative biasing error introduced in \eqref{EqExactBiaisTrace} is bounded as follows:
\begin{equation}
\abs{ \frac{ \expecin{\nu_{\mathrm{E}} }{ \sum_{n=0}^{\cT_{\cB} -1} f(Y_n) } - \expecin{\nu_{\mathrm{Q}}}{ \sum_{n=0}^{\cT_{\cB} -1} f(Y_n) } } {\expecin{\nu_{\mathrm{E}}}{\sum_{n=0}^{\cT_{\cB} -1} f(Y_n) }}}  \leq \frac{p^+T^{\mathrm{E}}_{\mathrm{Q}}}{1-p^+T^{\mathrm{E}}_{\mathrm{Q}}} \left( 1 + \frac{ \norm{f}_{\infty}}{|\pi_{0 | \mathcal A} f| }\right) .
\label{EqGenBound}
\end{equation}
\end{theorem}

The upper bound in~\eqref{EqGenBound} shows that the relative biasing error is small if $ T^{\mathrm{E}}_{\mathrm{Q}} \ll  1/p^+$,
namely if the timescale associated with
the relaxation time to the QSD in
$\cA$ is small compared to the
timescale associated with the transition
time from $\cA$ to $\cB$. Notice that the result holds for any QSD $\nu_\mathrm{Q}$, $T^{\mathrm{E}}_Q$ being
the associated convergence time for the $\nu_\mathrm{Q}$-return
process starting from $\nu_{\mathrm{E}}$. We refer to Appendix~\ref{SSec:DTTM}
for a discussion on the sharpness of the biasing error
estimate~\eqref{EqGenBound}, and in particular to Apprendix~\ref{SSSec:ToyModelOne} for
a situation with two QSDs.

\subsection{On pratical estimates of $T^{\mathrm{E}}_{\mathrm{Q}}$}\label{sec:estimTQ}
We discuss in this section two  ways to estimate $T^{\mathrm{E}}_{\mathrm{Q}}$.

\paragraph{Quasi-ergodicity.} The first one is based on the notion of
quasi-ergodicity.
\begin{assumption}[Quasi-ergodicity] 
\label{Ass:Finitecv}
There exist a QSD $\nu_{\mathrm{Q}}$ in $\cA$ and a constant $\eta<\infty$ such that 
\begin{equation}
\sum_{n  \geq 0} \norm{ \cL^{\nu_{\mathrm{E}}}(Y_n \vert \cT_{\cB}> n) - \nu_{\mathrm{Q}}}   \leq \eta .
\label{EqAssumFinitecv}
\end{equation}
\end{assumption}
Quasi-ergodicity provides an elementary bound on the relaxation time $T^{\mathrm{E}}_{\mathrm{Q}}$.
\begin{proposition}
\label{Prop:EqFinite}
Under Assumptions \ref{Ass:K} and \ref{Ass:Finitecv}, the relaxation time $T^{\mathrm{E}}_{\mathrm{Q}}$ to the QSD $\nu_{\mathrm{Q}}$ satisfies 
\begin{equation}
T^{\mathrm{E}}_{\mathrm{Q}}  \leq \eta .
\end{equation}
\end{proposition}

\paragraph{Geometric ergodicity.} A stronger assumption is the
geometric convergence of the conditioned process starting from
$\nu_{\mathrm{E}}$ to a QSD $\nu_{\mathrm{Q}}$. This provides a second manner to estimate~$T^{\mathrm{E}}_\mathrm{Q}$.
\begin{assumption}[Geometric ergodicity] 
\label{Ass:GeoErgoE}
There exist a QSD $\nu_{\mathrm{Q}}$, a constant $\alpha<\infty$, and a
constant $\rho \in (0,1)$ such that, for all $n\geq 0$,
\begin{equation}
\norm{ \cL^{\nu_{\mathrm{E}}}(Y_n \vert \cT_{\cB}> n) - \nu_{\mathrm{Q}}}  \leq \alpha \rho^n .
\label{EqAssumGeoErgodicE}
\end{equation}
\end{assumption}


%
%
%
%

Lemma \ref{Lemma:NuQConditioned} shows that, under Assumption \ref{Ass:GeoErgoE}, we have the upcoming result.

\begin{lemma} \label{Cor:GeoboundE}
 Under Assumptions \ref{Ass:K} and \ref{Ass:GeoErgoE}, the relaxation time satisfies
 \begin{equation}\label{eq:UBTQE}
T^{\mathrm{E}}_{\mathrm{Q}}  \leq \frac{\alpha}{1-\rho}.
 \end{equation}
\end{lemma}

\paragraph{Uniform geometric ergodicity.}
In practice, it may be easier to prove a stronger assumption than Assumption~\ref{Ass:GeoErgoE}, namely
the uniform geometric ergodicity, which writes:
\begin{assumption}[Uniform geometric ergodicity] 
\label{Ass:GeoErgo}
 There exist a QSD $\nu_{\mathrm{Q}}$, a constant $\alpha<\infty$, and a
constant $\rho \in (0,1)$ such that, for all $n\geq 0$,
\begin{equation}
\sup_{x \in \cA}\norm{ \cL^x(Y_n \vert \cT_{\cB}> n) - \nu_{\mathrm{Q}}}  \leq \alpha \rho^n .
\label{EqAssumGeoErgodic}
\end{equation}
\end{assumption}
With our definition \eqref{Eq:DefTV} of the total variation distance, one can prove that necessarily, in Assumption \ref{Ass:GeoErgo}, one has $\alpha\geq 1$, provided that $\cA$ is not reduced to a single point.
The uniform geometric ergodicity~\eqref{EqAssumGeoErgodic} is for example a consequence of
the two-sided condition stated in Equation~\eqref{Eq:UniformlyPos} as
will be illustrated in the upcoming section (see Section~\ref{Sec:birkoff},~\cite{Birkhoff1957},~\cite[Section
7.1]{Champagnat_Villemonais_2017},
and~\cite{atar1997exponential,le2004stability}). As already mentioned
in Remark~\ref{rem:unique_QSD}, Equation~\eqref{Eq:UniformlyPos}
actually implies that the so-called Dobrushin
ergodic coefficient is smaller than one, which also yields the uniform
geometric ergodicity~\eqref{EqAssumGeoErgodic}, see~\cite{del2001stability}
and~\cite[Section 12.2]{del2013mean}. For a thorough review of
sufficient conditions to get~\eqref{EqAssumGeoErgodic}, we refer to the recent
review paper~\cite{moral2021stability}. One can verify that if \eqref{EqAssumGeoErgodic} is satisfied, then for any initial distribution $\nu$ on $\cA$,
\begin{equation}
\norm{ \cL^\nu(Y_n \vert \cT_{\cB}> n) - \nu_{\mathrm{Q}}}  \leq \alpha \rho^n .
\label{lazxlkcx}
\end{equation}
 In particular, Assumption \ref{Ass:GeoErgo} implies the uniqueness of the quasi-stationary distribution
as well as
Assumption~\ref{Ass:GeoErgoE} and thus~\eqref{eq:UBTQE}. 
Actually, under~\eqref{EqAssumGeoErgodic}, one can prove the following equivalent of Lemma~\ref{Cor:GeoboundE}:
\begin{lemma} \label{Cor:Geobound}
 Under Assumptions \ref{Ass:K} and \ref{Ass:GeoErgo}, the relaxation time satisfies
 \begin{equation}\label{eq:UBTQ}
T^{\mathrm{E}}_{\mathrm{Q}} \le T_{\mathrm{Q}}  \leq \min \left\{ \frac{\alpha}{1-\rho}, \inf_{0< c <1}\frac{2}{1- c}\ceil[\Big]{ \frac{\ln(c \alpha^{-1})}{\ln(\rho)} }\right\} .
 \end{equation}
where $\ceil[\small]{\cdot} $ denotes the ceiling function. As a
consequence, $T_{\mathrm{Q}}$ is upper-bounded by $\min(\alpha, 2)$ when $\rho$ tends to
$0$. 
\end{lemma}

\section{Back to the case of a diffusion process}
\label{Sec:Diffusion}
\subsection{Verifying assumptions \ref{Ass:K} and \ref{Ass:GeoErgo}}
\label{sdlkcakc}
The setting in this section is the same as in
Section~\ref{sec:motiv}: the Markov chain $(Y_n)$ is
defined by~\eqref{Eq:MCDiffusion} from a diffusion process $(X_t)$
satisfying~\eqref{Eq:SDE}. Let us recall that the functions $f$ and $g$ in~\eqref{Eq:SDE} are
assumed to be smooth (globally Lipschitz is enough to get the
Feynman-Kac representations formulas in the proof of Proposition~\ref{prop:ellipticdiff} below) and such that $(X_t)$ is ergodic with respect to
a stationary measure which gives a non-zero probability to the sets
$A$ and $B$. The sets $A$ and $B$ are thus visited infinitely often and  the Markov chain $(Y_n)$ is then well defined for all
$n \ge 0$.

The purpose of the upcoming result is to exhibit some sufficient assumptions on the
diffusion $(X_t)$
for the associated Markov chain  $(Y_n)$ to  satisfy Assumptions
\ref{Ass:K} and \ref{Ass:GeoErgo}. Let us emphasize that we stick to a
relatively simple set of assumptions (smooth Lipschitz coefficients, elliptic
diffusion) for the sake of simplicity, but we expect the result to be
true for much more general diffusions, including the Langevin
dynamics~\eqref{eq:Langevin}, see also Remark~\ref{rem:Langevin}. The objective here is just to
illustrate how Assumptions
\ref{Ass:K} and \ref{Ass:GeoErgo} can be obtained in practice in a
simple setting.

\begin{proposition}\label{prop:ellipticdiff}
Let us assume that the domains $A$ and $B$ are chosen such that
Assumption \ref{Ass:K0}  on $\cA$ and $\cB$ is satisfied, meaning that $\cA= \partial A$ and $\cB= \partial B$ are compact disjoint sets and $ \cE=\cA\cup\cB$. Let us
assume moreover that the infinitesimal
generator of $(X_t)$
satisfying~\eqref{Eq:SDE} is elliptic, in the sense that
\begin{equation}\label{eq:elliptic}
\exists \lambda, \Lambda >0, \, \forall x \in \R^d, \, \lambda \le gg^T(x)\le \Lambda.
\end{equation}
Then, the associated Markov chain $(Y_n)$ defined by
\eqref{Eq:MCDiffusion} satisfies Assumptions~\ref{Ass:K} and~\ref{Ass:GeoErgo}.
\end{proposition}

\subsection{Numerical counterparts of the Hill relation}\label{sec:diff2}

Returning to the setting and notation of Section~\ref{sec:motiv}, see Equation~\eqref{eqFuncDelta},
let us
consider the case where $f=\Delta$ in Equation~\eqref{EqHillQSD}
in order to approximate the reaction time $T_{AB}$ in
Equation~\eqref{EqDeltaTimeWithPoisson}. We first rewrite  the
right-hand side of~\eqref{EqHillQSD} as follows:
\begin{align}
\expecin{\nu_{\mathrm{Q}}}{\sum_{n=0}^{\cT_{\cB} -1} \Delta(Y_n) }& =
 \frac{\expecin{\nu_{\mathrm{Q}}}{\Delta(Y_0)  \mathbf{1}_{Y_1 \in \cA} } }{\probin{\nu_{\mathrm{Q}}}{Y_1 \in \cB}} +  \frac{\expecin{\nu_{\mathrm{Q}}}{\Delta(Y_0)  \mathbf{1}_{Y_1 \in \cB} } }{\probin{\nu_{\mathrm{Q}}}{Y_1 \in \cB}}
\\
& = \econdin{\nu_{\mathrm{Q}}}{\Delta(Y_0)}{Y_1 \in \cA} \left( \frac{1}{\probin{\nu_{\mathrm{Q}}}{Y_1 \in \cB}} -1 \right) +
\econdin{\nu_{\mathrm{Q}}}{\Delta(Y_0)}{Y_1 \in \cB}   . \label{EqHillQSDprime}
\end{align}
The terms in~\eqref{EqHillQSDprime} can be computed as follows:
\begin{itemize}
\item $\econdin{\nu_{\mathrm{Q}}}{\Delta (Y_0)}{Y_1 \in \cA}$ is the mean time
of a loop starting from $\nu_\mathrm{Q}$ on $\cA$ to $\Upsigma$ and then
back to $\cA$ (without having touched $\cB$ in-between, which is anyway a
rare event which is thus not observed in
practice in direct simulation). It can be estimated by brute force Monte Carlo,
provided that one is able to sample from the QSD $\nu_\mathrm{Q}$. In
practice, sampling from the QSD is typically done by direct
simulations of the trajectories, using the fact that in our context, the process conditioned to stay in
$\cA$ quickly reaches the QSD, and rarely visits $\mathcal B$ from
$\mathcal A$. As explained above, this can be
formalized mathematically (see the assumption in
Theorem~\ref{Prop:EqGene}, and the discussion in Section~\ref{sec:estimTQ} on
practical estimates of the relaxation time to the QSD $T^{\mathrm{E}}_{\mathrm{Q}}$). Therefore, sampling
$\nu_{\mathrm{Q}}$ and estimating $\econdin{\nu_{\mathrm{Q}}}{\Delta
  (Y_0)}{Y_1 \in \cA}$  are done simultaneously by simulating the stationary state obtained
from successive loops from $\cA$ to $\Upsigma$ and then
back to $\cA$ (which indeed do not visit $\cB$ in-between). Notice
that samples from $\nu_\mathrm{Q}$ will also be used to estimate the quantities at stake in the
next two items.
\item $\probin{\nu_{\mathrm{Q}}}{Y_1 \in \cB}$ is the probability to
  observe a trajectory that starts from $\nu_\mathrm{Q}$ on $\cA$
  and directly goes to $\cB$ without going back to $\cA$ once $\Upsigma$ is crossed.  This probability is typically very
small, but can be efficiently estimated using rare event simulation
 methods such as splitting techniques (see for example Forward Flux Sampling~\cite{allen2009forward} or Adaptive Multilevel Splitting~\cite{cg,cerou2011multiple,teo-mayne-schulten-lelievre-16}).
\item $\econdin{\nu_{\mathrm{Q}}}{\Delta(Y_0)}{Y_1 \in \cB}$ is the mean
duration of a reactive trajectory that, starting from $\nu_\mathrm{Q}$ on $\cA$, crosses
$\Upsigma$ and then goes to $\cB$ without going back
to~$A$. Again, this is a quantity associated to the rare event $\{Y_1
  \in \cB\}$. As such, it can be approximated together with
$\probin{\nu_{\mathrm{Q}}}{Y_1 \in \cB}$, using the algorithms mentioned
in the previous item.
\end{itemize}
Notice that formula~\eqref{EqHillQSDprime} is 
exactly~\cite[Equation (10)]{cerou2011multiple} and very close
to~\cite[Equation (6)]{allen2009forward} (where the duration of the
reactive path is neglected), for example. In~\cite{allen2009forward},
this formula is called the ``effective positive flux'' formulation of
the rate constant, and it is mentioned that this formula is used in
combination with Forward Flux Sampling or Transition Interface Sampling
to get an estimate of the rate (which is the inverse of the 
reaction time). The mathematical setting presented here as well as Theorem~\ref{Prop:EqGene} thus give in
particular rigorous foundations to empirical methods
that have been used in the literature to estimate reaction times on
various molecular systems.

As explained in
Section~\ref{sec:motiv}, the submanifold $\Upsigma$ that is used to
define the Markov chain~$(Y_n)$ can be seen as a tuning
parameter: the reaction time $T_{AB}$ is the same whatever
this choice. As $\Upsigma$ is chosen further and further from $A$, it
is expected that the bias introduced when replacing
$\nu_\mathrm{E}$ by $\nu_\mathrm{Q}$ (analysed in Section~\ref{Sec:Bound}) gets larger and
larger since the underlying assumption that an equilibrium between
$A$ and $\Upsigma$ is reached before the transition becomes less
justified. Moreover, the sampling of $\nu_\mathrm{Q}$ becomes
more costly for the loops between $A$ and $\Upsigma$ are more
expensive to simulate. Nonetheless, when $\Upsigma$ gets further from $A$, the probability to observe a
transition to $B$ rather than to $A$ becomes larger, and thus easier to estimate. In this respect, it would be interesting to discuss if some
general recommendations on the choice of $\Upsigma$
could be given, taking into account the bias, the variance, and the algorithmic cost of
the involved estimators. Concerning the variance of the estimators, we refer
to~\cite[Chapter 3]{lopes19:PhD} for a discussion of importance
sampling methods which greatly reduce the variance of averages
over the reactive path ensemble.

Finally, let us mention that, in \eqref{EqHillQSDprime}, the first term is typically much larger than the second one in most situations of interest. For example, this is made explicit in \cite{MR3083930} for an overdamped Langevin dynamics in dimension 1 when the temperature parameter $T$ goes to zero.

\appendix

\section{Sharpness of the relative biasing error bound}
\label{SSec:DTTM}
In this section, we illustrate the sharpness of the bound in \Cref{Prop:EqGene} thanks to two discrete-time models.

\subsection{A toy example}
\label{SSSec:ToyModelOne}

\paragraph{Setting.}
 Consider the Markov chain $(Y_n)$ on $\{1,2,3\}$
 with transition matrix
 \begin{equation}
 K=
\begin{bmatrix}
 1-p & 0 & p \\ 
 q & 1-q &0 \\ 
 0 & r & 1-r
 \end{bmatrix}.
 \end{equation}
where the parameters $p$, $q$, and $r$ all belong to $(0,1)$.
Let $\cA :=\braces{1, 2}$ and $\cB :=\braces{3}$, so that $p^+ :=\sup_{x\in\cA}\probin{x}{Y_1 \in \cB}=p$ and $\nu_{\mathrm{E}} = \left[ 0,1 \right]$. $\cA$ and $\cB$ are metastable if $p\ll 1$ and $r\ll 1$. On this elementary example, a probabilistic reasoning on geometric laws straightforwardly gives $\expecin{ \nu_{\mathrm{E}}}{\cT_{\cB}} = \frac{1}{p}+\frac{1}{q}$. This may also be checked through direct computation since \eqref{Eq:MRT} implies that the mean hitting time of $\cB$ starting from a law $\nu$ on $\cA$ can be expressed as
 \begin{equation}
 \expecin{\nu}{\cT_{\cB}} = \nu (\id_{\cA}- K_{ \cA})^{-1} \mathds{1}_{\cA}.
 \end{equation}
 One may also see that
 $$\cL^1(Y_n \vert \cT_{\cB}> n)=[1,0],$$
 and
 $$ \cL^2(Y_n \vert \cT_{\cB}> n)=\frac{1}{\frac{q}{q-p}(1-p)^n-\frac{p}{q-p}(1-q)^n}\left[\frac{q}{q-p}\left((1-p)^n-(1-q)^n\right),(1-q)^n\right].$$
The eigenvalues $\lambda_1=(1-p)$ and $\lambda_2=(1-q)$  of $K_{\cA}$ are respectively associated to the left eigenmeasures $\nu_1=[1,0]$ and $\nu_2=[q/p,1-q/p]$. Consequently, if $0<p\leq q<1$, $\nu_1$ is the only quasi-stationary distribution for the process killed when leaving $\cA$,  and
$$H_1:= (\id_{\cA}-K_{\cA})^{-1}(\id_{\cA}-\mathds{1}_{\cA}\otimes \nu_1)=\frac{1}{q}\begin{bmatrix}
0& 0\\
-1& 1 
\end{bmatrix} .$$
Hence, we get 
$$T_1:=\norm{H_1}_\infty=\sup_{x \in \cA} \norm{H_1(x,\cdot)}=\frac{2}{q}=\norm{\nu_{\mathrm{E}}H_1}=:T_1^{\mathrm{E}}.$$
On the opposite, if $0<q<p<1$, there are two quasi-stationary distributions, namely $\nu_1$ and $\nu_2$. This time, we still have $T_1:=\norm{H_1}_\infty=\frac{2}{q}=\norm{\nu_{\mathrm{E}}H_1}=:T_1^{\mathrm{E}}$. However, we shall also consider
$$H_2:= (\id_{\cA}-K_{\cA})^{-1}(\id_{\cA}-\mathds{1}_{\cA}\otimes \nu_2)=\frac{1}{p}\begin{bmatrix}
1-q/p& q/p-1\\
-q/p& q/p 
\end{bmatrix},$$
and deduce
$$T_2:=\norm{H_2}_\infty=\sup_{x \in \cA} \norm{H_2(x,\cdot)}=\frac{2}{p}\max(1-q/p,q/p),$$
whereas $T_2^{\mathrm{E}}:=\norm{\nu_{\mathrm{E}}H_2}=2q/p^2$.

\paragraph{The case $0<p\leq q<1$.} 

In this situation, as just mentioned,  $\nu_1$ is the only QSD. 
Clearly, we have $ \expecin{\nu_1}{\cT_{\cB}} = \frac{1}{p}$, and the
relative biasing error~\eqref{EqExactBiaisTrace} is thus simply
$$\frac{\abs{ \expecin{ \nu_{\mathrm{E}}}{\cT_{\cB}} - \expecin{\nu_{\mathrm{Q}}}{\cT_{\cB}} }}{ \expecin{ \nu_{\mathrm{E}}}{\cT_{\cB}} } = \frac{p}{p+q}.$$
Since $p^+ =p$ and $T_1=T_1^{\mathrm{E}} = 2/q $, the requirement $
p^+T^{\mathrm{E}}_1<1$ of \Cref{Prop:EqGene} is satisfied as soon as
$p<q/2$. Accordingly, let us consider the regime $p^+T^{\mathrm{E}}_1 \ll 1$, i.e., $p \ll q$. On the one hand, one gets
$$\frac{\abs{ \expecin{ \nu_{\mathrm{E}}}{\cT_{\cB}} -
    \expecin{\nu_1}{\cT_{\cB}} }}{ \expecin{
    \nu_{\mathrm{E}}}{\cT_{\cB}} } \sim  \frac{p}{q}.$$
On the other hand, the bound on the relative biasing error
given by~\eqref{EqGenBound} scales like (take $f =\mathds{1}_{\cA}$):
$$2 \frac{p^+T_1}{1-p^+T_1}= 4 \frac{p}{q-2p} \sim 4 \frac{p}{q}.
$$
This illustrates the sharpness of
the bound in Theorem~\ref{Prop:EqGene}, even in a case where $\nu_\mathrm{E}$
differs a lot from the QSD $\nu_1$, meaning that they have disjoint supports.

Additionally, one can verify that Assumption \ref{Ass:GeoErgo} of uniform geometric ergodicity is fulfilled with $\alpha=2$ and $\rho=(1-q)/(1-p)$. Indeed, standard computations reveal that, for all $n \geq 0$ and all
initial condition $x \in \cA$,
$$\norm{ \cL^x(Y_n \vert \cT_{\cB}> n) - \nu_1}=2\frac{1}{1+\frac{q}{q-p}\left(\left(\frac{1-p}{1-q}\right)^n-1\right)}\leq 2\left(\frac{1-q}{1-p}\right)^n.$$
Considering the upper bound in Lemma \ref{Cor:Geobound} with $\alpha=2$ and $\rho=(1-q)/(1-p)$, one can then numerically check that, for any $0<p<q<1$,
$$\min \left\{ \frac{\alpha}{1-\rho}, \inf_{0< c <1}\frac{2}{1- c}\ceil[\Big]{ \frac{\ln(c \alpha^{-1})}{\ln(\rho)} }\right\}=\frac{\alpha}{1-\rho}=2\frac{1-p}{q-p}.$$
We retrieve the fact that the latter is always larger than $T_1=2/q$. More interestingly, in the regime $p\ll q$, we have 
$$\frac{\alpha}{1-\rho}=2\frac{1-p}{q-p}\sim\frac{2}{q}=T_1.$$

\paragraph{The case $0<q<p<1$.} 

This time, $\nu_1$ and  $\nu_2$ are the two quasi-stationary distributions of the process killed when leaving $\cA$. As such, Assumption \ref{Ass:GeoErgo} cannot be fulfilled.
One may notice that Assumption \ref{Ass:GeoErgoE} is not satisfied for $\nu_1$ because
$$\norm{\cL^{\nu_\mathrm{E}}(Y_n \vert \cT_{\cB}> n)-\nu_1}=\norm{\cL^2(Y_n \vert \cT_{\cB}> n)-\nu_1}\geq 2\frac{p-q}{p},$$
but the inequality
$$\norm{\cL^{\nu_\mathrm{E}}(Y_n \vert \cT_{\cB}> n)-\nu_2}=2\frac{p-q}{p}\frac{1}{\frac{p}{q}\left(\frac{1-q}{1-p}\right)^n-1}\leq 2\frac{q}{p}\left(\frac{1-p}{1-q}\right)^n$$
shows that Assumption \ref{Ass:GeoErgoE} is fulfilled for $\nu_2$ with $\alpha=2q/p$ and $\rho=(1-p)/(1-q)$. This is  consistent with Lemma \ref{Cor:GeoboundE}, which tells us that the relaxation time $T^{\mathrm{E}}_2=2q/p^2$ to $\nu_2$, starting from the reactive entrance distribution $\nu_\mathrm{E}$, satisfies 
\begin{equation}\label{zchaeceic}
T^{\mathrm{E}}_2\leq\frac{\alpha}{1-\rho}=2\frac{q(1-q)}{p(p-q)},
\end{equation}
and it is also worth noting that, in the regime $q\ll p$, this bound is tight.

Concerning the relative biasing errors, as before, we have
$$\frac{\abs{ \expecin{ \nu_{\mathrm{E}}}{\cT_{\cB}} - \expecin{\nu_1}{\cT_{\cB}} }}{ \expecin{ \nu_{\mathrm{E}}}{\cT_{\cB}} } = \frac{p}{p+q}.$$
Recall that $T_1= T_1^{\mathrm{E}}=2/q$, hence the condition $p^+T^{\mathrm{E}}_1<1$ is never satisfied when $0<q<p<1$. Nevertheless, since $T^{\mathrm{E}}_2=2q/p^2$, the condition $p^+T^{\mathrm{E}}_2<1$ is satisfied as soon as
$q<p/2$. Thus, under this condition, Theorem \ref{Prop:EqGene} yields
$$\frac{\abs{ \expecin{ \nu_{\mathrm{E}}}{\cT_{\cB}} - \expecin{\nu_2}{\cT_{\cB}} }}{ \expecin{ \nu_{\mathrm{E}}}{\cT_{\cB}} } = \frac{q}{p+q},$$
and for the right-hand side
$$ 2\frac{p^+T^{\mathrm{E}}_2}{1-p^+T^{\mathrm{E}}_2}=4\frac{q}{p-2q}.$$
We can remark that, in the regime $q\ll p$, up to a multiplicative factor
equal to 4, the upper-bound is sharp. Finally, since $T_2=2\max(1-q/p,q/p)/p$, the condition $p^+T_2<1$ is never fulfilled, and this illustrates the
importance of using $T^{\mathrm{E}}_{2}$ rather than
$T_2$ to measure the convergence time to the QSD.

\subsection{On the choice of $1/p^+$ to measure the reaction time}\label{sec:toymodel2}

The objective of this section is to answer to two questions related to
the three quantities: $\probin{\nu_{\mathrm{Q}}}{Y_1 \in \cB}$,
$\probin{\pi_{0 | \mathcal A}}{Y_1 \in \cB}$ and $p^+$ defined
by~\eqref{eq:p+}, and their use as a measure of the mean reaction time to $\cB$. The first question is:
Is $\probin{\nu_{\mathrm{Q}}}{Y_1 \in \cB}$ always larger than
$\probin{\pi_{0 | \mathcal A}}{Y_1 \in \cB}$? For the
toy model introduced in \Cref{SSSec:ToyModelOne} the probability $p =
\probin{\nu_{\mathrm{Q}}}{Y_1 \in \cB}$ indeed satisfies $p = p^{+} >\probin{\pi_{0 | \mathcal A}}{Y_1 \in \cB} $. 
The second and more important question is: Is $1/p^+$ a too pessimistic measure of the reaction time to
$\cB$, and, in particular, is the relative biasing error proportional to $p
T_{\mathrm{Q}}^{\mathrm{E}}$ or $\probin{\pi_{0 | \mathcal A}}{Y_1 \in \cB}
T_{\mathrm{Q}}^{\mathrm{E}}$ instead of the upper bound $p^+T^{\mathrm{E}}_{\mathrm{Q}}$ obtained in
Theorem~\ref{Prop:EqGene}? It turns out that the answers to both
questions are no, as will be shown in this section on a toy example. Finally, Remark \ref{kahcohdouzd} discusses the upper-bound given by \Cref{Cor:Geobound}. 

\paragraph{Setting.} Consider the Markov chain on $\set{1, 2,3}$ with $\cA=\set{1, 2}$ and $\cB=\set{3}$, with transition matrix 
\begin{equation}
K= \begin{bmatrix}
1-4 a& 3 a& a \\
2 b & 1-3 b & b \\
 a&a&1-2a
\end{bmatrix} ,
\end{equation}
where $0<b<a<1/4$, so that $p^+:=\sup_{x\in\cA}\probin{x}{Y_1 \in \cB} = a$. One can also notice that $\nu_{\mathrm{E}}=[1/2,1/2]$.

\paragraph{Properties of the killed chain and QSD.}
Consider the process killed when leaving $\cA$, whose sub-stochastic transition matrix is given by
\begin{equation}
K_{\cA}=
\begin{bmatrix}
1-4 a& 3 a\\
2 b& 1-3b 
\end{bmatrix} .
\end{equation}
The eigenvalues of $K_{\cA}$ are $\lambda_{1/2}= 1 - \frac{4a+3b \mp \sqrt{16 a^2+9b^2}}{2}$, with $\lambda_2 < \lambda_1$. Computing the left eigenvectors associated to these eigenvalues, one can check that the Markov chain admits a unique quasi-stationary distribution in $\cA$, namely
\begin{equation}
\nu_{\mathrm{Q}} = \frac{1}{a-b}[p-b, a-p] , 
\label{Eq:ExampleSQSD}
\end{equation}
 which is associated to
the largest eigenvalue $\lambda_1 = 1-p$, where
$$p=\probin{\nu_{\mathrm{Q}}}{Y_1 \in \cB}=\frac{4a+3b - \sqrt{16 a^2+9b^2}}{2}.$$ 
Indeed, is is readily seen
that for $0<b<a$, $\nu_{\mathrm{Q}}$ is a probability distribution since $p-b>0$ and
$a-p>0$. The left eigenvector corresponding to the eigenvalue
$\lambda_2$ cannot be chosen to be nonnegative so that, according to \Cref{Lemma:KilledEigen}, $\nu_{\mathrm{Q}}
$ is the unique quasi-stationary distribution.
Now, using Lemma \ref{Lemma:QSDandKilling}, one has
\begin{equation}
\expecin{\nu_{\mathrm{Q}}}{\cT_{\cB}}= \frac{1}{p} > \frac{2}{3 b} .
\label{Eq:ExBoundQSD}
\end{equation}
\paragraph{Properties of $\pi_{0|\cA}$.}
From $K$ we obtain $\pi_0=[5b,7a,6b]/(7a+11b)$ and deduce 
\begin{equation}
\pi_{0 | \mathcal A} = \frac{1}{7a+5b} [5b, 7a ] .
\end{equation}
Therefore 
\begin{equation}
\probin{\pi_{0 | \mathcal A}}{Y_1 \in \cB} = \frac{ 12 ab }{7a+5b}  .
\label{ksjcxbnaksjnc}
\end{equation}
When $0<b<a/5$, one always has
\begin{equation}
\probin{\pi_{0 | \mathcal A}}{Y_1 \in \cB} > 3b/2 >\probin{\nu_{\mathrm{Q}}}{Y_1 \in \cB} .
\label{hscbkajsbc}
\end{equation}
This answers negatively to the first question asked at the beginning
of this section.

\paragraph{Computation of the relative biasing error~\eqref{EqExactBiaisTrace}.}
For this Markov chain, the Hill relation with the reactive entrance
distribution (using Equation~\eqref{EqHillEntranceDist} with $f = \mathds{1}_{\cA}$) gives
\begin{equation}
\expecin{\nu_{\mathrm{E}}}{T_\cB} = \frac{1}{\probin{\pi_{0 | \mathcal A}}{Y_1 \in \cB}}
=\frac{7a+5b}{ 12 ab } .
\end{equation}
Therefore, if $0<b<a/5$,
\begin{equation}
\expecin{\nu_{\mathrm{E}}}{T_\cB} < \frac{2}{3b} < \expecin{\nu_{\mathrm{Q}}}{T_\cB}  .
\end{equation}
Since $\probin{\pi_{0 | \mathcal A}}{Y_1 \in \cB}
>\probin{\nu_{\mathrm{Q}}}{Y_1 \in \cB}$, the relative biasing error between mean hitting times satisfies
\begin{align}
\label{Eq:ExampleTR}
\abs{ \frac{ \expecin{\nu_{\mathrm{E}} }{\cT_{\cB}} - \expecin{\nu_{\mathrm{Q}}}{\cT_{\cB}} } {\expecin{\nu_{\mathrm{E}}}{\cT_{\cB}}}} &= \frac{ \probin{\pi_{0 | \mathcal A}}{Y_1 \in \cB} }{\probin{\nu_{\mathrm{Q}}}{Y_1 \in \cB} } -1 .
\end{align}

\paragraph{Relaxation time $T^{\mathrm{E}}_\mathrm{Q}$.}
For this Markov chain, one can compute the relaxation time
\begin{align}
T^{\mathrm{E}}_{\mathrm{Q}} &= \norm{\nu_{\mathrm{E}} (\id_{\cA}- K_{\cA})^{-1}(\id_{\cA}- \mathds{1}_{\cA} \otimes \nu_{\mathrm{Q} })}.
\end{align}
Indeed, since 
\begin{equation}
(\id_{\cA}- K_{\cA})^{-1} =\frac{1}{6ab}\begin{bmatrix}
3b &3a\\
2b	& 4a
\end{bmatrix} ,
\end{equation}
and as $\nu_{\mathrm{Q} }$ is given by~\eqref{Eq:ExampleSQSD}, we are led to
\begin{equation}
(\id_{\cA}- K_{\cA})^{-1}(\id_{\cA}- \mathds{1}_{\cA} \otimes \nu_{\mathrm{Q} })
=\frac{1}{6ab(a-b)} \begin{bmatrix}
6ab -3p(a+b) & - 6ab + 3p(a+b)\\
6ab -2p(2a+b)	& -6ab +2p(2a+b)
\end{bmatrix} 
\label{sjcnaznc}
\end{equation}
so that, since $\nu_{\mathrm{E}}=[\frac 1 2, \frac 1 2]$,
$$\nu_{\mathrm{E}} (\id_{\cA}- K_{\cA})^{-1}(\id_{\cA}- \mathds{1}_{\cA}
\otimes \nu_{\mathrm{Q} })= \frac{12 ab - p(7a+5b)}{12ab(a-b)}\left[1,-1\right]
$$
and
\begin{align}
T^{\mathrm{E}}_{\mathrm{Q}}& =\left|\frac{12 ab - p(7a+5b)}{6ab(a-b)}\right|.
\end{align}
Since $p<3b/2$ by \eqref{Eq:ExBoundQSD}, one gets that, when $0<b<a/5$,
\begin{equation}
T^{\mathrm{E}}_{\mathrm{Q}} = \frac{12 ab - p(7a+5b)}{6ab(a-b)} < \frac{2}{a-b} .
\label{Eq:ExampleLRT}
\end{equation}

\paragraph{Is $p T^{\mathrm{E}}_\mathrm{Q}$ an upper bound for~\eqref{EqExactBiaisTrace}?}
Using \eqref{Eq:ExBoundQSD}, one obtains
\begin{equation}
\frac{1}{pT^{\mathrm{E}}_{\mathrm{Q}}}> \frac{a-b}{3 b} .
\label{Eq:ExampleBAB}
\end{equation}
From~\eqref{ksjcxbnaksjnc}, \eqref{hscbkajsbc}, and \eqref{Eq:ExampleTR}, one also deduces
\begin{align}
\abs{ \frac{ \expecin{\nu_{\mathrm{E}} }{\cT_{\cB}} -
  \expecin{\nu_{\mathrm{Q}}}{\cT_{\cB}} }
  {\expecin{\nu_{\mathrm{E}}}{\cT_{\cB}}}} & > \frac{12 ab}{7a+5b}
                                           \frac{2}{3b} -1=
                                           \frac{a-5b}{7a+5b}  .
\label{Eq:ExampleBIB}
\end{align}
Combining \eqref{Eq:ExampleBAB} and \eqref{Eq:ExampleBIB},
one can check that
\begin{align}
\frac{1}{pT^{\mathrm{E}}_{\mathrm{Q}}}\abs{ \frac{ \expecin{\nu_{\mathrm{E}} }{\cT_{\cB}} - \expecin{\nu_{\mathrm{Q}}}{\cT_{\cB}} } {\expecin{\nu_{\mathrm{E}}}{\cT_{\cB}}}} 
&> \frac{a-b}{3 b} \frac{a-5b}{7a+5b}= \frac{a}{3b} \left(1-\frac{b}{a}\right) \frac{1- 5\frac{b}{a}}{7+5 \frac{b}{a}}
\end{align}
which is unbounded when $b=o(a)$. Therefore,
$pT^{\mathrm{E}}_{\mathrm{Q}}$ is not a bound for the relative biasing
error~\eqref{EqExactBiaisTrace}.

\paragraph{Is $ \probin{\pi_{0 |
    \mathcal A}}{Y_1 \in \cB} T^{\mathrm{E}}_{\mathrm{Q}}$ an upper bound for~\eqref{EqExactBiaisTrace}?}
One has
\begin{align}
\frac{1}{ \probin{\pi_{0 | \mathcal A}}{Y_1 \in \cB} T^{\mathrm{E}}_{\mathrm{Q}}}\abs{ \frac{ \expecin{\nu_{\mathrm{E}} }{\cT_{\cB}} - \expecin{\nu_{\mathrm{Q}}}{\cT_{\cB}} } {\expecin{\nu_{\mathrm{E}}}{\cT_{\cB}}}} &>
\frac{(a-b)(a-5b)}{24ab}
\end{align}
and the right-hand side is unbounded when $b=o(a)$. Therefore, $ \probin{\pi_{0 |
    \mathcal A}}{Y_1 \in \cB} T^{\mathrm{E}}_{\mathrm{Q}}$ is not either a bound for
the relative biasing error~\eqref{EqExactBiaisTrace}.

In conclusion, for this Markov chain, the relative biasing error is not proportional to $ p T^{\mathrm{E}}_{\mathrm{Q}}$ nor $\probin{\pi_{0 | \mathcal A}}{Y_1 \in \cB} T^{\mathrm{E}}_{\mathrm{Q}}$ in the regime $b= o(a)$.
This answers negatively to the second question asked at the beginning
of this section, and illustrates again the sharpness of our bias estimate.

\begin{remark}[About \Cref{Cor:Geobound}.]\label{kahcohdouzd}

Let  $\nu_{1/2}$, $u_{1/2}$ denote respectively left and right eigenvectors of $K_\cA$ associated with the eigenvalues $\lambda_{1/2}$. We can take
$$
\nu_1 =  \frac{1}{a-b}\left[p-b,a-p\right], \qquad u_1= \frac{1}{2p-(4a+3b)}\left[p-(3a+3b),p-(4a+2b)\right]^T $$
and
$$\nu_2 =  \frac{1}{a-b}\left[(4a+2b)-p,p-(3a+3b)\right], \qquad u_2 = \frac{1}{2p-(4a+3b)}\left[p-a,p-b\right]^T$$
so that $\nu_{\mathrm{Q}}= \nu_1$. Next, note that the law of $Y_n$ conditioned to stay in $\cA$  with initial distribution $\mu$ (row vector) can be written
\begin{equation}
 \cL^\mu(Y_n \vert \cT_{\cB}> n) = \frac{ \mu K_\cA^n }{ \mu K_\cA^n [1,1]^T}.
\end{equation}
Using the spectral decomposition
\begin{equation}
K_\cA=\lambda_1 u_1  \nu_{\mathrm{Q}} + \lambda_2 u_2  \nu_2,
\end{equation}
one obtains
\begin{equation}
  \cL^\mu(Y_n \vert \cT_{\cB}> n)  =  \nu_{\mathrm{Q}}+\frac{1}{1+\left(\frac{\lambda_1}{\lambda_2}\right)^n\frac{\mu u_1}{\mu u_2}}\left(\nu_2-\nu_{\mathrm{Q}}\right),
\end{equation}
and, therefore,
\begin{equation}
\norm{  \cL^\mu(Y_n \vert \cT_{\cB}> n) -  \nu_{\mathrm{Q}}} =\frac{1}{\left|1+\left(\frac{\lambda_1}{\lambda_2}\right)^n\frac{\mu u_1}{\mu u_2}\right|}\norm{\nu_2-\nu_{\mathrm{Q}}}.
\end{equation}
From now on, let us focus on the case $b=o(a)$. Then, we have in particular that $p\sim 3b/2$, $\lambda_1\sim 1$, and $\lambda_2\sim 1-4a>0$, so that $\rho:={\frac{\lambda_2}{\lambda_1}}\sim 1-4a$. One may also notice that
$$\norm{\nu_2-\nu_{\mathrm{Q}}}=\frac{2(4a+3b-2p)}{a-b}\sim 8.$$
Next, let us successively consider the initial conditions $x=1$ and $x=2$. On the one hand, if $\mu=\delta_1=[1,0]$, then 
$$\frac{\mu u_1}{\mu u_2}=\frac{\delta_1 u_1}{\delta_1 u_2}=\frac{(3a+3b)-p}{a-p}\sim 3>0,$$
which yields, for all $n\geq 0$,
\begin{equation}
\norm{  \cL^1(Y_n \vert \cT_{\cB}> n) -  \nu_{\mathrm{Q}}} =\frac{1}{1+\left(\frac{\lambda_1}{\lambda_2}\right)^n\frac{\delta_1 u_1}{\delta_1 u_2}}\norm{\nu_2-\nu_{\mathrm{Q}}}\leq \frac{\delta_1 u_2}{\delta_1 u_1}\norm{\nu_2-\nu_{\mathrm{Q}}}\rho^n.
\end{equation}
On the other hand, if $\mu=\delta_2=[0,1]$, then 
$$\frac{\mu u_1}{\mu u_2}=\frac{\delta_2 u_1}{\delta_2
  u_2}=\frac{p-(4a+2b)}{p-b}\sim\frac{-8a}{b},$$
which is smaller than $-1$ in the regime $b=o(a)$
and it is readily seen that, for all $n\geq 0$,
\begin{equation}
\norm{  \cL^2(Y_n \vert \cT_{\cB}> n) -  \nu_{\mathrm{Q}}} =\frac{1}{\left(\frac{\lambda_1}{\lambda_2}\right)^n\left|\frac{\delta_2 u_1}{\delta_2 u_2}\right|-1}\norm{\nu_2-\nu_{\mathrm{Q}}}\leq \frac{1}{\left|\frac{\delta_2 u_1}{\delta_2 u_2}\right|-1}\norm{\nu_2-\nu_{\mathrm{Q}}}\rho^n.
\end{equation}
Putting all things together, we conclude that Assumption \ref{Ass:GeoErgo} is fulfilled, namely
$$\forall x \in \cA, \, \forall n \ge0, \, \norm{ \cL^x(Y_n \vert \cT_{\cB}> n) - \nu_{\mathrm{Q}}}  \leq \alpha \rho^n,$$
with $\rho={\frac{\lambda_2}{\lambda_1}}\sim 1-4a$, and
$$\alpha:=\max\left(\frac{1}{\left|\frac{\delta_2 u_1}{\delta_2 u_2}\right|-1},\frac{\delta_1 u_2}{\delta_1 u_1}\right)\norm{\nu_2-\nu_{\mathrm{Q}}}=\frac{\delta_1 u_2}{\delta_1 u_1}\norm{\nu_2-\nu_{\mathrm{Q}}}\sim\frac{8}{3}.$$
In particular, considering the upper bound in \Cref{Cor:Geobound}, we have $\frac{\alpha}{1-\rho}\sim\frac{2}{3a}>2$ since $a<1/4$, and 
$$\inf_{0< c <1}\frac{2}{1- c}\ceil[\Big]{ \frac{\ln(c \alpha^{-1})}{\ln(\rho)} }\xrightarrow[a\to\frac{1}{4}]{}2.$$
Considering the relaxation times in \Cref{Cor:Geobound}, recall from \eqref{Eq:ExampleLRT} that
$$T^{\mathrm{E}}_{\mathrm{Q}} = \frac{12 ab - p(7a+5b)}{6ab(a-b)}\sim\frac{1}{4a},$$
and, by \eqref{sjcnaznc},
$$T_{\mathrm{Q}} =\frac{1}{3ab(a-b)}\max\left(|6ab -3p(a+b)|,|6ab -2p(2a+b)|\right)\sim\frac{1}{2a}\xrightarrow[a\to\frac{1}{4}]{}2.$$
This gives an example where the upper-bound in \Cref{Cor:Geobound} is reached by the right-hand term, and not by $\alpha/(1-\rho)$.

\end{remark}


\section{About the reversibility of the reactive entrance process}
\label{SSec:ExMarkovNonRev}


We consider the situation where $ \cE=\set{1, 2, \ldots,5}$ endowed with the discrete topology. Let $a,b,c,d$ be four strictly positive real numbers attached to the edges of the weighted undirected graph $G$ on \Cref{FigCounterExampleNonRev}. For $i,j$ in $ \cE$, the weight $w_{ij}$ is the value of the edge $(i,j)$ if it exists and zero otherwise. From this graph, let us consider the Markov chain on $ \cE$ with transition probability matrix
\begin{equation}
K_{ij} = \frac{w_{ij}}{\sum_{k=1}^{5} w_{ik}} \quad \forall i,j \in  \cE  .
\end{equation}
Therefore, the transition matrix is
\begin{equation}
K=\begin{bmatrix}
0 & \frac{a}{a+b+2d} &\frac{b}{a+b+2d} &\frac{d}{a+b+2d} &\frac{d}{a+b+2d}\\
\frac{a}{a+c} & 0 & 0 & 0 & \frac{c}{a+c}\\
\frac{b}{b+c} & 0 & 0 & \frac{c}{b+c}& 0\\
\frac{d}{c+d} & 0 &\frac{c}{c+d}& 0 & 0\\
\frac{d}{c+d} & \frac{c}{c+d} & 0& 0 & 0
\end{bmatrix} .
\end{equation}

\begin{figure}
\centering
\begin{tikzpicture}
\tikzset{node styleA/.style={state, 
 text=blue!30!white, 
 fill=gray!30!black,minimum size=1cm}}
 
 \tikzset{node styleB/.style={state, 
 text=red!30!white, 
 fill=gray!30!black,minimum size=1cm}}
 
 \tikzset{node styleS/.style={state, 
 text=gray!30!white, 
 fill=gray!30!black,minimum size=1.cm}}
 
 \node[node styleA] (S1) {$1$};
 \node[node styleA, below=1cm of S1] (S2) {$2$};
 \node[node styleA, above=1cm of S1] (S3) {$3$};

 \node[node styleS, below right=0.5cm and 3.5cm of S3] (S4) {$4$};
 \node[node styleS, below=1cm of S4] (S5) {$5$};

 \draw[auto=left] (S2) edge node {$a$} (S1);
 \draw[auto=left] (S1) edge node {$b$} (S3);
 \draw[auto=left] (S3) edge node {$c$} (S4);
 \draw[auto=left] (S2) edge node {$c$} (S5);
 \draw[auto=left] (S1) edge node {$d$} (S4);
 \draw[auto=left] (S1) edge node {$d$} (S5);

 \end{tikzpicture}
 \caption{Graph of a reversible Markov chain on $\set{1,2,\ldots,5}$ which is non-reversible with respect to its reactive entrance distribution in $\cA=\set{1,2,3}$.}\label{FigCounterExampleNonRev}
 \end{figure}
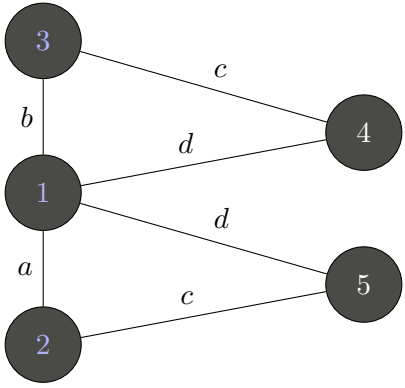

Assumption \ref{Ass:K} is clearly satisfied for this Markov chain which, by construction, is reversible with respect to its invariant distribution $\pi_0$ given by 
\begin{equation}
\pi^i_{0} = \frac{ \sum_{k=1}^5 w_{ik} }{\sum_{\ell=1}^5\sum_{k=1}^5 w_{\ell k}}\qquad \forall i\in  \cE .
\end{equation}
Thus, 
\begin{equation}
\pi_0=\frac{1}{2a+2b+4c+4d} \left[a+b+2d, a+c,b+c,c+d,c+d\right].
\end{equation}
 Taking $\cA=\set{1,2,3}$ and $\cB=\set{4,5}$, the reactive 
 entrance distribution given by \eqref{EqLinkDist} is
\begin{equation}
\nu_{\mathrm{E}}=\frac{1}{2(c+d)} \left[2d, c,c\right].
\end{equation} 
One can check that the process described by $(Y^{\mathrm{E}}_n)$ is non-reversible with respect to $\nu_{\mathrm{E}}$.
Indeed, let us denote by $K^{\mathrm{E}}$ its $3\times 3$ transition matrix. To show that the entrance process is non-reversible with respect to $\nu_{\mathrm{E}}$, we just have to verify that $K^{\mathrm{E}}_{32} \neq K^{\mathrm{E}}_{23}$, which is true when $a\neq b$ since we deduce from~\Cref{Def:REProcess} that
\begin{equation}
K^{\mathrm{E}}_{32}=\frac{abcd+abc^2+ bc^2d}{(a+b+2d)(a+c)(b+c)(c+d)}  ,
\end{equation}
whereas
\begin{equation}
K^{\mathrm{E}}_{23} =\frac{abcd+abc^2+ ac^2d}{(a+b+2d)(a+c)(b+c)(c+d)}  .
\end{equation}

\section{Proofs}
\label{SSec:Proofs}

\subsection*{Proof of Lemma \ref{Lemma:}}
The fact that $\sup_{x\in \cA}\probin{x}{Y_1 \in \cB}>0$ is a direct consequence of Assumptions~\ref{Ass:K2} and~\ref{Ass:K3}.
 Next, suppose that for all $n \in \N$, we have $\inf_{x \in \cA} \probin{x}{\cT_{\cB}  \leq n}=0$. Then,  one could exhibit a sequence $(x_n) \in \cA$ such that, for all $n$, we would have $\probin{x_n}{\cT_{\cB}  \leq n }  \leq 1/n$. Hence, for all $m  \leq n$, $\probin{x_n}{\cT_{\cB}  \leq m}  \leq 1/n$. By compactness of $\cA$ (see Assumption~\ref{Ass:K0}), up to extracting a subsequence of $(x_n)$, there exists $ \overline{x} \in \cA$ such that $x_n \rightarrow \overline{x}$. Now, as  $\cA$ and  $\cB$ are disjoint compact sets, the mapping $x\mapsto\mathbf{1}_{\cA}(x)$ is continuous and  
so is $x\mapsto K\mathbf{1}_{\cA}(x)$ by Assumption~\ref{Ass:K1}. Since $K=K_{\cA}$ on $\cA$ and $(x_n) \in \cA$, we deduce from \eqref{EqProbinGE} that the mapping $x \mapsto \probin{x}{\cT_{\cB}  \leq m} =1- K_{\cA}^{m}\mathds{1}_{\cA}(x) $ is continuous on $\cA$, so that
\begin{equation}
\forall m \in \N,\hspace{1cm} \probin{\overline{x}}{\cT_{\cB}  \leq m}  \leq \lim_{n \rightarrow \infty} \frac{1}{n} =0.
\end{equation}
Thus $\probin{\overline{x}}{\cT_{\cB} < \infty}=0$, which is in
contradiction with the $\pi_0$-irreducibility Assumption~\ref{Ass:K2} and Equation~\eqref{Eq:DefPhiIrreducibleMC}. Thus, we can  conclude that there exists an integer $n$ such that $\inf_{x \in \cA} \probin{x}{\cT_{\cB}  \leq n} > 0$.

\subsection*{Proof of Corollary \ref{Cor:Invert}}
Let us first prove that the operator $(\id_{\cA}-K_{\cA})$ is invertible. By \eqref{Eq:HittingTimes}, let $n_0$ be such that
$\inf_{x \in \cA} \probin{x}{\cT_{\cB}  \leq n_0} > 0$. Then, by definition of the operator norm defined in \eqref{Eq:DefNorm} and thanks to \eqref{EqProbinGE}, we are led to
\begin{equation}
\norm{K^{n_0}_{{\cA}}}_{\infty} 
= \supSur{x\in \cA} K^{n_0}_{\cA} \mathds{1}_{\cA}(x) = \sup_{x \in \cA}\probin{x}{\cT_{\cB} > n_0} <1 . \label{Eq:NormSmallOne}
\end{equation}
Hence, the inequality
\begin{equation}
\sum_{n=0}^\infty \norm{K^n_{\cA}}_{\infty}=\sum_{m=0}^{n_0-1}\sum_{\ell=0}^\infty \norm{K^{\ell n_0+m}_{\cA}}_{\infty}\leq\sum_{m=0}^{n_0-1}\norm{K^m_{\cA}}_{\infty}\sum_{\ell=0}^\infty (\norm{K^{n_0}_{\cA}}_{\infty})^\ell
 \label{aicjacj}
\end{equation}
ensures that the series $\sum_n \norm{K^n_{\cA}}_{\infty}$ converges and $(\id_{\cA}-K_{\cA})$ is indeed invertible as an operator on the Banach space $(B(\cA,\R),\norm{\cdot}_{\infty})$. The uniqueness of the solution
to~\eqref{EqNonDirichletBoundaryValueProblem}
is then immediate since the operator $(\id_{\cA} -
K_{ \cA})$ is invertible. Moreover,
\begin{align}
r(x)&= (\id_{\cA}- K_{\cA})^{-1} g(x)= \sum_{n \geq 0} K^n_{\cA} g(x)= \sum_{n \geq 0} \expecin{x}{g(Y_n) \mathbf{1}_{\cT_{\cB}>n}}.
\end{align}
To deduce \eqref{EqSolutionNonDirichletBoundaryValueProblem}, it remains to apply Fubini's theorem, which is possible since 
$$\sum_{n \geq 0} \expecin{x}{\left|g(Y_n) \mathbf{1}_{\cT_{\cB}>n}\right|}\leq \norm{g}_{\infty}\sum_{n \geq 0}K^n_{\cA} \mathds{1}_{\cA}(x)\leq \norm{g}_{\infty}\sum_n \norm{K^n_{\cA}}_{\infty}<\infty.$$

\subsection*{Proof of Proposition \ref{prop:pi0}}
Since $\pi_0 K=\pi_0$, one just has to consider the two-block decomposition~\eqref{EqTwoBlocks} to deduce that
$$\pi_0^{\cA} (\id_{\cA}- K_{\cA}) = \pi_0^{\cB} K_{\cB \cA}.$$
From this, by switching the roles of $\cA$ and $\cB$ and taking into account that~$(\id_{\cB}- K_{\cB})$ is invertible, we immediately get 
$$\pi_0^{\cB} = \pi_0^{\cA} K_{\cA \cB} (\id_{\cB}- K_{\cB})^{-1},$$
so that putting all things together yields
$$\pi_0^{\cA} (\id_{\cA}- K_{\cA}) = \pi_0^{\cA} K_{\cA \cB} (\id_{\cB}- K_{\cB})^{-1} K_{\cB \cA}.$$
Since $\pi_0(\cA)>0$, it suffices to divide both terms by $\pi_0(\cA)$ to obtain \eqref{Eq:pi0trace}.

\subsection*{Proof of the equivalence in Definition \ref{Def:REProcess}}
Let us check that both definitions coincide. Using the strong Markov property, one obtains that for all $x\in \cA$ and all $\cC \in \mathscr{B}(\cA)$
\begin{equation}
 \probin{x}{Y_{\cT^{\cA}} \in \cC} =\int_{\cB} \probin{x}{Y_{\cT_{\cB}^{+}} \in dy} \probin{y}{Y_{\cT_{\cA}^{+}} \in \cC} .
\end{equation}
The first term in this integral can be decomposed as
$$\probin{x}{Y_{\cT_{\cB}^{+}} \in dy}=\sum_{n=0}^\infty\int_{\cA}K_{\cA}^n(x,dx')K_{\cA\cB}(x',dy),$$
which amounts to saying that
$$\probin{x}{Y_{\cT_{\cB}^{+}} \in dy}=\left[ (\id_{\cA}-K_{\cA})^{-1} K_{\cA \cB}\right](x,dy).$$
In the same vein, we may write 
$$\probin{y}{Y_{\cT_{\cA}^{+}} \in \cC}=\left[ (\id_{\cB}-K_{\cB})^{-1} K_{\cB \cA}\right](y,\cC),$$
so that, finally,
$$ \probin{x}{Y_{\cT^{\cA}} \in \cC} =\left[ (\id_{\cA}-K_{\cA})^{-1} K_{\cA \cB}(\id_{\cB}-K_{\cB})^{-1} K_{\cB \cA}\right](x,\cC).$$
This shows the equivalence between the formulations of $K^\mathrm{E}$ in \Cref{Def:REProcess}.

\subsection*{Proof of Proposition \ref{Prop:ReactiveEntranceHarris}}
The chain $(Y^\mathrm{E}_n)$ inherits the positive Harris recurrent
property from the original Markov chain $(Y_n)$. The proof can be done
in two steps. First, let us show that the Markov chain
$Z_n=(Y_n,Y_{n+1})$ defined on $\mathcal E \times \mathcal E$ is
Harris recurrent for the probability measure $\pi_0 \otimes K$, where
$\pi_0 \otimes K (dx_0,dx_1) = \pi_0(dx_0) K(x_0,dx_1)$ and $\pi_0$ is
the equilibrium measure for $(Y_n)$. Let us
consider a measurable set $\overline{\mathcal C} \subset \mathcal E
\times \mathcal E$ such that $\pi_0 \otimes K(\overline{\mathcal C}) >
0$. We would like to check that, almost surely, $(Z_n)$ visits infinitely often the set
$\overline{\mathcal C}$. Let us define
$$\mathcal C=\{x \in \mathcal E, \, K(x, \overline{\mathcal C}_x) >0\} \text{ where }
\overline{\mathcal C}_x=\{y \in \mathcal E, \, (x,y) \in \overline{\mathcal C}\}.$$
One can check that $\pi_0(\mathcal C) > 0$. Indeed, $0 <\pi_0 \otimes
K(\overline{\mathcal C})= \int_{x \in \mathcal C} \pi_0(dx) K(x,
\overline{\mathcal C}_x)$ which implies $\pi_0(\mathcal C) > 0$.
Let us now introduce, for $k \ge 1$
$$\mathcal C_k=\{x \in \mathcal E, \, K(x, \overline{\mathcal C}_x) \ge 1/k\}$$
so that $\mathcal C=\cup_{k \ge 1} \mathcal C_k$ (increasing
union). One has $\lim_{k \to \infty} \pi_0(\mathcal C_k)=\pi_0(\mathcal
C)>0$ and let us therefore fix a $k \ge 1$ such that $\pi_0(\mathcal
C_k)>0$. By the Harris recurrence property and Remark~\ref{rem:psi}, almost surely, $(Y_n)$ visits infinitely
often the set $\mathcal C_k$, and each time $Y_n \in \mathcal C_k$, there is a
probability at least $1/k$ for $Y_{n+1}$ to be in
$\overline{\mathcal C}_{Y_n}$, and thus for $(Y_n,Y_{n+1})$ to be in~$\overline{\mathcal C}$. This implies that $(Z_n)$ visits
$\overline{\mathcal C}$ infinitely often, almost surely. This concludes the proof of
the Harris recurrence of $(Y_n,Y_{n+1})$. In particular,
$(Y_n,Y_{n+1})$ admits a unique invariant probability
measure, which obviously is $\pi_0 \otimes K$, for which the Harris
recurrence property holds, by Remark~\ref{rem:psi}.

In a second step, let us now conclude that
$(Y^\mathrm{E}_n)$ is Harris recurrent for the following probability measure on
$\mathcal A$: $\pi^{\mathcal B}_0 K_{\mathcal B A} / \pi^{\mathcal
  B}_0  K_{\mathcal B A} (\mathcal A)$. Let $ x \in \mathcal A$ and
$\mathcal C \subset \mathcal A$ such that $\pi^{\mathcal B}_0
K_{\mathcal B A}(\mathcal C)>0$. Let us fix the initial conditions $Y_0
= Y^\mathrm{E}_0 =x$,  and let us introduce $\overline{\mathcal C}=\mathcal B
\times \mathcal C$. One checks that $\pi_0 \otimes K
(\overline{\mathcal C}) = \pi^{\mathcal
  B}_0K_{\mathcal B \mathcal A}(\mathcal C) > 0$. Thus, almost surely, $(Y_n, Y_{n+1})$ visits infinitely
often $\overline{\mathcal C}$, and this in turn implies that $(Y^\mathrm{E}_n)$ visits
infinitely often~$\mathcal C$, which concludes the proof of the Harris
recurrence of $(Y^\mathrm{E}_n)$.

As a consequence, by Proposition \ref{oaihzdozihd}, $(Y^\mathrm{E}_n)$ admits, up to a multiplicative constant, a unique invariant measure. Since  
$$K^{\mathrm{E}} = (\id_{\cA}-K_{\cA})^{-1}K_{\cA\cB}(\id_{\cB}-K_{\cB})^{-1}K_{\cB \cA},$$
and, by \eqref{Eq:pi0trace},
$$\pi_{0 | \mathcal A}(\id_{\cA}-K_{ \cA}) =  \pi_{0 | \mathcal A} K_{ \cA \cB} (\id_{\cB}-K_{\cB})^{-1} K_{\cB \cA},$$
we see that this invariant measure is $\pi_{0 | \mathcal A}(\id_{\cA}-K_{ \cA})$. To normalize it, just notice that, via~\eqref{Eq:PropKAOverOne},
$$\pi_{0 | \mathcal A}(\id_{\cA}-K_{ \cA})\mathds{1}_{\cA}=\pi_{0 | \mathcal A}K_{\cA \cB } \mathds{1}_{\cB}=\probin{ \pi_{0 | \mathcal A}}{Y_1 \in \cB}.$$
To deduce \eqref{kajcpcj} from \eqref{EqLinkDist}, it suffices to apply \eqref{Eq:MRT}.

\subsection*{Proof of Remark \ref{Rem:EntranceStation}}
Let us give the proof of~\eqref{eq:pi0nuAnuB}. Notice first that
$$\pi_0(\cC)=\pi_{0 | \cA}(\cC) \pi_0(\cA) + \pi_{0 | \cB}(\cC)\pi_0(\cB).$$
 From Equation~\eqref{EqLinkDist}, we also know that 
 $$\pi_{0 | \mathcal A}(\cC)=  \left( \nu^{\cA}_{\mathrm{E}} (\id_{\cA}-K_{\cA})^{-1} \ind{\cA \cap
   \cC} \right) \probin{ \pi_{0 | \mathcal A}}{Y_1 \in \cB}.$$ 
 We thus have 
 $$\pi_0(\cC)=\left( \nu^{\cA}_{\mathrm{E}} (\id_{\cA}-K_{\cA})^{-1} \ind{\cA \cap
   \cC} \right) \probin{ \pi_{0}^{\mathcal A}}{Y_1 \in \cB} + \left( \nu^{\cB}_{\mathrm{E}} (\id_{\cB}-K_{\cB})^{-1} \ind{\cB \cap
   \cC} \right) \probin{ \pi_{0}^{\mathcal B}}{Y_1 \in \cA}.$$
 In addition, one can check that $\probin{ \pi_{0}^{\mathcal A}}{Y_1
  \in \cB} = \probin{ \pi_{0}^{\mathcal B}}{Y_1 \in \cA}$. Indeed,
from Equations~\eqref{Eq:PropKAOverOne}, \eqref{alzj} and~\eqref{Eq:StatPi0}, 
$$\probin{ \pi_{0}^{\mathcal A}}{Y_1
  \in \cB} =\pi_0^{\mathcal A} K_{\cA \cB} \ind{\cB}=\pi_0^{\mathcal
  A} (\id_{\cA}-K_{\cA}) \ind{\cA}=\pi_0^{\mathcal B} K_{\cB \cA} \ind{\cA}=\probin{ \pi_{0}^{\mathcal B}}{Y_1
  \in \cA}.$$
  Therefore, 
  $$\pi_0(\cC)= c^{-1} \bigpar{
  \nu_{\mathrm{E}}^{\cA}(\id_{\cA}-K_{\cA})^{-1} \ind{\cA \cap \cC}+
  \nu_{\mathrm{E}}^{\cB}(\id_{\cB}-K_{\cB})^{-1}\ind{\cB \cap \cC}}$$
where 
$$c^{-1}=\probin{ \pi_{0}^{\mathcal A}}{Y_1
  \in \cB} = \probin{ \pi_{0}^{\mathcal B}}{Y_1 \in \cA}$$ 
  is independent of $\cC$ and, necessarily,
$c=\expecin{ \nu_{\mathrm{E}}^{\cA}}{\cT_{\cB}}
  +\expecin{\nu_{\mathrm{E}}^{\cB}}{\cT_{\cA}}$ by considering
$\cC=\cE$ and \eqref{kajcpcj} since $\pi_0(\cE)=1$.

\subsection*{Proof of Lemma \ref{Lemma:KilledEigen}}
Recall that, for all $n\geq0$, \eqref{EqProbinGE} gives
\begin{equation}
\probin{\pi}{Y_n \in \cC,\cT_{\cB} >n}=\pi K_{ \cA}^n(\cC).
\label{uchozhcoaihjc}
\end{equation}
In particular, one has for any probability measure $\pi$ on $\cA$ and any $\cC \in \mathscr{B}(\cA)$
\begin{equation}
\pi K_{ \cA}(\cC) = \probin{\pi}{Y_1 \in \cC, \cT_{\cB}>1}.
\end{equation}
Thus, in view of \eqref{Eq:DefQSD}, if $\pi$ is a quasi-stationary distribution, this leads to
\begin{equation}
\pi K_{ \cA}= \probin{\pi }{ \cT_{\cB}>1} \pi  ,
\end{equation}
which shows that a quasi-stationary distribution is a left probability eigenmeasure associated to the eigenvalue $ \probin{\pi }{ \cT_{\cB}>1}$. Conversely, let us assume that there exists a probability measure $\pi$ on $\cA$ and a real number $\theta \geq 0$ such that $\pi K_{\cA} = \theta \pi$. Then, for any $n$, \eqref{uchozhcoaihjc} yields
$$\probin{\pi}{Y_n \in \cC,\cT_{\cB} >n}=\pi K^n_{\cA}(\cC) = \theta^n \pi(\cC)$$
and $\pi$ satisfies \eqref{Eq:DefQSD} with $\probin{\pi}{\cT_{\cB} > n}=\theta^n$.
As a consequence, showing the existence of a QSD amounts to asserting the existence of a left probability eigenmeasure for the kernel $K_{\cA}$.
Under \Cref{Ass:K}, the existence of such a QSD is ensured from \cite[Proposition 2.10]{Collet_Martinez_SanMartin_book} since $\cA$ is compact and $K$ is weak-Feller.

\subsection*{Proof of Lemma \ref{Lemma:QSDandKilling}}
For all $m,n \in \N$, \eqref{uchozhcoaihjc} gives
$$\probin{\nu_{\mathrm{Q}}}{\cT_{\cB} >m+n}=\nu_{\mathrm{Q}} K_{ \cA}^{m+n} \mathds{1}_{\cA}.$$
Taking into account that $\nu_{\mathrm{Q}}$ is a QSD, Lemma \ref{Lemma:KilledEigen} tells us that
$$\probin{\nu_{\mathrm{Q}}}{\cT_{\cB} >m+n}=\probin{\nu_{\mathrm{Q}}}{\cT_{\cB}>1}^{m+n},$$
which is the desired result.

\subsection*{Proof of Proposition \ref{Prop:InvMeasureReturn}}
Let us show that under~\Cref{Ass:K}, the $\pi$-return process is $\pi$-recurrent. Since
\begin{equation}
\expecin{x}{ \sum_{n =1}^\infty \mathbf{1}_{\braces{Y^\pi_n \in \cC}}}= \sum_{n =1}^\infty\probin{x}{Y^\pi_n \in \cC} \geq \sum_{n=1}^\infty \probin{x}{\cT_{\cB} = n} \pi(\cC) ,
\end{equation}
with $ \sum_{n=1}^\infty \probin{x}{\cT_{\cB} = n}=\infty$ by \eqref{kahcbikzck} and the fact that $(Y_n)$ is $\pi_0$-recurrent with $\pi_0(\cB)>0$, one obtains 
\begin{equation}
\expecin{x}{ \sum_{n =1}^\infty \mathbf{1}_{\braces{Y^\pi_n \in \cC}}}=\infty , \qquad \forall x \in \cA, \forall \cC \in \mathscr{B}(\cA) \text{ such that } \pi(\cC)>0 .
\end{equation}
Thus the $\pi$-return process is $\pi$-recurrent. Therefore, by Proposition \ref{oaihzdozihd}, it admits, up to a multiplicative constant, a unique invariant measure.
To check that $\pi (\id_{\cA}-K_{\cA})^{-1}$ is invariant, just notice that
\begin{equation}
\pi (\id_{\cA}-K_{\cA})^{-1} K^{\pi} = \pi (\id_{\cA}-K_{\cA})^{-1} K_{ \cA} +( \pi (\id_{\cA}-K_{\cA})^{-1} K_{\cA \cB}\mathds{1}_{\cB} ) \pi  .
\end{equation} 
Using \eqref{Eq:PropKAOverOne}, the right-hand side can be simplified as
\begin{equation}
\pi (\id_{\cA}-K_{\cA})^{-1} K^{\pi} = \pi (\id_{\cA}-K_{\cA})^{-1} K_{ \cA} + \pi .
\end{equation}
Therefore, one has
\begin{align}
\pi (\id_{\cA}-K_{\cA})^{-1} K^{\pi} &= \pi (\id_{\cA}-K_{\cA})^{-1} (K_{ \cA}-\id_{\cA})+ \pi (\id_{\cA}-K_{\cA})^{-1} + \pi= \pi (\id_{\cA}-K_{\cA})^{-1}.
\end{align}
Since \eqref{Eq:MRT} shows that $\expecin{\pi}{\cT_{\cB}}= \pi (\id_{\cA} - K_{\cA})^{-1} \mathds{1}_{\cA}$, the proof is complete.

\subsection*{Proof of Proposition \ref{PropHill}}
Since $R(\pi)$ is the stationary distribution of the $\pi$-return process with transition kernel~$K^{\pi}= K_{\cA} +(K_{\cA \cB} \mathds{1}_{\cB}) \otimes \pi $, \eqref{alzj} implies 
\begin{align}
R(\pi) (\id_{\cA}-K_{\cA}) &= ( R(\pi) K_{\cA \cB} \mathds{1}_{\cB} ) \pi  =\probin{R(\pi)}{Y_1 \in \cB}   \pi,
\end{align}
so that
\begin{equation}
\frac{R(\pi)}{\probin{R(\pi)}{Y_1 \in \cB}} = \pi (\id_{\cA}-K_{\cA})^{-1}  .
\label{EqProofReturnProcessHill}
\end{equation}
The Hill relation is then just a consequence of \eqref{lakcnlazknclckajcb}.

\subsection*{Proof of Proposition \ref{Prop:UniquStationaryReturnQSD}}
The transition kernel of the $\pi$-return process is $K^\pi = K_{ \cA} +( K_{\cA \cB}\mathds{1}_{\cB} ) \otimes \pi$, so one has $\pi=\pi K^\pi$  if and only if
$$\pi=\pi K_{ \cA}+( \pi K_{\cA \cB}\mathds{1}_{\cB} ) \pi=\pi K_{ \cA}+\probin{\pi}{ \cT_{\cB} =1 }\pi,$$
which is equivalent to say that $\pi K_{ \cA}=\probin{\pi}{ \cT_{\cB} >1 }\pi$ or,  by Lemma \ref{Lemma:KilledEigen}, that $\pi$ is a QSD for the process $(Y^{\mathrm{Q}}_n)$ killed when leaving $ \cA$. The Hill relation \eqref{EqHillQSD} is then a consequence of \eqref{eq:Hill}.

\subsection*{Proof of Lemma \ref{kajclajc}}
Since $\expecin{x}{\cT_{\cB}} = \sum_{k=0}^{\infty} \probin{x}{\cT_{\cB}>k}$, and
\begin{align}
 \probin{x}{\cT_{\cB}>k} &= \prod_{ \ell=1}^k \pcondin{x}{ \cT_{\cB} > \ell }{\cT_{\cB} >  \ell-1 } = \prod_{ \ell=1}^k \bigpar{1- \pcondin{x}{ \cT_{\cB} =  \ell }{\cT_{\cB} >  \ell-1 }} ,
\end{align} 
one obtains that
\begin{equation}
 \expecin{x}{\cT_{\cB}}  \geq \sum_{k=0}^{\infty} (1-p^+)^k = \frac{1}{p^+}.
\end{equation}

\subsection*{Proof of Proposition \ref{aicjapzcjpzc}}

Concerning~\eqref{EqHfReturnprime}, by the same reasoning as in the proof of Corollary~\ref{Cor:Invert}, Fubini's theorem yields 
\begin{align}
\expecin{x}{ \sum_{n = 0}^{\cT_{\cB}-1} \braces{ f(Y_n) - \nu_{\mathrm{Q}}f }} &= \sum_{n = 0}^{\infty} \expecin{x}{ \braces{ f(Y_n) - \nu_{\mathrm{Q}}f } \mathbf{1}_{\cT_{\cB} > n} } \\
&= \sum_{n = 0}^{\infty} \expecin{x}{ \braces{ f(Y^{\nu_{\mathrm{Q}}}_n) - \nu_{\mathrm{Q}}f } \mathbf{1}_{\cT_{\cB} > n} }  ,
\label{EqHfFirstPart}
\end{align}
where we have used that $\expecin{x}{f(Y_n) \mathbf{1}_{\cT_{\cB} > n} } =\expecin{x}{f(Y^{\nu_{\mathrm{Q}}}_n) \mathbf{1}_{\cT_{\cB} > n} } $.
 In addition,
\begin{align}
 \sum_{n = 0}^{\infty} \expecin{x}{ \braces{ f(Y^{\nu_{\mathrm{Q}}}_n) - \nu_{\mathrm{Q}}f } \mathbf{1}_{\cT_{\cB}  \leq n} } 
& = \sum_{n = 1}^{\infty} \econdin{x}{ f(Y^{\nu_{\mathrm{Q}}}_{n}) -
  \nu_{\mathrm{Q}}f }{\cT_{\cB}  \leq n} \probin{x}{ \cT_{\cB}
   \leq n} =0 . \label{EqHfStrongMarkov}
\end{align}
The last equality comes from the strong Markov property: indeed,  when the
${\nu_{\mathrm{Q}}}$-return process reaches~$\cB$, it is instantaneously redistributed
according to $\nu_{\mathrm{Q}}$, and then starting from the
QSD $\nu_{\mathrm{Q}}$, one has $\expecin{\nu_{\mathrm{Q}}}{ f(Y^{\nu_{\mathrm{Q}}}_n) - \nu_{\mathrm{Q}}f }=0$.
Equation~\eqref{EqHfReturnprime} is then obtained by adding~\eqref{EqHfFirstPart} and~\eqref{EqHfStrongMarkov}.
Let us finally notice that
\begin{align}
T^{\mathrm{E}}_{\mathrm{Q}} &\le T_{\mathrm{Q}}= \norm{H_{\mathrm{Q}}}_\infty=\norm{(\id_{\cA}-K_{\cA})^{-1}(\id_{\cA}-\mathds{1}_{\cA}\otimes \nu_{\mathrm{Q}})}_\infty\leq 2 \norm{(\id_{\cA}-K_{\cA})^{-1}}_{\infty}
\end{align} 
is finite under~\Cref{Ass:K}, by \eqref{aicjacj}.

\subsection*{Proof of Lemma \ref{Lemma:NuQConditioned}}
By definition of the $\nu_{\mathrm{Q}}$-return process, we have
\begin{equation}
\expecin{\mu}{ f(Y^{\nu_{\mathrm{Q}}}_n) \mathbf{1}_{\cT_{\cB}> n} } = \expecin{\mu}{ f(Y_n) \mathbf{1}_{\cT_{\cB}> n} } ,
\label{Eq:PppReturn}
\end{equation}
whereas
\begin{align}
\expecin{\mu}{ f(Y^{\nu_{\mathrm{Q}}}_n) \mathbf{1}_{\cT_{\cB} \leq n} } 
&= \sum_{m=0}^n \expecin{\mu}{ f(Y^{\nu_{\mathrm{Q}}}_n) \mathbf{1}_{\cT_{\cB} = m} }\\
&= \sum_{m=0}^n \expecin{\mu}{ \expecin{\nu_{\mathrm{Q}}}{ f(Y^{\nu_{\mathrm{Q}}}_{n-m}) } \mathbf{1}_{\cT_{\cB} = m }}= (\nu_{\mathrm{Q}}f ) \probin{\mu}{\cT_{\cB} \leq n}  .
\label{Eq:QqqReturn}
\end{align}
As a consequence, summing \eqref{Eq:PppReturn} and \eqref{Eq:QqqReturn}, one has
\begin{align}
\expecin{\mu}{f(Y^{\nu_{\mathrm{Q}}}_n) - \nu_{\mathrm{Q}}f } &= 
\expecin{\mu}{ f(Y_n) \mathbf{1}_{\cT_{\cB}>n}}+ ( \nu_{\mathrm{Q}}f ) \probin{\mu}{\cT_{\cB} \leq n} - \nu_{\mathrm{Q}}f \\
&= \expecin{\mu}{ f(Y_n) \mathbf{1}_{\cT_{\cB}>n}} - ( \nu_{\mathrm{Q}}f ) \probin{\mu}{ \cT_{\cB}>n}\\
&= \bigpar{ \econdin{\mu}{ f(Y_n) }{ \cT_{\cB}>n} - \nu_{\mathrm{Q}}f } \probin{\mu}{ \cT_{\cB}>n}  .
\end{align}
This yields the first claim of Lemma~\ref{Lemma:NuQConditioned}. The second claim is obtained by using the trivial bound $\probin{\mu}{ \cT_{\cB}>n}\leq1$.

\subsection*{Proof of Lemma \ref{lzeknvclakev}}

Since $\pi_{0 | \mathcal A}-\nu_{\mathrm{Q}}= \pi_{0 | \mathcal A}
(\id_{\cA}-\mathds{1}_{\cA}\otimes \nu_{\mathrm{Q}})$, we get, using~\eqref{EqKernelHQOp} for $H_\mathrm{Q}$:
\begin{align}
\pi_{0 | \mathcal A}-\nu_{\mathrm{Q}} &= \pi_{0 | \mathcal A} (\id_{\cA}-K_{\cA} )(\id_{\cA}-K_{\cA})^{-1} (\id_{\cA}-\mathds{1}_{\cA} \otimes \nu_{\mathrm{Q}})= \pi_{0 | \mathcal A} (\id_{\cA}-K_{\cA} )H_{\mathrm{Q}} .
\end{align}
Now, from~\eqref{EqLinkDist}, $\probin{ \pi_{0 | \mathcal A}}{Y_1 \in
  \cB} \nu_{\mathrm{E}}= \pi_{0 | \mathcal A} (\id_{\cA}-K_{ \cA})$, and
it follows that
$$\pi_{0 | \mathcal A}-\nu_{\mathrm{Q}}  = \probin{ \pi_{0 | \mathcal A}}{Y_1 \in
  \cB} \nu_{\mathrm{E}}H_{\mathrm{Q}}.$$
This yields~\eqref{EqDiffMeasure} by taking the total variation norm.

\subsection*{Proof of Theorem \ref{Prop:EqGene}}
From \eqref{EqExactBiaisTrace}, the relative biasing error is bounded by
\begin{align}
&\abs{ \frac{ \expecin{\nu_{\mathrm{E}}}{ \sum_{n=0}^{\cT_{\cB} -1}
  f(Y_n) } - \expecin{\nu_{\mathrm{Q}}}{ \sum_{n=0}^{\cT_{\cB} -1}
  f(Y_n) } } {\expecin{\nu_{\mathrm{E}}}{ \sum_{n=0}^{\cT_{\cB} -1}
  f(Y_n) }}}  \\
&\qquad \leq  \abs{
1-\frac{\probin{\pi_{0 | \mathcal A}}{Y_1 \in \cB}}{ \probin{\nu_{\mathrm{Q}}}{Y_1 \in \cB}}}
 + \frac{\probin{\pi_{0 | \mathcal A}}{Y_1 \in \cB}}{ \probin{\nu_{\mathrm{Q}}}{Y_1 \in \cB}} \abs{1 - \frac{\nu_{\mathrm{Q}}f}{ \pi_{0 | \mathcal A}f }} .
\end{align}
From~\eqref{EqDiffMeasure}, one has, for any test
function $f: \mathcal A \to \R$,
$$|\expecin{\pi_{0|\cA}}{f}-\expecin{\nu_{\mathrm{Q}}}{f}| \le \probin{\pi_{0 | \mathcal A}}{Y_1 \in \cB} T^{\mathrm{E}}_{\mathrm{Q}} \|f\|_\infty.$$
Applying this inequality with $f(x)=\probin{x}{Y_1 \in \cB}$ gives, since $\|f\|_\infty= p^+$, 
\begin{equation}
\abs{
1-\frac{\probin{\nu_{\mathrm{Q}}}{Y_1 \in \cB}}{ \probin{\pi_{0|\cA}}{Y_1 \in \cB}}}  \leq {p^+T^{\mathrm{E}}_{\mathrm{Q}}}  .
\label{EqProbinBoundTraceKilled}
\end{equation}
As $p^+T^{\mathrm{E}}_{\mathrm{Q}} < 1$ by assumption, one
deduces that
\begin{equation}
\abs{
1-\frac{\probin{\pi_{0 | \mathcal A}}{Y_1 \in \cB}}{ \probin{\nu_{\mathrm{Q}}}{Y_1 \in \cB}}}  \leq \frac{p^+T^{\mathrm{E}}_{\mathrm{Q}}}{1-p^+T^{\mathrm{E}}_{\mathrm{Q}}} .
\end{equation}
From this equation, we also have
\begin{equation}
\frac{\probin{\pi_{0 | \mathcal A}}{Y_1 \in \cB}}{ \probin{\nu_{\mathrm{Q}}}{Y_1 \in \cB}}  \leq \frac{1}{1-p^+T^{\mathrm{E}}_{\mathrm{Q}}} .
\end{equation}
Finally, using again~\eqref{EqDiffMeasure}, one obtains
\begin{equation}
\abs{1 - \frac{\nu_{\mathrm{Q}}f}{ \pi_{0 | \mathcal A}f }}  \leq 
T^{\mathrm{E}}_{\mathrm{Q}}\probin{\pi_{0 | \mathcal A}}{Y_1 \in \cB} \frac{\norm{f}_{\infty}}{| \pi_{0 | \mathcal A} f| }, 
\end{equation}
and the result follows by bounding $\probin{\pi_{0 | \mathcal A}}{Y_1 \in \cB}$ by $p^+$.

\subsection*{Proof of Proposition \ref{Prop:EqFinite}}
Let us consider the QSD $\nu_{\mathrm{Q}}$ introduced in Assumption \ref{Ass:Finitecv}.
Let $f: \cA \rightarrow \R$ be a test function, then for all $x\in \cA$, \eqref{EqHfReturnprime}, \eqref{EqHfFirstPart}, and \eqref{EqHfStrongMarkov} ensure that
\begin{equation}
\nu_{\mathrm{E}} H_{\mathrm{Q}}f = \sum_{n = 0}^{\infty} \bigexpecin{\nu_{\mathrm{E}}}{ \braces{ f(Y_n) - \nu_{\mathrm{Q}}f} \mathbf{1}_{\cT_{\cB}> n}}  .
\end{equation}
Note that, by conditioning,
\begin{align}
\bigexpecin{\nu_{\mathrm{E}}}{ \braces{ f(Y_n) - \nu_{\mathrm{Q}} f } \mathbf{1}_{\cT_{\cB}> n} } & = \econdin{\nu_{\mathrm{E}}}{f(Y_n) - \nu_{\mathrm{Q}}f}{ \cT_{\cB} >n } \probin{\nu_{\mathrm{E}}}{ \cT_{\cB} >n }\\
&= \left(\econdin{\nu_{\mathrm{E}}}{f(Y_n) }{ \cT_{\cB} >n } - \nu_{\mathrm{Q}}f \right) \probin{\nu_{\mathrm{E}}}{ \cT_{\cB} >n } .
\label{eqDiffExp}
\end{align}
For all $n  \geq 0$, let us introduce
\begin{equation}
I(n) = \sum_{m= 0 }^n \left( \econdin{\nu_{\mathrm{E}}}{f(Y_m) }{ \cT_{\cB} >m } - \nu_{\mathrm{Q}}f \right) ,
\end{equation}
with the convention $I(-1)=0$. One has
\begin{equation}
\econdin{\nu_{\mathrm{E}}}{f(Y_n) }{ \cT_{\cB} >n } - \nu_{\mathrm{Q}}f = I(n)-I(n-1) .
\end{equation}
Besides, notice that under \Cref{Ass:Finitecv},
\begin{equation}
\abs{ I(n) }  \leq \eta \norm{f}_{\infty} .
\label{Eq:InAss}
\end{equation}
As a consequence, a summation by parts yields
\begin{align}
\nu_{\mathrm{E}} H_{\mathrm{Q}}f &= \sum_{n  \geq 0} \bigpar{ I(n)- I(n-1)} \probin{\nu_{\mathrm{E}}}{\cT_{\cB} >n}\\
&= \sum_{n  \geq 0} I(n) \bigpar{ \probin{\nu_{\mathrm{E}}}{\cT_{\cB} >n} - \probin{\nu_{\mathrm{E}}}{\cT_{\cB} >n+1} }\\
&= \sum_{n  \geq 0} I(n) \probin{\nu_{\mathrm{E}}}{\cT_{\cB} = n+1}  .
\end{align}
Using \eqref{Eq:InAss}, the result follows by taking the supremum over all test functions $f$ such that $\norm{f}_\infty  \leq 1$ since,  under Assumption \ref{Ass:K},
 $ \sum_{n  \geq 0} \probin{\nu_{\mathrm{E}}}{\cT_{\cB} = n+1}= 1 $.

\subsection*{Proof of Lemma \ref{Cor:Geobound}}
Remember that $T_{\mathrm{Q}}$ defined by~\eqref{Eq:TQ} is an
upper-bound of $T^{\mathrm{E}}_{\mathrm{Q}}$.
Next, by the very definition of $T_{\mathrm{Q}}$ and Lemma \ref{Lemma:NuQConditioned}, one has
$$T_{\mathrm{Q}}= \sup_{x \in \cA} \norm{H_{\mathrm{Q}}(x,\cdot)}\leq\supSur{x \in \cA } \displaystyle \sum_{n=0}^\infty \norm{ \cL^x(Y^{\nu_{\mathrm{Q}}}_n) - \nu_{\mathrm{Q}}} \leq\supSur{x \in \cA } \displaystyle \sum_{n=0}^\infty \norm{ \cL^x(Y_n \vert \cT_{\cB}> n) - \nu_{\mathrm{Q}}}  \leq \frac{\alpha}{1-\rho} .
$$
Additionally, for any $0 <c <1$ and the deterministic time $T = \ceil[\Big]{\frac{\ln({c \alpha^{-1}})}{\ln(\rho)}}$, it holds
\begin{align}
\norm{ \cL^x(Y_T \vert \cT_{\cB}> T) - \nu_{\mathrm{Q}}}  \leq \alpha \rho^T  \leq c  .
\end{align}
Therefore, if $0< c <1 $, using \eqref{Def:TStop} and again the fact
that, by Lemma~\ref{Lemma:NuQConditioned}, 
$$\norm{ \cL^x(Y^{\nu_{\mathrm{Q}}}_n) - \nu_{\mathrm{Q}}} \leq\norm{ \cL^x(Y_n \vert \cT_{\cB}> n) - \nu_{\mathrm{Q}}},$$
one has $T_{\mathrm{stop}}(c)  \leq T$.
 By \eqref{Eq:RelTStop}, the relaxation time satisfies
 \begin{equation}
T_{\mathrm{Q}}  \leq \frac{2}{1-c} 
 \ceil[\Big]{\frac{\ln(c \alpha^{-1})}{\ln(\rho)} } .
 \end{equation}
When $\rho$ tends to $0$, this upper-bound converges to 2 by considering $c=\alpha\rho$.

\subsection*{Proof of Proposition \ref{prop:ellipticdiff}}
Recall that $S$ is such that $A \subset S$, $B \subset S^c$ and $\partial
S=\Upsigma$. Let us start with some notation and some classical
results to relate the process $(X_t)$ with the solutions to some partial
differential equations. Let $\varphi : \Upsigma \to \R$ and $\psi: \cA
\to \R$  be continuous functions. Define for all $x \in  \R^d$ 
\begin{equation}
u(x) = \expecin{x}{\varphi(X_{\tau_{\Upsigma}})} \text{ and } v(x) = \expecin{x}{\psi(X_{\tau_{\cA}}) \mathbf{1}_{\tau_{\cA} < \tau_{\cB}}}  ,
\end{equation}
where $\tau_{C} = \inf\setsuch{t>0}{ X_t \in C}$.
Hence, denoting by $L$ the generator of the diffusion process $(X_t)$,
$u$ and $v$ satisfy (see for example~\cite[Theorem 5.1]{friedman-75})
\begin{equation}\label{eq:u}
\begin{cases}
Lu(x)=0				& \text{for $x \in S $,}\\
u(x)=\varphi(x)				& \text{for $x \in \Upsigma$,}
\end{cases}
\end{equation}
and
\begin{equation}\label{eq:v}
\begin{cases}
Lv(x)=0				& \text{for $x \in (A \cup B)^c$,}\\
v(x)=\psi(x)				& \text{for $x \in \cA $,}\\
v(x)=0				& \text{for $x \in \cB $.}
\end{cases}
\end{equation}
In particular, $u$ and $v$ are $\mathcal C^\infty$ functions in the
interior of their domains of definition, from standard
elliptic regularity results (see for example~\cite[Corollary 8.11]{Gilbarg_Trudinger}). 
Then, denoting by $K_{S \Upsigma}(x,dy)$ the measure of the first
hitting point on $\Upsigma$ for the process $(X_t)$ starting from $x
\in S$, and $K_{(A\cup B)^c \cA}(x,dy)$ the measure of the first
hitting point on $\cA$ for the process $(X_t)$ starting from $x
\in (A \cup B)^c$ and reaching $\cA$ before $\cB$, one has 
\begin{equation}\label{aeicjccjacnjcn}
u(x) = \int_{\Upsigma} K_{S \Upsigma}(x,dy) \varphi(y) 
\text{ and } v(x) = \int_{ \cA} K_{ (A \cup B)^c \cA}(x,dy) \psi(y) .
\end{equation}

We first verify that Assumption~\ref{Ass:K1} is satisfied. Let us take a bounded
continuous function $f: \cE \to \R$. One would like to check that $Kf(x)
= \E^x[f(Y_1)]$ is a bounded continuous function of $x$. Let us
consider $x \in \cA$ (the reasoning is similar if $x \in \cB$). Then
$$Kf(x) = \E^x[f(Y_1) \mathbf{1}_{Y_1 \in \cA}] + \E^x[f(Y_1) \mathbf{1}_{Y_1 \in
  \cB}].$$
 Considering the first term (the reasoning is similar for
the second one), one has
$$\E^x[f(Y_1) \mathbf{1}_{Y_1 \in \cA}] = \E^x [ \varphi(X_{\tau_{\Upsigma}})]$$
where
$$\varphi(x)= \E^{x} [f(X_{\tau_A}) \mathbf{1}_{\tau_A < \tau_B}   ].$$
Thus $\varphi$ satisfies~\eqref{eq:v} for the boundary condition
$\psi=f \mathbf{1}_{\cA}$. In particular, $\varphi$ is continuous and bounded
on $\Upsigma$. Then, $\E^x [ \varphi(X_{\tau_{\Upsigma}})]$
satisfies~\eqref{eq:u}, and is thus again a continuous and bounded
function. This concludes the proof of Assumption~\ref{Ass:K1}.

Let us now prove Assumptions~\ref{Ass:K2} and~\ref{Ass:K3}. As shown in the proof of \cite[Theorem~1.7]{Lu_Nolen_2015},
the reactive entrance processes in $\cA$ and $\cB$ are positive Harris
recurrent under the ellipticity condition~\eqref{eq:elliptic}. Thus, by~\Cref{Rem:EntranceStation}, the Markov chain
$(Y_n)$ defined by~\eqref{Eq:MCDiffusion} is positive Harris
recurrent, with an invariant measure $\pi_0$ which satisfies~\eqref{eq:pi0nuAnuB}. This yields Assumption~\ref{Ass:K2}.
 Assumption~\ref{Ass:K3} is
satisfied because, via \eqref{kajcpcj} and \eqref{eq:pi0nuAnuB},
$$\pi_0(\cA)= \frac{\expecin{
    \nu_{\mathrm{E}}^{\cA}}{\cT_{\cB}}}{\expecin{
    \nu_{\mathrm{E}}^{\cA}}{\cT_{\cB}}
  +\expecin{\nu_{\mathrm{E}}^{\cB}}{\cT_{\cA}} }\hspace{1cm}\mathrm{and}\hspace{1cm} \pi_0(\cB)= \frac{\expecin{
    \nu_{\mathrm{E}}^{\cB}}{\cT_{\cA}}}{\expecin{
    \nu_{\mathrm{E}}^{\cA}}{\cT_{\cB}}
  +\expecin{\nu_{\mathrm{E}}^{\cB}}{\cT_{\cA}} } $$ 
  are non-zero since $\expecin{
    \nu_{\mathrm{E}}^{\cA}}{\cT_{\cB}}=T_{AB}$ and $\expecin{
    \nu_{\mathrm{E}}^{\cB}}{\cT_{\cA}}=T_{BA}$ are strictly positive and
  finite (see \cite[Proposition 1.8]{Lu_Nolen_2015}). 


It remains to establish that \Cref{Ass:GeoErgo} is fulfilled as well. Let $(Y^Q_n)$ be the process killed when leaving $\cA$. Its
(sub-Markov) transition kernel is given for $x \in \mathcal A$ by
$$K_{\mathcal A}(x,dy) = \int_{z \in \Upsigma} K_{S \Upsigma} (x, dz)
K_{(A\cup B)^c \mathcal A}
(z,dy).$$
Let us show that it satisfies the following two-sided condition: there
exists a non-zero positive finite measure $\pi_\cA$ and a constant
$C>0$ such that
\begin{equation}\label{eq:TS}
\pi_\cA(dy) \le K_{\mathcal A}(x,dy) \le C \pi_\cA(dy).
\end{equation}

Note that by the maximum principle, for any non-zero $\varphi \ge 0$,
$u$ (solution to~\eqref{eq:u}) is strictly positive on $S$. Moreover, by the Harnack inequality for elliptic operators \cite[Corollary~9.25]{Gilbarg_Trudinger} and the compactness of $\cA$, we then have 
\begin{equation}
0 \le \supSur{x \in \cA}{u(x)}  \leq C \infSur{x \in \cA}{u(x)},
\end{equation}
where the constant $C$ is independent of $\varphi \ge 0$ (considering \eqref{Eq:SDE}, it only
depends on the lower and upper bounds of $gg^T$ and on the maximum of $|f|$ on
some compact set $\cA' \subset S$ which contains a neighborhood of
$\cA$). Therefore, by \eqref{aeicjccjacnjcn}, for all smooth function $\varphi \ge 0$, one has
$$0 \le \supSur{x \in \cA}{\int_{\Upsigma} K_{S \Upsigma}(x,dy) \varphi(y) } \le C \infSur{x \in \cA}{\int_{\Upsigma} K_{S \Upsigma}(x,dy) \varphi(y)},$$
with equality to $0$ if and only if $\varphi=0$. Let $O$ be a non-empty open
subset of $\cA$, and let us introduce $\varphi_O(x) = K_{(A \cup B)^c
  \cA}(x,O)$. The function $\varphi_O$ is smooth in the interior of $(A \cup B)^c$ by standard
regularity results
on elliptic operators (since it satisfies~\eqref{eq:v} with
$\psi=\mathbf{1}_{O}$), and non-zero (by the maximum principle, since
$\mathbf{1}_{O}$ is nonnegative
and non-zero). One thus has:
\begin{equation}\label{eq:harnack}
0 < \supSur{x \in \cA}{\int_{z\in \Upsigma} K_{S \Upsigma}(x,dz)  K_{(A \cup B)^c
  \cA}(z,O) } \le C \infSur{x \in \cA}{\int_{z\in \Upsigma} K_{S \Upsigma}(x,dz)  K_{(A \cup B)^c
  \cA}(z,O)},
\end{equation}
where $C$ is independent of $O$. Let us now introduce
$$\pi_\cA(dy) = \inf_{x \in \cA} \int_{z \in \Upsigma}  K_{S \Upsigma}(x,dz)  K_{(A \cup B)^c
  \cA}(z,dy).$$
This is a nonnegative measure on $\cA$ as the infimum of positive measures
(see for example~\cite[Lemma 5.2]{Champagnat_Villemonais_16}) which is non-zero
since $\pi_{\cA}(O) > 0$ for any non-empty open set $O \subset \cA$ thanks
to~\eqref{eq:harnack}. Notice also that $\pi_{\cA}(\cA) \le
1$. Moreover, from~\eqref{eq:harnack}, one has, for all $x\in\cA$,
$$\pi_\cA(dy) \le \int_{z \in \Upsigma}  K_{S \Upsigma}(x,dz)  K_{(A \cup B)^c
  \cA}(z,dy) \le C \pi_\cA(dy)$$
which yields~\eqref{eq:TS}.

Assumption \ref{Ass:GeoErgo} is then a consequence of
the two-sided condition~\eqref{eq:TS}, and actually of the more general
two-sided condition stated in Equation~\eqref{Eq:UniformlyPos}.  See
for example
Proposition~\ref{prop:twosided}, which
yields~\eqref{EqAssumGeoErgodic} with $\rho=\frac{C-1}{C+1}$.

\section{Two-sided condition and convergence to the QSD}\label{Sec:birkoff}

The objective of this appendix is to rewrite in our specific probabilistic
setting the results
of~\cite{Birkhoff1957} to prove existence, uniqueness and
  convergence to the quasi-stationary distribution
for a sub-Markov kernel under the so-called two-sided condition
(see Equations~\eqref{Eq:UniformlyPos} and~\eqref{eq:TS} above and Equation~\eqref{eq:2S} below). We thus do not claim any
originality here, and this material is only provided for the sake of
completeness. This result is stated
in~\cite{champagnat-coulibaly-pasquier-villemonais-18}, but not proved
exactly in the discrete-time setting we consider here, see also~\cite[Section
7.1]{Champagnat_Villemonais_2017}. Notice that Birkhoff's seminal
paper~\cite{Birkhoff1957} is followed by a large body of literature, see in
particular~\cite{liverani-95,nussbaum-88} and references therein. See
also the
two references~\cite{atar1997exponential,le2004stability} for very
similar statements.

Let us emphasize again (see Remark~\ref{rem:unique_QSD}) that the
two-sided condition is one example of a sufficient condition to get
exponential convergence to the quasi-stationary distribution.  Equation~\eqref{Eq:UniformlyPos}
actually implies that the so-called Dobrushin
ergodic coefficient is smaller than one, which also yields the uniform
geometric ergodicity~\eqref{EqAssumGeoErgodic}, see~\cite{del2001stability}
and~\cite[Section 12.2]{del2013mean}. For a thorough review of
sufficient conditions to get~\eqref{EqAssumGeoErgodic}, we refer to the recent
review paper~\cite{moral2021stability}.

This section is
organized as follows. After introducing some notation in
Section~\ref{sec:setting}, the Hilbert's projective metric is defined
in Section~\ref{sec:Hilbert}. This is a projective metric on the set
of positive (non-zero) measures, which thus defines a metric on the
probability measures. In Section~\ref{sec:T}, we analyse how the distance
between two measures evolves under the application of a sub-Markov
kernel $K$. Section~\ref{sec:FP} finally gives the main result, namely the
existence and convergence to a QSD under a contraction assumption
(in the spirit of the Banach fixed point theorem), and
the fact that this contraction assumption is satisfied under the
two-sided condition~\eqref{eq:2S}.

\subsection{The setting}\label{sec:setting}

Let us consider the topological vector space $\mathcal M$ of Radon measures on a
Polish space $(X,d)$ such that
\begin{equation}\label{eq:finite_meas}
\forall \lambda \in \mathcal M, \, \lambda(X) < \infty,
\end{equation}
 equipped with the topology of the convergence in
distribution. We denote by $\mathcal X$ the ensemble of Borel sets on
$X$. For $\lambda$ and $\mu$ in $\mathcal M$, we
denote $\lambda \le \nu$ if for all Borel set $A \in \mathcal X$, $\lambda(A) \le
\nu(A)$. This defines a partial ordering such that if $\lim_{n \to \infty} \lambda_n
= \lambda$, and $\lambda_n \ge \nu$ for all $n$, then $\lambda \ge
\nu$. 
Let us now define the  convex cone
$$\mathcal M_+ = \{ \lambda \in \mathcal M \text{ s.t. } \lambda \ge 0
\text{ and } \lambda \neq
0\}.$$
For any $\lambda \in \mathcal M_+ $, one thus has $\lambda(X)>0$. Notice that $\mathcal M_+ \cup \{0\}$ is closed. Let us introduce
the equivalence relation:
$$\lambda \sim \nu \iff \exists c > 0, \lambda=c\nu.$$
The quotient of $\mathcal M_+$ under $\sim$ is a closed
convex set, which, thanks to~\eqref{eq:finite_meas}, can be identified with the ensemble of probability
measures on $(X,d)$, denoted by $\mathcal M_1$ in the following.

Let us now consider a non-zero sub-Markov kernel\footnote{One can
  check that all the results presented in this appendix still hold assuming that $K: X \times \mathcal X \to
[0,M]$ for some $M>0$. We stick to the case $M=1$ having in mind the
probabilistic framework of a sub-Markov kernel.} $K: X \times \mathcal X \to
[0,1]$:
\begin{itemize}
\item for a fixed $A \in \mathcal X$, $x\mapsto K(x,A)$ is measurable;
\item for a fixed $x \in X$, $A \mapsto K(x,A)$ is a measure with
  total mass smaller than $1$.
\end{itemize}
Let us denote by 
$$\mathcal T:
\left\{
\begin{aligned}
\mathcal M &\to \mathcal M\\
\lambda &\mapsto \int_{x \in X} \lambda(dx)  K(x,dy) 
\end{aligned}
\right.
$$
the associated transition kernel on measures. It is a linear map,
which is such that
$$\mathcal T(\mathcal M_+) \subset \mathcal M_+ \cup \{0\}.$$
In the notation of Section~3, $\mathcal T (\lambda) = \lambda K$.
Notice that $\mathcal T$ is a bounded operator if $\mathcal M$ is
endowed with the total variation norm, since for any two measures
$\lambda$ and $\nu$, and any Borel set $A \in \mathcal X$,
$$|\mathcal T (\lambda) (A) - \mathcal T (\nu) (A)| = \int_{x \in X}
K(x,A) (\lambda-\nu)(dx) \le \|K(x,A)\|_{\infty} \|\lambda -
\nu\|_{TV} \le \|\lambda -
\nu\|_{TV} .$$
In all the following, it is assumed that
\begin{equation}\label{eq:AS1}
\forall \lambda \in\mathcal M_+, \, \mathcal T (\lambda) \neq 0,
\end{equation}
so that
$$\mathcal T(\mathcal M_+) \subset \mathcal M_+.$$
\begin{lemma}\label{lem:AS1}
  The assumption~\eqref{eq:AS1} is equivalent to
\begin{equation}\label{eq:AS1prime}
\forall x \in  X, \, K(x, X)\neq 0,
\end{equation} 
\end{lemma}
\begin{proof}
  Since $\mathcal T(\lambda)=\lambda K \neq 0$ for any $\lambda \in
  \mathcal M_+$, it is obvious that~\eqref{eq:AS1} implies~\eqref{eq:AS1prime} by
simply considering $\lambda=\delta_x$, for any $x \in X$. Let
us now assume that~\eqref{eq:AS1prime} holds, and let us consider
$\lambda \in \mathcal M_+$. It is clear that $\mathcal T (\lambda) \ge
0$, so that it only remains to check that $\mathcal T
(\lambda)( X)>0$ to get~\eqref{eq:AS1}. Let us assume that $\mathcal T
(\lambda)( X)=0$ which rewrites as $\int_{y \in X} \int_{x \in X}
\lambda(dx) K(x,dy)=0$. By Fubini's theorem, one then gets $\int_{x \in X}
\lambda(dx) K(x,X)=0$ which implies that $\lambda (\{x, \, K(x,X)>
0\})=0$. But this is in contradiction with the facts that, on the one
hand $\{x, \, K(x,X)> 0\}=X$, and on the other hand $\lambda(X)>0$,
since $\lambda \in \mathcal M_+$.
\end{proof}
Notice that~\eqref{eq:AS1prime} is not a very strong assumption, in
the sense that if $K(x,X)=0$ for some $x$, one can modify $X$ to $X
\setminus \{x, \, K(x,X)=0\}$ to satisfy \eqref{eq:AS1prime}.
We will make explicit in Proposition~\ref{prop:twosided} below a
sufficient practical condition (the two-sided condition~\eqref{eq:2S}) to get~\eqref{eq:AS1prime}.

\subsection{The Hilbert's projective metric}\label{sec:Hilbert}
The Hilbert's projective metric $\Theta$ on $\mathcal M_+$ is defined
by
$$\forall \lambda, \nu \in \mathcal M_+, \, \Theta(\lambda,\nu) = \ln \left( \frac{C(\lambda,\nu)}{ c(\lambda,\nu)}\right),$$
where
$$c(\lambda,\nu)=\sup\{ c >0, \, c \lambda \le \nu\}\hspace{1cm}\mathrm{and}\hspace{1cm}C(\lambda,\nu)=\inf\{ C >0, \,\nu \le  C \lambda \}.$$
We use the standard conventions that $c(\lambda,\nu)=0$ if $\{ c
>0, \, c \lambda \le \nu\}=\emptyset$ and $C(\lambda,\nu)=\infty$ if $\{ C
>0, \,\nu \le  C \lambda \} = \emptyset$. One can check that for any
measures $\lambda, \nu, \rho \in \mathcal M_+$,
$\Theta(\lambda,\nu) =\Theta(\nu, \lambda)$,  $\Theta(\lambda,\nu) \le
\Theta(\lambda,\rho) +  \Theta(\rho,\nu)$ and $\Theta(\lambda,\nu) =0$
if and only if $\lambda=t\nu$ for some $t >0$.
Notice that if $\Theta(\lambda,\nu)<\infty$ then $\lambda$ and
$\nu$ are equivalent. Finally, remark that this is a projective metric
in the sense that for any $t >0$ and $u>0$, $\Theta(t\lambda,u\nu)=\Theta(\lambda,\nu)$: it is
thus a metric on $\mathcal M_1$, and only a pseudo-metric on $\mathcal M_+$.

Let
us make a link between the Hilbert's projective metric and other more
usual metrics. One can read the next two lemmas with the total
variation norm in mind, as an example.
\begin{lemma}\label{lem:norm}
Let $\|\cdot\|$ be a norm on $\mathcal M$ such that
\begin{equation}\label{eq:vector_lattice}
\forall \lambda,\nu \in \mathcal M, \quad -\lambda \le \nu \le \lambda \Rightarrow \|\nu\| \le \|\lambda\| .
\end{equation}
Then, for any $\lambda, \nu$ in $\mathcal M_+$ such that $\|\lambda\|=\|\nu\|$, one has
$$\|\lambda - \nu \| \le \left(\exp(\Theta(\lambda,\nu)) -1 \right) \|\lambda\|.$$
\end{lemma}
\begin{proof}
We follow~\cite[Lemma 1.3]{liverani-95}. Let $\lambda, \nu \in \mathcal M_+$ such that $\|\lambda\|=\|\nu\|$. Notice that necessarily, under \eqref{eq:vector_lattice}, $c(\lambda,\nu) \le 1$ and $C(\lambda,\nu) \ge 1$. Indeed
$0 \le \nu- c(\lambda,\nu) \lambda$ and thus $-\nu \le 0 \le
c(\lambda,\nu) \lambda
\le \nu$ which yields $c(\lambda,\nu) \|\lambda\| \le \|\nu\|$. The
proof is similar to get $C(\lambda,\nu) \ge 1$.
Therefore, we have
$$\nu-\lambda \le (C(\lambda,\nu)-1) \lambda \le (C(\lambda,\nu)-
c(\lambda,\nu))  \lambda,$$
and
$$\nu-\lambda \ge (c(\lambda,\nu)-1) \lambda \ge -(C(\lambda,\nu)-
c(\lambda,\nu))  \lambda.$$
This implies
$$\|\nu-\lambda\|\le (C(\lambda,\nu)-
c(\lambda,\nu)) \|\lambda\| \le \frac{(C(\lambda,\nu)-
c(\lambda,\nu))}{
c(\lambda,\nu)}\|\lambda\|  = \left(\exp(\Theta(\lambda,\nu)) -1 \right) \|\lambda\|.$$
\end{proof}

\begin{lemma}\label{lem:complete}
Let $\|\cdot\|$ be a norm on $\mathcal M$ such
that~\eqref{eq:vector_lattice} holds, and such that $\mathcal M$ is
complete for this norm. Then the set
$$\mathcal M^{\|\cdot \|}_1=\{\lambda \in \mathcal M_+, \, \|\lambda\|=1\}$$
is complete for the Hilbert's projective metric $\Theta$.
\end{lemma}
\begin{proof}
We rely here on~\cite[Theorem 5]{birkhoff-62} or~\cite[Lemma
4]{Birkhoff1957}. Let $(\lambda_n)_{n \ge 0}$ be a Cauchy sequence for the metric $\Theta$,
with values in $\mathcal M^{\|\cdot \|}_1$. One can extract a subsequence
$\nu_i=\lambda_{n(i)}$ such that for all $i \ge 0$, $\Theta(\nu_i,\nu_{i+1}) \le
2^{-(i+1)}$. Therefore, for all $i \ge 0$,
$$c(\nu_i,\nu_{i+1}) \nu_i \le \nu_{i+1} \le C(\nu_i,\nu_{i+1})
\nu_i $$
with $\ln(C(\nu_i,\nu_{i+1}) / c(\nu_i,\nu_{i+1})) \le 2^{-(i+1)}$. 
By the same argument as in the proof of Lemma~\ref{lem:norm}, one has
$$|\nu_{i+1}-\nu_i| \le \left(\exp(\Theta(\nu_{i},\nu_{i+1})) -1\right) \nu_i.$$
Using the fact that 
$$\exp(2^{-(i+1)})
-1=\int_0^{2^{-(i+1)}} \exp(x) \, dx \le 2^{-(i+1)} \exp(1/2) \le 2^{-i},$$
this yields
\begin{align}
|\nu_{i+1}-\nu_i| &\le \left(\exp(2^{-(i+1)}) -1\right) \nu_i \le 
                    2^{-i} \nu_i. \label{eq:nui}
\end{align}
From this inequality and since $\|\nu_i\|=1$, one gets 
$$\|\nu_{i+1}-\nu_i\| \le 2^{-i},$$
which implies that $\nu_i$ converges to some $\nu_\infty \in \mathcal M^{\|\cdot \|}_1$ when $i \to
\infty$ in the $\|\cdot\|$-norm because $\mathcal M$ is
assumed complete for the $\|\cdot\|$-norm. 
From~\eqref{eq:nui}, one gets
$$ (1-2^{-i}) \nu_i \le  \nu_{i+1} \le (1+2^{-i}) \nu_i, $$
and thus, for $1 \le i < j$,
$$\prod_{k=i}^{j-1}  (1-2^{-k} ) \nu_i \le \nu_j \le \prod_{k=i}^{j-1}
(1+2^{-k} ) \nu_i.$$
Using Lemma~\ref{lem:prod} below, this yields, for $1 \le i < j$,
$$  (1-2^{2-i} ) \nu_i \le \nu_j \le (1+ \exp(1) 2^{1-i}) \nu_i$$
and thus
$$|\nu_j - \nu_i| \le \exp(1) 2^{1-i} \nu_i.$$
By letting
$j \to \infty$,
$$|\nu_\infty - \nu_i|\le \exp(1) 2^{1-i} \nu_i.$$
This implies that $\lim_{i \to \infty} \Theta(\nu_i,\nu_\infty) =
0$. Using the triangular inequality, one thus obtains (remember that $\nu_i=\lambda_{n(i)}$)
$$\Theta(\lambda_n,\nu_\infty) \le \Theta(\lambda_n,\lambda_{n(i)}) +
\Theta(\nu_i,\nu_\infty),$$
which goes to zero when $n$ goes to infinity. Indeed, for $\epsilon >0$, one
first chooses $n_0$ such that for all $m \ge n \ge n_0$, $\Theta(\lambda_n,\lambda_{m}) \le \epsilon
/2$ and then, for any $n \ge n_0$, one takes $i$ sufficiently large so that
$n(i)\ge n$ and $\Theta(\nu_i,\nu_\infty)\le\epsilon/2$. This yields
that for all $n \ge n_0$, $\Theta(\lambda_n,\nu_\infty) \le \epsilon$.
\end{proof}
Lemmas~\ref{lem:norm} and~\ref{lem:complete} apply for example to the total
variation norm $\|\cdot\|_{TV}$, since the vector space $\mathcal M$ of Radon measures is complete for the
total variation norm, and since \eqref{eq:vector_lattice} is satisfied. In this case $\mathcal M^{\|\cdot\|_{TV}}_1=\mathcal M_1$ is simply the set of
probability measures. Let us conclude this section with a simple
technical lemma which has been used in the previous proof.
\begin{lemma}\label{lem:prod}
Let $\alpha \in (0,1/2]$. Then, for any $1 \le i \le  j$,
$$1-2 \alpha^{i-1} \le \prod_{k=i}^j (1-\alpha^{k}) \le \prod_{k=i}^j
(1+\alpha^{k}) \le  1 + \exp(1) \alpha^{i-1}.$$
\end{lemma}
\begin{proof}
Let $1 \le i \le j$,
$$
\ln\left( \prod_{k=i}^j (1+\alpha^{k}) \right)
=\sum_{k=i}^j \ln (1+\alpha^{k})\le \sum_{k=i}^j \alpha^{k}\le
  \frac{\alpha^{i}}{1-\alpha} \le \alpha^{i-1}
$$
and thus
$$
\prod_{k=i}^j (1+\alpha^{k})
\le\exp\left( \alpha^{i-1}\right) \le 1 + \exp(1) \alpha^{i-1}.
$$
Likewise, using the fact that for $x \in (0,1/2)$, $\frac{1}{1-x} \le 1+2x$,
$$
\ln\left( \prod_{k=i}^j (1-\alpha^{k}) \right)
=- \sum_{k=i}^j \ln \left(\frac{1}{1-\alpha^{k}}\right)\ge -
  \sum_{k=i}^j \ln \left(1+2\alpha^{k}\right) \ge - 2 \sum_{k=i}^j
  \alpha^{k} \ge -2 \alpha^{i-1}
$$
and thus
$$\prod_{k=i}^j (1-\alpha^{k})
\ge\exp\left( - 2 \alpha^{i-1}\right) \ge 1 - 2 \alpha^{i-1}.
$$

\end{proof}

\subsection{The projective metric norm $\Delta$ of $\mathcal T$}\label{sec:T}
\begin{proposition}
Let us define
$$\Delta = \sup_{\lambda,\nu \in \mathcal M_+} \Theta(\mathcal T (\lambda), \mathcal T
(\nu)).$$
Then, for all $\lambda$ and $\nu$ in $\mathcal M_+$,
\begin{equation}\label{eq:contrac}
\Theta(\mathcal T (\lambda), \mathcal T (\nu)) \le \tanh\left( \frac
  \Delta 4\right) \Theta(\lambda,\nu),
\end{equation}
with the convention $\tanh (+\infty)=1$.
\end{proposition}
\begin{proof}
We here follow~\cite[Theorem 1.1]{liverani-95}. For a geometric interpretation of
the computations, we refer to~\cite{Birkhoff1957}. Let  $\lambda$ and $\nu$ in
$\mathcal M_+$ be two positive (non-zero) measures. If
$c(\lambda,\nu)=0$ or $C(\lambda,\nu)=\infty$, then~\eqref{eq:contrac}
is satisfied. Likewise, if $c(\lambda,\nu)=C(\lambda,\nu)$, which is
equivalent to say that $\nu$ is proportional to $\lambda$,
then~\eqref{eq:contrac} is satisfied.
Otherwise, denoting for simplicity $c=c(\lambda,\nu)$
and $C=C(\lambda,\nu)$, one has, using the continuity property of the
partial ordering mentioned in Section~\ref{sec:setting},
$$c \lambda \le \nu \le C \lambda, \text{ with } C > c > 0 $$
and
$$\Theta(\lambda,\nu) =\ln \left( \frac{C}{c}\right) \in (0, \infty).$$
If $\Delta = +\infty$, then~\eqref{eq:contrac} holds since $\mathcal T(\nu - c\lambda
) \ge 0$ and $\mathcal T(C \lambda - \nu) \ge 0$ implies
$$c \mathcal T(\lambda) \le \mathcal T(\nu) \le C \mathcal T(\lambda),$$
so that $c(\mathcal T(\lambda),\mathcal T(\nu)) \ge c$ and $C(\mathcal T(\lambda),\mathcal T(\nu)) \le C$
which yields
$$\Theta(\mathcal T(\lambda),\mathcal T(\nu))\le \ln  \left(
  \frac{C}{c}\right)=\Theta(\lambda,\nu).$$
If $\Delta < \infty$, then one has, by assumption,
$$\Theta(\mathcal T(\nu - c\lambda
), \mathcal T(C \lambda - \nu)) \le \Delta,$$
which implies that there exist two positive real numbers $m$ and $M$ such that
\begin{equation}\label{eq:ineq}
m \mathcal T(\nu - c\lambda
) \le \mathcal T(C \lambda - \nu)) \le M \mathcal T(\nu - c\lambda
),
\end{equation}
and
\begin{equation}\label{eq:mM}
\ln  \left( \frac{M}{m}\right) \le \Delta.
\end{equation}
Notice that~\eqref{eq:ineq} rewrites:
$$\frac{cM+C}{M+1} \mathcal T(\lambda) \le \mathcal T(\nu) \le
\frac{cm+C}{m+1} \mathcal T(\lambda).$$
Therefore,
\begin{align}
\Theta( \mathcal T(\lambda),  \mathcal T(\nu)) &\le
\ln \left( \frac{cm+C}{m+1} \frac{M+1}{cM+C} \right) \nonumber\\
&=\ln \left( \frac{m+\exp(\Theta(\lambda,\nu) )}{m+1} \right) - \ln \left( \frac {M+ \exp(\Theta(\lambda,\nu) )}{M+1}
\right)  \nonumber\\
&=\int_0^{\Theta(\lambda,\nu) } \frac{\exp(x)}{m+\exp(x)} -
  \frac{\exp(x)}{M+\exp(x)} \, dx  \nonumber\\
&=\int_0^{\Theta(\lambda,\nu) } \varphi(\exp(x)) \, dx  \nonumber \\
& \le \Theta(\lambda,\nu) \max_{\R_+} \varphi,  \label{eq:maxtheta}
\end{align}
where
$$\varphi(y)=\frac{y}{m+y} - \frac{y}{M+y} =- \frac{m}{m+y} +
  \frac{M}{M+y}.$$
The function $\varphi$ attains its maximum over $\R_+$ at $y=\sqrt{mM}$, and
  its maximum value is
$$
\max_{\R_+} \varphi =- \frac{m}{m+\sqrt{mM}} +
  \frac{M}{M+\sqrt{mM}}= \frac{-\sqrt{m} + \sqrt{M}}{\sqrt{m} + \sqrt{M}}= \frac{
1-       \sqrt{\frac{m}{M}}}{1 + \sqrt{\frac{m}{M}}},
$$
and thus, using the fact that, by~\eqref{eq:mM}, $\frac{m}{M}\ge \exp(-\Delta)$,
$$
\max_{\R_+} \varphi
\le \frac{
1-       \exp(-\Delta/2) }{1 + \exp(-\Delta/2)} = \tanh\left(\frac
  \Delta 4\right).
$$
Plugging this upper bound of $\max_{\R_+} \varphi$
in~\eqref{eq:maxtheta} yields~\eqref{eq:contrac}.
\end{proof}
One thus gets a contraction in the Hilbert's projective metric if
$\Delta < \infty$. Notice that we implicitly used the fact that
$\mathcal T (\mathcal M_+) \subset \mathcal M_+$ to define $\Delta$
(otherwise $\Theta( \mathcal T \lambda, \mathcal T \nu)$ may not be
defined). This explains why~\eqref{eq:AS1} is needed in the first place.

\subsection{Fixed point theorem and two-sided condition}\label{sec:FP}

For what follows, recall that $\mathcal T^0(\lambda)=\lambda$ by convention.

\begin{theorem}\label{theo:CV}
Let us assume that
\begin{equation}\label{eq:AS2} 
\Delta = \sup_{\lambda,\nu \in \mathcal M_+} \Theta(\mathcal T (\lambda), \mathcal T
(\nu)) < \infty,
\end{equation}
and let us introduce $\rho=\tanh\left( \frac \Delta 4 \right) \in (0, 1)$.
Then, there exists a unique probability measure $\nu_\infty \in
\mathcal M_1$ such that
\begin{equation}\label{eq:FP}
\mathcal T (\nu_\infty) = c \nu_\infty
\end{equation}
for some $c > 0$. Moreover, for any $\lambda_0 \in \mathcal M_+$, one
has, for all $n \ge 0$,
\begin{equation}\label{eq:expo_cv}
\Theta(\nu_n,\nu_\infty) \le \frac{\rho^n}{1-\rho}  \Theta
(\nu_{1},\nu_{0}) 
\end{equation}
where $\nu_n=\frac{\mathcal T^n(\lambda_0)}{\mathcal
  T^n(\lambda_0)(X)}$. This implies in particular: for all $n \ge 0$,
\begin{equation}\label{eq:expo_cv2}
\|\nu_n-\nu_\infty\|_{TV} \le \frac{\Theta
(\nu_{1},\nu_{0})}{1-\rho} \exp\left(\frac{\Theta
(\nu_{1},\nu_{0})}{1-\rho} \right) \rho^n.
\end{equation}
\end{theorem}
\begin{proof}
We refer to~\cite[Theorem 1]{Birkhoff1957} for a similar reasoning.
The uniqueness of a solution to~\eqref{eq:FP} is easy to obtain from
the assumption $\Delta < \infty$. Indeed, let us assume that two
probability measures $\mu_1$ and $\mu_2$ are such that
$$\mathcal T (\mu_1) = c_1 \mu_1 \text{ and } \mathcal T (\mu_2) = c_2
\mu_2$$
for some $c_1>0$ and $c_2>0$.
Then, using~\eqref{eq:contrac}, one has
$$\Theta (\mathcal T (\mu_1),\mathcal T (\mu_2)) \le \rho \Theta
(\mu_1,\mu_2)$$
where $\rho=\tanh\left( \frac \Delta 4 \right) \in (0, 1)$. Therefore,
$$\Theta (c_1 \mu_1,c_2 \mu_2) \le \rho \Theta
(\mu_1,\mu_2)$$
and one gets that $ \Theta
(\mu_1,\mu_2)=0$ since $\Theta (c_1 \mu_1,c_2 \mu_2) =\Theta  (\mu_1,
\mu_2) $. This implies $\mu_1=\mu_2$ since both are probability measures.

We will now show the existence of a solution to~\eqref{eq:FP} using
the standard argument of Banach fixed-point theorem. Let $\lambda_0 \in \mathcal M_+$, and let us consider, for $n \ge 0$,
$$\lambda_n = \mathcal T^n (\lambda_0).$$
Using~\eqref{eq:contrac}, one has, for all $n \ge 1$,
$$\Theta (\lambda_{n+1},\lambda_n) \le \rho \Theta
(\lambda_{n},\lambda_{n-1}).$$
Thus, for all $ n \ge 1$,
\begin{equation}\label{eq:expo_1}
\Theta (\lambda_{n+1},\lambda_n) \le \rho^{n} \Theta
(\lambda_{1},\lambda_{0}),
\end{equation}
and the triangular inequality yields, for all $m \ge n \ge 1$,
$$\Theta(\lambda_m, \lambda_n) \le \left(\sum_{k=n}^{m-1} \rho^k
\right) \Theta
(\lambda_{1},\lambda_{0}) \le \frac{\rho^n}{1-\rho}  \Theta
(\lambda_{1},\lambda_{0}). $$
Let us introduce the probability measures: $\forall n \ge 0$,
$$\nu_n = \frac{\lambda_n}{ \lambda_n(X)}.$$
One has that $\Theta(\nu_m, \nu_n) =\Theta (\lambda_m, \lambda_n)$ and
thus for all $m \ge n \ge 1$,
\begin{equation}\label{eq:cauchy}
\Theta(\nu_m, \nu_n) \le \frac{\rho^n}{1-\rho}  \Theta
(\nu_{1},\nu_{0}) 
\end{equation}
which shows that $(\nu_n)_{n \ge 0}$ is a Cauchy sequence for the
$\Theta$-metric. From Lemma~\ref{lem:complete}, this implies that
$\nu_n \in \mathcal M_1$ converges to some $\nu_\infty \in \mathcal M_1$ as $n \to \infty$ in the $\Theta$-metric, and
thus, from Lemma~\ref{lem:norm}, in total variation
norm. From~\eqref{eq:expo_1}, one has that
$$\lim_{n \to \infty} \Theta\left(\frac{\mathcal T (\nu_n)}{(\mathcal T (\nu_n))(X)},\nu_n\right)=0$$
and thus, using Lemma~\ref{lem:norm},
$$\lim_{n \to \infty}\left\|\frac{\mathcal T (\nu_n)}{(\mathcal T
  (  \nu_n))(X)} - \nu_n \right\|_{TV} = 0.$$
Since $\lim_{n \to \infty} \nu_n =\nu_\infty$ in total variation norm,
then $\lim_{n \to \infty}\mathcal T (\nu_n)= \mathcal T (\nu_\infty) $ in
total variation norm, and thus $\lim_{n \to \infty}(\mathcal T
    (\nu_n))(X)= (\mathcal T (\nu_\infty))(X)$. One thus obtains:
$$\frac{\mathcal T (\nu_\infty)}{(\mathcal T (\nu_\infty))(X)} = \nu_\infty$$
which establishes the existence of the solution to~\eqref{eq:FP}.

Moreover, by
letting $m \to \infty$ in~\eqref{eq:cauchy}, one
gets~\eqref{eq:expo_cv}. Using Lemma~\ref{lem:norm}, one obtains from~\eqref{eq:expo_cv}: for all
$n \ge 0$,
$$\|\nu_n-\nu_\infty\|_{TV}\le \left( \exp \left(\frac{  \Theta
(\nu_{1},\nu_{0}) }{1-\rho} \rho^n \right) - 1 \right)$$
which yields~\eqref{eq:expo_cv2}.
\end{proof}
\begin{remark}
As usual in a Banach fixed point theorem, it is easy to obtain a
similar result assuming that $\mathcal T^r$ satisfies both
assumptions~\eqref{eq:AS1} and~\eqref{eq:AS2} for some positive integer~$r$.
\end{remark}

It remains to discuss how to get Assumptions~\eqref{eq:AS1}
and~\eqref{eq:AS2} in practice. A natural sufficient condition is the so-called two-sided
condition~\eqref{eq:2S} (notice that in~\eqref{eq:2S}, $s$ is actually with values in $(0,1]$ since for all $x\in
X$, $s(x) \le K(x,X)$).
\begin{proposition}\label{prop:twosided}
Assume that there exist a measurable function $s: X \to \R^*_+$, 
a constant $R >0$, and a probability
measure $\pi \in \mathcal M_1$ such that for all $x \in X$,
\begin{equation}\label{eq:2S}
s(x) \pi(dy) \le K(x,dy) \le R s(x) \pi(dy).
\end{equation}
Then, $\mathcal T$ satisfies~\eqref{eq:AS1} and~\eqref{eq:AS2} (with
$\Delta \le 2 \ln R$). In
particular, the results of Theorem~\ref{theo:CV} hold and one thus
obtains: for all initial condition $\lambda_0 \in \mathcal M_+$, for
all $n \ge 1$,
\begin{equation}\label{eq:expo_cv3}
\|\nu_n - \nu_\infty\|_{TV} \le (R+1) (\ln R) R^{R+1} \left(
  \frac{R-1}{R+1} \right)^{n-1}
\end{equation}
where $\nu_n=\frac{\mathcal T^n(\lambda_0)}{\mathcal
  T^n(\lambda_0)(X)}$.
\end{proposition}
\begin{proof}
We refer to~\cite[Theorem 3]{Birkhoff1957} for a similar reasoning.
Let us first check that~\eqref{eq:AS1} holds. By
Lemma~\ref{lem:AS1}, it is enough to check that $K(x,X) \neq 0$ for
all $x \in X$, but this is obvious since $K(x,X) \ge s(x) > 0$.

 Let us now check~\eqref{eq:AS2}, namely
$$\Delta = \sup_{\lambda,\nu \in \mathcal M_+} \Theta(\mathcal T (\lambda), \mathcal T
(\nu)) < \infty.$$
 Let $\lambda$ and $\nu$ be two measures in $\mathcal M_+$. One has
$$\mathcal T (\nu) \le R \left(\int_X s(x) \nu(dx) \right) \pi(dy) \le R \frac{
  \int_X s(x) \nu(dx)}{\int_X s(x) \lambda(dx)} \left(\int_X s(x)
\lambda(dx) \right) \pi(dy),$$
so
$$\mathcal T (\nu)  \le R \frac{
  \int_X s(x) \nu(dx)}{\int_X s(x) \lambda(dx)} \mathcal T
(\lambda),$$
which shows that 
$$C(\mathcal T (\lambda), \mathcal T (\nu)) \le R \frac{
  \int_X s(x) \nu(dx)}{\int_X s(x) \lambda(dx)} < \infty.$$
Likewise,
$$\mathcal T( \nu )\ge  \left(\int_X s(x) \nu(dx) \right) \pi(dy) \ge  \frac{
  \int_X s(x) \nu(dx)}{R \int_X s(x) \lambda(dx)} R \left(\int_X s(x)
\lambda(dx) \right) \pi(dy),$$
so
$$\mathcal T( \nu )\ge  \frac{
  \int_X s(x) \nu(dx)}{R\int_X s(x) \lambda(dx)} \mathcal T
(\lambda),$$
which shows that 
$$c(\mathcal T (\lambda), \mathcal T (\nu)) \ge  \frac{
  \int_X s(x) \nu(dx)}{R\int_X s(x) \lambda(dx)}  > 0.$$
 Therefore,
$$\Theta(\mathcal T (\lambda), \mathcal T (\nu))=\ln\left(
  \frac{C(\mathcal T (\lambda), \mathcal T (\nu))}{c(\mathcal T( \lambda),
    \mathcal T (\nu))}\right) \le \ln (R^2)$$
so that $\Delta \le \ln (R^2) < \infty$. Thus,
using~\eqref{eq:expo_cv2}, one
gets: for all $n \ge 1$,
\begin{align}
\|\nu_n-\nu_\infty\|_{TV} 
&\le \frac{\Theta
(\nu_{2},\nu_{1})}{1-\rho} \exp\left(\frac{\Theta
(\nu_{2},\nu_{1})}{1-\rho} \right) \rho^{n-1}, \nonumber\\
&\le \frac{\Delta}{1-\rho} \exp\left(\frac{\Delta}{1-\rho} \right) \rho^{n-1}, \nonumber
\end{align}
with $\rho=\tanh\left(\frac \Delta 4 \right) \le \tanh\left(\frac{\ln
    R}{2} \right) = \frac{R-1}{R+1}$, so that $\frac{\Delta}{1-\rho}
\le (\ln R) (R+1)$. This yields~\eqref{eq:expo_cv3}. 
\end{proof}


%
%
%

\section*{Acknowledgements}
The authors would like to thank Florian Angeletti who worked on the
early stages of this project and Julien Reygner for fruitful
discussions in particular about the proof of Proposition~\ref{Prop:ReactiveEntranceHarris} and connection with
potential theory.
This work benefitted from the support of
the European Research Council under the European Union's Seventh
Framework Programme (FP/2007-2013) / ERC Grant Agreement number
614492. Part of this project was carried out as TL was a visiting
professor at Imperial College of London (ICL), with a visiting professorship
grant from the Leverhulme
Trust. The Department of Mathematics at ICL and
the Leverhulme Trust are warmly thanked for their support.



\def\cprime{$'$}

\end{document}

%% file: f1.pstex_t
\begin{picture}(0,0)%
\includegraphics{f1.pstex}%
\end{picture}%
\setlength{\unitlength}{1973sp}%
\begingroup\makeatletter\ifx\SetFigFont\undefined%
\gdef\SetFigFont#1#2#3#4#5{%
  \reset@font\fontsize{#1}{#2pt}%
  \fontfamily{#3}\fontseries{#4}\fontshape{#5}%
  \selectfont}%
\fi\endgroup%
\begin{picture}(10510,3999)(2101,-5623)
\put(11326,-3961){\makebox(0,0)[lb]{\smash{{\SetFigFont{11}{13.2}{\rmdefault}{\mddefault}{\updefault}{\color[rgb]{0,0,0}$B$}%
}}}}
\put(9751,-4036){\makebox(0,0)[lb]{\smash{{\SetFigFont{11}{13.2}{\rmdefault}{\mddefault}{\updefault}{\color[rgb]{0,0,0}$\tau_1^B$}%
}}}}
\put(10576,-2686){\makebox(0,0)[lb]{\smash{{\SetFigFont{11}{13.2}{\rmdefault}{\mddefault}{\updefault}{\color[rgb]{0,0,0}$\tau_0^B$}%
}}}}
\put(2851,-2611){\makebox(0,0)[lb]{\smash{{\SetFigFont{11}{13.2}{\rmdefault}{\mddefault}{\updefault}{\color[rgb]{0,0,0}$\tau_0^A$}%
}}}}
\put(3151,-4036){\makebox(0,0)[lb]{\smash{{\SetFigFont{11}{13.2}{\rmdefault}{\mddefault}{\updefault}{\color[rgb]{0,0,0}$A$}%
}}}}
\put(4501,-3736){\makebox(0,0)[lb]{\smash{{\SetFigFont{11}{13.2}{\rmdefault}{\mddefault}{\updefault}{\color[rgb]{0,0,0}$\tau_1^A$}%
}}}}
\put(3826,-5311){\makebox(0,0)[lb]{\smash{{\SetFigFont{11}{13.2}{\rmdefault}{\mddefault}{\updefault}{\color[rgb]{0,0,0}$\tau_2^A$}%
}}}}
\end{picture}%

%% file: f2.pstex_t
\begin{picture}(0,0)%
\includegraphics{f2.pstex}%
\end{picture}%
\setlength{\unitlength}{1973sp}%
\begingroup\makeatletter\ifx\SetFigFont\undefined%
\gdef\SetFigFont#1#2#3#4#5{%
  \reset@font\fontsize{#1}{#2pt}%
  \fontfamily{#3}\fontseries{#4}\fontshape{#5}%
  \selectfont}%
\fi\endgroup%
\begin{picture}(11500,3724)(1111,-5888)
\put(3751,-5461){\makebox(0,0)[lb]{\smash{{\SetFigFont{11}{13.2}{\rmdefault}{\mddefault}{\updefault}{\color[rgb]{0,0,0}$Y_1$}%
}}}}
\put(3151,-4036){\makebox(0,0)[lb]{\smash{{\SetFigFont{11}{13.2}{\rmdefault}{\mddefault}{\updefault}{\color[rgb]{0,0,0}$A$}%
}}}}
\put(10351,-2761){\makebox(0,0)[lb]{\smash{{\SetFigFont{11}{13.2}{\rmdefault}{\mddefault}{\updefault}{\color[rgb]{0,0,0}$Y_{T_{\cal B}}$}%
}}}}
\put(1276,-5311){\makebox(0,0)[lb]{\smash{{\SetFigFont{11}{13.2}{\rmdefault}{\mddefault}{\updefault}{\color[rgb]{0,0,0}$\Upsigma$}%
}}}}
\put(9676,-4711){\makebox(0,0)[lb]{\smash{{\SetFigFont{11}{13.2}{\rmdefault}{\mddefault}{\updefault}{\color[rgb]{0,0,0}${\cal B}$}%
}}}}
\put(1501,-4261){\makebox(0,0)[lb]{\smash{{\SetFigFont{11}{13.2}{\rmdefault}{\mddefault}{\updefault}{\color[rgb]{0,0,0}${\cal A}$}%
}}}}
\put(1126,-3811){\makebox(0,0)[lb]{\smash{{\SetFigFont{11}{13.2}{\rmdefault}{\mddefault}{\updefault}{\color[rgb]{0,0,0}$S$}%
}}}}
\put(2701,-4936){\makebox(0,0)[lb]{\smash{{\SetFigFont{11}{13.2}{\rmdefault}{\mddefault}{\updefault}{\color[rgb]{0,0,0}$Y_0$}%
}}}}
\put(4126,-4111){\makebox(0,0)[lb]{\smash{{\SetFigFont{11}{13.2}{\rmdefault}{\mddefault}{\updefault}{\color[rgb]{0,0,0}$Y_2$}%
}}}}
\put(3976,-3586){\makebox(0,0)[lb]{\smash{{\SetFigFont{11}{13.2}{\rmdefault}{\mddefault}{\updefault}{\color[rgb]{0,0,0}$Y_3$}%
}}}}
\put(11326,-3961){\makebox(0,0)[lb]{\smash{{\SetFigFont{11}{13.2}{\rmdefault}{\mddefault}{\updefault}{\color[rgb]{0,0,0}$B$}%
}}}}
\end{picture}%